%% file: main.tex
\documentclass[12pt, reqno]{amsart}
\usepackage[utf8]{inputenc}
\usepackage[letterpaper,top=1.5in,bottom=1.5in,left=1in,right=1in]{geometry}
\usepackage[english]{babel}
\usepackage[T1]{fontenc}
\usepackage{graphicx,caption,subcaption}
\usepackage{amsmath,amsthm,amssymb}
\usepackage{mathabx,epsfig}
\usepackage{mathrsfs} 
\usepackage{comment}
\usepackage[colorlinks=true, allcolors=blue]{hyperref}
\usepackage{enumitem}

\usepackage[textsize=small]{todonotes}

\usepackage{array}

\usepackage[scr=boondox]{mathalpha} 

\usepackage{tikz-cd,tikz}
\usetikzlibrary{matrix}

\def\from{\colon}
\def\Cbb{\mathbb{C}}

\def\Hbb{\mathbb{H}}
\def\Rbb{\mathbb{R}}

\def\Ccal{\mathcal{C}}

\def\Fcal{\mathcal{F}}
\def\Gcal{\mathcal{G}}
\def\Hcal{\mathcal{H}}
\def\Mcal{\mathcal{M}}
\def\Ncal{\mathcal{N}}

\def\Tcal{\mathcal{T}}
\def\Vcal{\mathcal{V}}

\def\MF{\mathcal{MF}}
\def\PMF{\mathcal{PMF}}

\newcommand{\pair}[2]{\left< #1, #2 \right>}

\DeclareMathOperator{\EL}{EL}
\DeclareMathOperator{\Vertex}{Vertex}
\DeclareMathOperator{\Edge}{Edge}
\DeclareMathOperator{\Hor}{Hor}

\DeclareMathOperator{\Proj}{\bf{P}}
\DeclareMathOperator{\Pull}{Pull}
\DeclareMathOperator{\Push}{Push}
\DeclareMathOperator{\interior}{int}

\DeclareMathOperator{\supp}{supp}
\DeclareMathOperator{\link}{link}
\DeclareMathOperator{\pre}{pre}
\DeclareMathOperator{\Area}{Area}
\DeclareMathOperator{\Teich}{Teich}
\DeclareMathOperator{\MCG}{MCG}

\newtheorem{thm}{Theorem}[section]

\newtheorem{lem}[thm]{Lemma}
\newtheorem{prop}[thm]{Proposition}

\newtheorem{rem}[thm]{Remark}

\newtheorem{theorem}{Theorem}
\newtheorem{question}{Question}

\newtheorem{conjecture}[theorem]{Conjecture}

\theoremstyle{definition}
\newtheorem{defn}[thm]{Definition}

\theoremstyle{remark}

\title[Puncture-forgetting maps for measured foliations]{Puncture-forgetting maps for measured foliations and applications in Teichm\"uller space and complex dynamics}
\begin{document}

\author{Jeremy Kahn}
\address{Department of Mathematics, Brown University, Providence, RI 02912}
\email{jeremy\_kahn@brown.edu}

\author{Insung Park}
\address{Institute of Mathematical Sciences, Stony Brook University, Stony Brook, NY 11794}
\email{insung.park@stonybrook.edu}

\begin{abstract}
We introduce puncture-forgetting maps for measured foliations and investigate their relations with mapping class groups, Teichm\"uller spaces, and extremal length. To this end, we develop the notions of cube complexes of pre-homotopic multicurves and tree coordinate systems on CAT(0) cube complexes. As an application to the dynamics of post-critically finite rational maps on the Riemann sphere, we obtain a partial result toward the finite curve attractor conjecture posed by Kevin Pilgrim. We also apply these methods to uncover a relation between horospheres and geodesic flows in the universal curve over Teichm\"uller space.
\end{abstract}

\maketitle
\tableofcontents

\section{Introduction}
The idea of puncture-forgetting in the study of topological objects and geometric structures of surfaces has been developed in many aspects, e.g., the Birman exact sequence of mapping class groups and the puncture-forgetting maps between Teichm\"uller spaces. However, puncture-forgetting has received less attention in the context of measured foliations, another classical object in the study of surfaces. One reason might be the lack of continuity; puncture-forgetting maps between Teichm\"uller spaces do not continuously extend to the Thurston boundaries, which are the spaces of projective measured foliations. In this article, we introduce a natural (and upper semi-continuous) definition of puncture-forgetting maps $\Phi$ for measured foliations. We also study cube complexes of pre-homotopic multicurves, which serve as a foundation for constructing $\Phi$. Moreover, we investigate the properties of $\Phi$ in relation to Teichm\"uller geometry and explore its applications in the dynamics of rational maps on the Riemann sphere.

\subsection*{Puncture-forgetting maps for measured foliations}
We let $S_{g, n}$ denote the ``standard model'' of a surface with genus $g$ and $n$ punctures, and let $i\from S_{g,n+1} \to S_{g, n}$ be the standard inclusion. Throughout, we assume $\chi(S_{g,n})<0$ and denote by $x$ the only point in $S_{g,n}$ not in the image of $i$. Let $C(S_{g,n})$ be the set of homotopy classes of essential simple closed curves of $S_{g,n}$, and let $\MF(S_{g,n})$ be the set of equivalence classes of measured foliations on $S_{g,n}$. We have $i_* \from C(S_{g, n+1}) \to C(S_{g, n})\cup \{0\}$, where we denote by $0$ the homotopy classes of $n$ peripheral loops. Denote by $\pi\from \Tcal_{g, n+1} \to \Tcal_{g, n}$ the puncture forgetting maps between Teichm\"uller spaces.

The set $\MF(S_{g,n+1})$ can be embedded in the space of functionals $(C(S_{g,n+1}))^*$ by mapping each $\Fcal \in \MF(S_{g,n+1})$ to 
\[
    \pair{\Fcal}{-}\from [\alpha] \mapsto \pair{\Fcal}{[\alpha]}\quad \forall [\alpha] \in C(S_{g,n+1}),
\]
where $\pair{-}{-}$ denotes the geometric intersection number.
We can define a puncture-forgetting map for functionals $\varphi\from (C(S_{g,n+1}))^* \to (C(S_{g,n}))^*$ by
\[
    \varphi(f)([\alpha]):=\inf_{i_*[\gamma]=[\alpha]} f([\gamma]).
\]
Then, for every $\Fcal\in \MF(S_{g,n+1})$, we obtain a functional $\varphi(\Fcal)\from C(S_{g,n})\to \Rbb$ defined by
\begin{equation}\label{eqn:PForget Functional}
    \varphi(\Fcal)([\alpha])= \varphi([\alpha]) \mapsto \pair{\Fcal}{[\alpha]})
=    \inf_{i_*[\gamma]=[\alpha]} \pair{\Fcal}{[\gamma]}.
\end{equation}
In many cases, we obtain $\varphi(\Fcal)=0$, where $0$ is the constant functional that maps its domain to zero. We would like to realize the functional $\varphi(\Fcal)$ as the intersection number $\pair{\Phi(\Fcal)}{-}$ for some measured foliation $\Phi(\Fcal) \in \MF(S_{g,n})$.  Hence, we introduce the null measured foliation, denoted by $0$, defined by the property that $\pair{0}{[\alpha]} = 0$ for all $[\alpha] \in C(S_{g,n})$. We then have:

\begin{theorem}\label{theorem:MF}
There is a map $\Phi\from \MF(S_{g,n+1}) \to \MF(S_{g,n})$ such that for any $[\zeta] \in C(S_{g,n})$,
    \[
        \pair{\Phi(\Fcal)}{[\zeta]}=\varphi(\Fcal)([\zeta]).
    \]
    Moreover, the map $\Phi$ satisfies the following properties.
\begin{enumerate}
    \item For each $[\zeta] \in C(S_{g, n})$, the function
    $\Fcal \mapsto \pair{\Phi(\Fcal)}{[\zeta]}$ is upper semi-continuous. 
    \item The vanishing locus $\Phi^{-1}(0)$ is a countable intersection of open dense subsets. In particular, $\Phi(\Fcal)=0$ if $\Fcal$ is invariant under a point-pushing pseudo-Anosov mapping class in $\Push(\pi_1(S_{g,n},x))$.
    \item (Theorem \ref{thm:InequalityEL}) For every $X\in \Tcal_{g,n+1}$ and $\Fcal \in \MF(S_{g,n+1})$, we have
    \[
        \EL(\Phi(\Fcal);\pi(X)) = \inf_{X'\in \pi^{-1}(\pi(X))}\EL(\Fcal;X'),
    \]
    where $\EL(\Fcal;X)$ denotes the extremal length of $\Fcal$ on $X$. The infimum is attained as the minimum if and only if $\Fcal$ is $x$-non-generic of non-annular type.
    \item (Theorem \ref{thm:LowBddTeichDist}) Suppose $\Phi(\Fcal)\neq 0$. For any $X\in \Tcal_{g,n+1}$, denote by $(X_t)_{t\ge0}$ the Teichm\"{u}ller geodesic starting from $X_0=X$ generated by $\Fcal$ such that $d_{\Teich}(X_0,X_t)=t$, where $d_{\Teich}$ denotes the Teichm\"uller metric. Then, for any $t\ge0$, we have
    \[
        t - \frac{1}{2} \log \frac{\EL(\Fcal;X_0)}{\EL(\Phi(\Fcal);\pi(X_0))} ~\le~ d_{\Teich}(\pi(X_0),\pi(X_t))) ~\le~ t.
    \]
    \item (Theorem \ref{thm:Shards of memory}) The set of $x$-generic minimal measured foliations $\Fcal$ with $\Phi(\Fcal)\neq 0$ is dense in $\MF(S_{g,n+1})$.
\end{enumerate}
Moreover, the statements (1)--(4) generalize to the setting of forgetting $m>1$ punctures from $S_{g,n+m}$ to $S_{g,n}$(Proposition \ref{prop:m Punctures}).
\end{theorem}

A Teichm\"uller geodesic $(X_t)$ is {\it generated by $\Fcal$} if it is vertically stretched and horizontally contracted by $e^t$ with respect to the quadratic differential whose horizontal foliation is equivalent to $\Fcal$. 

Suppose $\Fcal$ is a measured foliation on $S_{g,n+1}$.  
A point $y \in S_{g,n+1}$ is called a {\it singular point} of $\Fcal$ if either $y$ is one of the $n+1$ marked points or $\deg_\Fcal(y) \neq 2$. The $1$-prong singularity, i.e., $\deg_\Fcal(y)=1$, is only allowed when $y$ is a marked point.
A {\it leaf} of $\Fcal$ is a maximal immersed $1$--dimensional manifold satisfying: (1) its endpoints, if any, are singular points of $\Fcal$, and (2) its interior contains no singular points. A {\it saddle connection} is a compact leaf, joining two singular points. A measured foliation $\Fcal$ is {\it minimal} if it has positive intersection number with every homotopy class of simple closed curves, i.e., $\pair{\Fcal}{[\gamma]} > 0$ for all $[\gamma] \in C(S_{g,n+1})$.

Recall $x$ denotes the marked point of $S_{g,n+1}$ we forget. We say that $\Fcal \in \MF(S_{g,n+1})$ is {\it $x$-generic} if $\deg_\Fcal(x)=1$ and the leaf containing $x$ as an endpoint is an immersion of $[0,\infty)$ (Definition \ref{defn:genericMF}).

When $\Fcal$ is $x$-non-generic, we can construct $\Phi(\Fcal)$ in a natural way. See Figure \ref{fig:Forget_nondeg} and Proposition \ref{prop:PForget NonGeneric}.

When $\Fcal$ is $x$-generic, we prove the realization of $\varphi(\Fcal)$ as an intersection number with a measured foliation $\Phi(\Fcal)$ in Section \ref{subsec:Proof of Thm1 I}. We first establish this for minimal foliations. Subsequently, we discuss how to reduce the case of non-minimal foliations to the minimal case. The decomposition of measured foliations into annulus and minimal components is discussed in Section \ref{subsec:NonMinimal}. The statements (1) and (2) of Theorem \ref{theorem:MF} are also proven in Section \ref{subsec:Proof of Thm1 I}.




\subsection*{Horospherical limit set}
For a non-elementary subgroup $G$ of $\MCG^+(S_{g,n})$, its {\it (small) horospherical limit set} $\Lambda^{\rm hor}_G\subset \PMF(S_{g,n})$ is the collection of $\Proj(\Fcal)$, where $\Proj\from \MF(S_{g,n})\setminus\{0\}\to \PMF(S_{g,n})$, such that for all $r>0$ and $X\in \Tcal_{g,n}$ (with $B_h(\Fcal,r)$ defined in \eqref{eqn:horo_ball_sphere}),
\[
    G\cdot X\,\cap\,B_h(\Fcal,r)\neq \emptyset, 
\]
or equivalently,
\begin{equation}\label{eqn:Inf_EL_Group_Action}
    \inf_{g\in G}\EL(\Fcal,g\cdot X)=0.
\end{equation}
\begin{theorem}\label{theorem:Horo_Limit_Set}
    Let $G:=\Push(\pi_1(S_{g,n},x))\le \MCG^+(S_{g,n+1})$. Then,
    \[
        \Lambda^{\rm hor}_G
        =
        \left\{
        \Proj(\Fcal): \Phi(\Fcal)=0
        \right\}.
    \]
    Moreover, $\Lambda^{\rm hor}_G$ is a countable intersection of open dense subsets of $\PMF(S_{g,n+1})$ such that its complement contains a dense subset of projective minimal measured foliations.
\end{theorem}

\subsection*{Cube complexes of multicurves}
To show the existence of a measured foliation $\Phi(\Fcal)\in \MF(S_{g,n})$ in Theorem \ref{theorem:MF}, we investigate the complex of pre-homotopic multicurves $X_Z$, which is defined as follows.

Let $Z$ be a multicurve on $S_{g,n}$. (Here by ``multicurve'' we mean a finite set of simple closed curves, which are allowed to intersect.) Define the vertex set of $X_Z$ as the set of homotopy classes of multicurves $[\Gamma]$ on $S_{g,n+1}$ such that every component $\gamma$ of $\Gamma$ is essential in $S_{g,n}$ (pre-essential), no two components of $\Gamma$ are homotopic to each other in $S_{g,n}$ (pre-non-parallel), $\pair{[\Gamma]}{[\Gamma]}_{S_{g,n+1}}=\pair{[Z]}{[Z]}_{S_{g,n}}$ (pre-minimal), and $i_*[\Gamma]=[Z]$. Edges are assigned to simple pulls (Figure \ref{fig:SimplePull}), which defines $X_Z^{(1)}$. Then, for $d>1$, we inductively attach a $d$-cube to every $(d-1)$-dimensional subcomplex that is isomorphic to the $(d-1)$-skeleton of a $d$-cube. Below is a short list of our findings about $X_Z$, which we call the complex of pre-homotopic multicurves.
\begin{theorem}\label{theorem:CubeCplx} For the complex of pre-homotopic multicurves $X_Z$, the following hold.
\begin{enumerate}
    \item (Theorem \ref{thm:CubeCplxIso}) $X_Z$ is isomorphic as cube complexes to the dual cube complex $\mathscr{X}_Z$ of the wall space $(\widetilde{S_{g,n}},p^{-1}(Z))$, where $p\from \widetilde{S_{g,n}} \to S_{g,n}$ is the universal covering.
    \item (Theorem \ref{thm:HypCubeCplx}) $\mathscr{X}_Z$ is Gromov-hyperbolic, and thus so is $X_Z$.
\end{enumerate}
For every $x$-generic minimal foliation $\Fcal$, we orient edges of $X_Z$ toward the direction that the intersection numbers with $\Fcal$ decrease. That is, orient an edge from $[\Gamma_1]$ to $[\Gamma_2]$ if $\pair{\Fcal}{[\Gamma_1]}>\pair{\Fcal}{[\Gamma_2]}$. With this orientation, we have the following.
\begin{enumerate}[resume]
    \item (Theorems \ref{thm:Roller-to-Gromov}) Infinite directed edge-paths are asymptotic to each other.
\end{enumerate}
Let $([\Gamma(l)])_{l\ge0}$ be the sequence of vertices along an infinite directed edge-path. We can write $\Gamma(l)=\{\gamma_1(l),\gamma_2(l),\dots,\gamma_k(l)\}$ such that $\gamma_i(l)$ and $\gamma_i(l')$ are pre-homotopic to each other for any $l,l'\ge0$. Let $i_*[\gamma_i(l)]=[\zeta_i]\in [Z]$.
\begin{enumerate}[resume]
    \item (Theorems \ref{thm:CurveReduction}, \ref{thm:MulticurveReduction}) For each $1\le i\le k$, we have $\pair{\Fcal}{[\gamma_i(l)]}\to \varphi(\Fcal)([\zeta_i])$ as $l\to \infty$.
    \item (Theorems \ref{thm:CurveReduction}, \ref{thm:MulticurveReduction}) For each $1\le i\le k$, we have $[\gamma_i(l)]\to \underline{\Fcal}\in \partial_\infty CC(S_{g,n+1})=\mathcal{EL}(S_{g,n+1})$ as $l\to \infty$, where $\underline{\Fcal}$ is the underlying minimal foliation of $\Fcal$ (with the transverse measure forgotten), $CC(S_{g,n+1})$ is the cube complex of $S_{g,n}$, and $\mathcal{EL}(S_{g,n+1})$ is the space of ending laminations of $S_{g,n+1}$.
\end{enumerate}
\end{theorem}

In Section \ref{subsec:TreeCoordinateSystem}, we introduce the notion of {\it tree coordinate systems} for cube complexes, which is useful in proving properties of $X_Z$ with purely combinatorial methods related to cube complexes.

In our proof of Theorem \ref{theorem:MF}, we use a multicurve $\Gamma$ that fills $S_{g,n+1}$ and determines measured foliations in $S_{g,n}$. We then decrease the intersection number with $\Fcal$ by changing $\Gamma$ along a directed edge-path in $X_\Gamma$ (Theorem \ref{theorem:CubeCplx}-(4)). Since $\pi_1(S_{g,n})$ acts cocompactly on $X_\Gamma$, we can approximate this edge-path by the orbit of a sequence of point-pushing mapping classes, $(\Push([\psi_k]))_{k\in \mathbb{N}}$. It follows that the limit $\Gcal:=\lim_{k\to \infty}[\psi_k]^{-1}\cdot \Fcal$ exists such that $\Gcal$ is $x$-non-generic (Proposition \ref{prop:LimitMFfromIntReduction}). Hence, $\Phi(\Gcal)$ is defined in an obvious way (Proposition \ref{prop:PForget NonGeneric}). We then define $\Phi(\Fcal):=\Phi(\Gcal)$ and show it satisfies the desired properties. 

\subsection*{Horospheres, puncture forgetting maps, and the dynamics on the universal curves over Teichm\"uller spaces}

Fix $r>0$. For every $\Fcal\in \MF(S_{g,n})$, its {\it open} (resp.\@ {\it closed}\,) {\it $r$-horoball} $B_h(\Fcal,r)$ (resp.\@ $\overline{B_h(\Fcal,r)}$) and {\it $r$-horosphere $S_h(\Fcal,r)$} in $\Tcal_{g,n}$ are defined by
\begin{align}
    \begin{split}\label{eqn:horo_ball_sphere}
        B_h(\Fcal,r)&=\left\{X\in \Tcal_{g,n}:\EL(\Fcal;X)< r\right\}\\
        \overline{B_h(\Fcal,r)}&=\left\{X\in \Tcal_{g,n}:\EL(\Fcal;X)\le r\right\}\\
        S_h(\Fcal,r)&=\left\{X\in \Tcal_{g,n}:\EL(\Fcal;X)= r\right\}.
    \end{split}
\end{align}

Recall that $\pi \colon \Tcal_{g,n+1} \to \Tcal_{g,n}$ denotes the puncture-forgetting map between Teichm\"uller spaces. Theorem \ref{theorem:MF} can be used to understand the behavior of $\pi$ near the (Thurston) boundary. For example, we obtain the following classification of $\pi$-images of horospheres as a corollary of Theorem \ref{theorem:MF}.
\begin{theorem}[=Theorem \ref{thm:Im_horospheres_}]
       For all $r>0$ and non-zero $\Fcal\in \MF(S_{g,n+1})$, the following hold.
       \begin{enumerate}
            \item If $\Phi(\Fcal)= 0$, then
                \[
                    \pi\left(S_h(\Fcal,r)\right)=\Tcal_{g,n}.
                \]
            \item If $\Phi(\Fcal)\neq 0$ and $\Fcal$ is $x$-generic or $x$-non-generic of annular type, then
                \[
                    \pi\left(S_h(\Fcal,r)\right)=B_h\left(\Phi(\Fcal),r\right).
                \]
            \item If $\Phi(\Fcal)\neq0$ and $\Fcal$ is $x$-non-generic of non-annular type, then 
                \[
                    \pi\left(S_h(\Fcal,r)\right)=\overline{B_h\left(\Phi(\Fcal),r\right)}.
                \]
       \end{enumerate}
\end{theorem}
See Section \ref{subsec:Non-generic} for the definitions of the annular and non-annular types.

Denote by $\Ccal_{g,n}$ the {\it universal curve over $\Tcal_{g,n}$}, which is the tautological Riemann surface bundle over $\Tcal_{g,n}$ whose fibers are Riemann surfaces with genus $g$ and $n$ punctures. The puncture-forgetting map $\pi\from \Tcal_{g,n+1}\to \Tcal_{g,n}$ factors as the composition of two projections, $\pi^1\from \Tcal_{g,n+1}\to \Ccal_{g,n}$ and $\pi^2\from \Ccal_{g,n}\to \Tcal_{g,n}$. That is,
\[
    \begin{tikzcd}
		\Tcal_{g,n+1} \arrow{r}{\pi^1}
		& \mathcal{C}_{g,n} \arrow{r}{\pi^2}
        & \Tcal_{g,n}
	\end{tikzcd}.
\]

The universal curve $\Ccal_{g,n}$ can also be obtained as $\Tcal_{g,n+1}/\pi_1(S_{g,n},x)$, where $\pi_1(S_{g,n},x)$ acts as point-pushing mapping classes. Since $\pi_1(S_{g,n})\triangleleft\MCG^+(S_{g,n+1})$, the quotient group $\MCG^+(S_{g,n})=\MCG^+(S_{g,n+1})/\pi_1(S_{g,n})$ acts on $\Ccal_{g,n}$ such that $\Ccal_{g,n}/\MCG^+(S_{g,n})$ coincides with the moduli space $\Mcal_{g,n+1}=\Tcal_{g,n+1}/\MCG^+(S_{g,n+1})$.

It is natural to ask about the classification of Teichm\"uller geodesics and horospheres in $\Ccal_{g,n}$. In this context, inspired by Eberlein’s work \cite[Theorem 5.5]{Eberlein_GeodFLowII}, we consider Question~\ref{question:Dyn_Ccal}.

Consider a Teichm\"uller geodesic $(X_t)$ generated by $\Fcal \in \MF(S_{g,n+1})$. We say that $\pi^1(X_t)$ is {\it quasi-minimizing} in $\Ccal_{g,n}$ if there exists a constant $C>0$ such that
\[
    d_{\mathrm{Teich}}(\pi^1(X_0), \pi^1(X_t)) \ge t - C
    \quad \text{for all } t \ge 0.
\]
We say that $(X_t)$ is {\it pre-quasi-minimizing} if there exists a constant $C>0$ such that
\[
    d_{\mathrm{Teich}}(\pi(X_0), \pi(X_t)) \ge t - C
    \quad \text{for all } t \ge 0.
\]

\begin{question}\label{question:Dyn_Ccal}
In the above setting, what are the relations among the following properties?
\begin{enumerate}
    \item $\Phi(\mathcal{F}) \neq 0$.
    \item $(X_t)$ is pre-quasi-minimizing.
    \item The projected Teichm\"uller geodesic $\pi^1(X_t)$ in $\Ccal_{g,n}$ is quasi-minimizing.
    \item $\pi^1\!\big(S_h(\mathcal{F}, r)\big)$ is not dense in $\Ccal_{g,n}$.
\end{enumerate}
\end{question}
(1)$\Rightarrow$(2) follows from Theorem~\ref{theorem:MF}-(4), and (2)$\Rightarrow$(3) follows from the fact that Teichm\"uller metrics contract under holomorphic maps. In Section \ref{subsec:QuasiMini_Horosph}, we discuss Question \ref{question:Dyn_Ccal} for $x$-non-generic measured foliations and projectively invariant measured foliations of point-pushing pseudo-Anosov mapping classes.

Question \ref{question:Dyn_Ccal} can also be viewed as an inhomogeneous analogue of recent works on geodesic and horocyclic flows on geometrically infinite hyperbolic surfaces \cite{FLM_Horo_Zcover,DFLM_New_HoroDyn}. See also \cite{SSWY_Horosphere_Inv_subvar} for the closures of horospheres in the space of translation surfaces. 

\subsection*{Pullback maps of post-critically finite rational maps}
Consider a branched self-covering $f\from S^2\to S^2$ of the $2$-sphere. We call every branch point a {\it critical point}. Suppose that $f$ is {\it post-critically finite}, i.e., every critical point is periodic or pre-periodic. Denote by ${\rm Crit}(f)$ the set of critical points. Define the post-critical set $P_f$ by $P_f:=\bigcup_{c\in {\rm Crit(f)},n\ge1}f^{n}(c)$, which is finite. Then $f(P_f)\subset P_f\subset f^{-1}(P_f)$.

For a post-critically finite branched self-covering $f\colon (S^2,P_f)\to (S^2,P_f)$, W.\ Thurston constructed a holomorphic map $\sigma_f\colon \Tcal(S^2,P_f)\to \Tcal(S^2,P_f)$, called the (Thurston) {\it pullback map} of $f$ \cite{DH_ThurstonChar}. The existence of a (unique) fixed point of $\sigma_f$ in characterizes when $f$ is conjugate, up to isotopy, to a post-critically finite rational map.

It has been an important theme to study the relationship between the dynamical behavior of $\sigma_f$ on $\Tcal(S^2,P_f)$ and the dynamics of $f$ on $\widehat{\mathbb{C}}$.

There exists a continuous extension $\overline{\sigma_f}$ of $\sigma_f$ to the augmented Teichm\"uller space \cite{Selinger_ThuMapAugTeich}. If $\sigma_f$ has no fixed point in $\mathcal{T}(S^2,P_f)$, then $\overline{\sigma_f}$ admits a fixed point on the boundary of the augmented Teichm\"uller space. The multicurve defining the stratum containing such a fixed point is called the {\it canonical Thurston obstruction}; see also \cite{Pil_Canonical_Obs, Sel_TopChac_CanonicalThu}.

There exists a finite-time algorithm deciding whether ${\rm Fix}(\sigma_f)\neq \emptyset$ for a given $f\colon (S^2,P_f)\to (S^2,P_f)$ \cite{BBY_ThuEq_Decidable}. See also \cite{FPP_Rationality_NET} for a quantitative study of an algorithm for nearly Euclidean Thurston maps. The first-named author and D.\ Thurston are making progress on the quantitative decidability problem in the case $|P_f|=4$. All of these works analyze the dynamical behavior of the pullback map $\sigma_f$.

Even when $f$ is already a post-critically finite rational map, the dynamics of $\sigma_f$ and $f$ are still closely related; see \cite{KochPilgrimSelinger_Pullback, ParryPilgrim_Contract_Hurwitz}. The finite curve attractor conjecture (Conjecture \ref{conjecture:FCA}) concerns the dynamics of simple closed curves in $(\widehat{\mathbb{C}},P_f)$ under $f$. Since simple multicurves correspond to boundary strata of the augmented Teichm\"uller space, this conjecture can be interpreted in terms of the boundary behavior of $\sigma_f$. Further details are given below.

\subsection*{Finite curve attractor conjecture}
One motivation for studying puncture-forgetting maps for measured foliations is Pilgrim's finite curve attractor conjecture; see \cite{Pil_PullbackCurve_survey} for a survey. Suppose that $f\from \hat{\Cbb}\to \hat{\Cbb}$ is a post-critically finite rational map and consider the dynamics defined by its iterations $(f^{\circ n})_{n\ge 0}$. 


\begin{conjecture}[Finite curve attractor]\label{conjecture:FCA}
Suppose that $f$ is a post-critically finite rational map that is not a flexible Latt\`{e}s map. Then there exists a finite set $X$ of homotopy classes of essential simple closed curves on $(\hat{\Cbb},P_f)$ such that for any simple closed curve $\gamma$ on $(\hat{\Cbb},P_f)$, every essential component of $f^{-n}(\gamma)$ is contained in $X$ for any sufficiently large $n>0$.
\end{conjecture}

Here is a list of post-critically finite rational maps for which the finite curve attractor is proved.
\begin{enumerate}
    \item The tetrahedral map $f(z)=\frac{3z^2}{2z^3+1}$ \cite{Lodge_BdryThuPullBack}.
    \item Quadratic rational maps $f$ with $|P_f|=4$ \cite{KL_QuadPf4}.
    \item Critically fixed rational maps \cite{Hlu_CritFixed}.
    \item Rational maps $f$ whose the virtual endomorphisms of $\MCG^+(S^2,P_f)$ induced from the moduli correspondences are contracting \cite{KPS_PullbackInv}. By \cite[Theorem 7.2]{Nek_CombModel}, hyperbolic polynomials satisfy this property.
    \item Polynomials \cite{BLMW_LiftingTrees}.
    \item Rational maps $f$ with $|P_f|=4$ that are obtained from the $2\times 2$-Latt\`{e}s map by flappings \cite{BHI_ElimThuObs}.
    \item Rational maps $f$ with $|P_f|=4$ for which the augmented Teichm\"uller space of $(\hat{\Cbb},P_f)$ has a cover satisfying certain properties regarding the pullback map $\sigma_f$ \cite{BHL_Pullback_InvCover}.
    \item Post-critically finite rational maps with $|P_f|=4$ \cite{BDP_CorrRieSurf}.
\end{enumerate}

If there is no finite attractor $X$, then either (a) there exists an infinite sequence of distinct homotopy classes of simple closed curves $([\gamma_n])_{n\ge0}$ such that $\gamma_{n+1}$ is homotopic to a component of $f^{-1}(\gamma_n)$ or (b) there exists an infinite sequence of completely $f^{p_n}$-invariant simple multicurves $([\Gamma_n])_{n\ge0}$. In both cases, we obtain an infinite sequence in the space of projective measured foliations $\PMF(\hat{\Cbb},P_f)$. In Section \ref{sec:CplxDyn}, we define sequences of measures $(\mu_n)$ on $\PMF(\hat{\Cbb},P_f)$ by Equations \eqref{eqn:ProbMsrWandering} and \eqref{eqn:ProbMsrIntPer}, by using $\gamma_n$ and $\Gamma_n$ described above. Then, we show that any weak$^*$-limit $\mu$ of $(\mu_n)$ has measure zero on the set of geodesically recurrent measured foliations, which has full measure with respect to the Thurston measure. A measured foliation $\Fcal$ is {\it geodesically recurrent} if the Teichm\"uller geodesic in the moduli space $\Mcal(\hat{\Cbb},P_f)$ generated by $\Fcal$ at any initial point returns to a compact subset infinitely often. 

\begin{theorem}[Theorem \ref{thm:LimitMsr nonRecurrent}]\label{theorem:CplxDyn}
    Let $f$ be a post-critically finite rational map with $|P_f|>3$. Suppose that $f$ does not have global finite curve attractor. Let $(\mu_n)$ be a sequence of probability measures on $\PMF(\hat{\Cbb},P_f)$ constructed in Equation \eqref{eqn:ProbMsrWandering} or \eqref{eqn:ProbMsrIntPer}, obstructing the existence of the finite curve attractor. Then, for every weak$^*$-limit $\mu$ of $(\mu_n)_{n\ge0}$, we have $(\sigma_f)_*\mu=\mu$ and 
    \[
        \mu(\{\Proj(\Fcal)\in \PMF(\hat{\Cbb},P_f) ~|~ \Fcal~{\rm is~geodesically~recurrent}\})=0.
    \]
\end{theorem}


\subsection*{Notations} We summarize some notations we use in this article.

$S_{g,n}$: a closed orientable surface with genus $g$ and $n$ punctures. We assume $\chi(S_{g,n})<0$.

$C(S_{g,n})$: the set of homotopy classes of essential simple closed curves of $S_{g,n}$.

$CC(S_{g,n})$: the curve complex of $S_{g,n}$.

$\mathcal{EL}(S_{g,n})$: the set of ending laminations.

$\MF(S_{g,n})$: the set of measured foliations of $S_{g,n}$, including the zero element $0$.

$\PMF(S_{g,n})$: the set of projective measured foliations of $S_{g,n}$.

$\Proj\from \MF(S_{g,n})\setminus\{0\} \to \PMF(S_{g,n})$: the projectivization map.

$(\hat{\Cbb},P_f)$: the Riemann sphere with punctures at the points in the post-critical set $P_f$.

$\Tcal_{g,n}$ and $\mathcal{M}_{g,n}$: the Teichm\"ller and the moduli spaces of $S_{g,n}$.

$\Ccal_{g,n}$: the universal curve over $\Tcal_{g,n}$.

$\Tcal(\hat{\Cbb},P_f)$ and $\mathcal{M}(\hat{\Cbb},P_f)$: the Teichm\"uller and the moduli spaces of $(\hat{\Cbb},P_f)$.

$\MCG^+(S_{g,n})$: the orientation preserving mapping class group of $S_{g,n}$.

$\Push\from \pi_1(S_{g,n},x)\hookrightarrow \MCG^+(S_{g,n+1})$: the $\Push$ map in the Birman exact sequence
\[
    0\to \pi_1(S_{g,n},x) \hookrightarrow \MCG^+(S_{g,n+1})\to \MCG^+(S_{g,n})\to 0.
\]

\subsection*{Acknowledgments} The authors thank Dylan Thurston, Kevin Pilgrim, Inhyeok Choi, and Dongryul Kim for helpful discussions. I.P. is grateful to the Institute of Mathematical Sciences at Stony Brook University for travel support during this project. J.K. would like to acknowledge the support of a Simons Fellows in Mathematics grants SFI-MPS-SFM-00011896 and 919571 and Simons Travel Grant for Mathematicians MPS-TSM-00002848.

\section{CAT(0) cube complexes}\label{sec:CubeCplx}

Assume that $S_{g,n}$ is an orientable surface with genus $g$ and $n$ punctures such that $\chi(S_{g,n})<0$. In later sections, given a multicurve $Z$, we investigate the cube complex $\mathscr{X}_Z$ determined by the wall space $(\widetilde{S_{g,n}},\mathcal{W}_Z)$. In Section \ref{sec:CubeCplx}, we first review basic notions of CAT(0) cube complexes $X$. Then, we discuss {\it tree coordinate systems} (Sections \ref{subsec:TreeCoordinateSystem} and \ref{subsec:OrientedTreeCoordi}) and orientations of edges and walls of $X$ (Proposition \ref{prop:EdgeOri_WallOri}). We also compare the Roller boundary $\partial_R X$ and the Gromov boundary $\partial_G X$ (Theorem \ref{thm:Roller-to-Gromov}). The content in this section applies to general CAT(0) cube complexes. 

\subsection*{Definitions regarding curves and multicurves} A simple closed curve $\gamma$ on $S_{g,n}$ is {\it trivial} if it bounds a disk, {\it peripheral} if it bounds a once-punctured disk, and {\it essential} otherwise. A {\it multicurve} on $S_{g,n}$ is a collection of essential simple closed curves that are pairwise non-homotopic to each other. We emphasize that our definition allows curves in a multicurve to {\bf intersect} each other.

For a simple closed curve or an arc $\gamma$ joining two punctures, we denote by $[\gamma]$ its homotopy class. For two simple closed curves or arcs joining punctures $\alpha$ and $\beta$, we denote by $\pair{\alpha}{\beta}$ their intersection number  $|\alpha \cap \beta|$. The intersection number of homotopy classes is defined by $\pair{[\alpha]}{[\beta]}=\min |\alpha' \cap \beta'|$ where the minimum is taken over representatives $\alpha'\in [\alpha]$ and $\beta'\in [\beta]$. If $\pair{[\alpha]}{[\beta]}=\pair{\alpha}{\beta}$, we say that $\alpha$ and $\beta$ are in {\it minimal position}. We extend these definitions to multicurves in a natural way. For multicurves $A=\{\alpha_1,\alpha_2,\dots,\alpha_k\}$ and $B=\{\beta_1,\beta_2,\dots,\beta_l\}$, the intersection number $\pair{[A]}{[B]}$ is defined as $\sum_{i,j}\pair{[\alpha_i]}{[\beta_j]}$. A multicurve $\Gamma$ is in minimal position if and only if $\pair{[\Gamma]}{[\Gamma]}=\pair{\Gamma}{\Gamma}$. We implicitly assume that multicurves are in minimal position unless stated otherwise.

For a multicurve $\Gamma$ of $S_{g,n}$, we say that $\Gamma$ {\it fills} $S_{g,n}$, or, equivalently, $\Gamma$ is a {\it filling} multicurve, if $\pair{[\Gamma]}{[\alpha]}>0$ for every essential simple closed curve $\alpha$ in $S_{g,n}$ 

\subsection{Wall spaces, dual cube complexes, and hyperbolicity}
Let us review basic definitions and theorems for CAT(0) cube complexes and wall spaces.

\subsection*{CAT(0) cube complexes} For an integer $n\ge0$, an {\it $n$-cube} is a copy of $[0,1]^n$. Its {\it faces} are subcubes defined by restricting of some coordinates to $0$ or $1$. A {\it cube complex} $X$ is a cell complex obtained by gluing cubes along their faces. We also refer to $0$-cubes and $1$-cubes in $X$ as {\it vertices} and {\it edges}, respectively.

The \emph{link} of $(0, \ldots, 0)$ in $[0, 1]^n$ is the simplicial complex formed by the intersection of $[0, 1]^n$ with $\sum_i x_i =1 $.  The link $\link(v)$ of any other vertex of $[0, 1]^n$ is defined similarly, so that the vertices of $\link(v)$ are the vertices of $[0, 1]^n$ that are adjacent to $v$.
For a vertex $v$ of  a cube complex $X$, we let ${\rm link}(v)$ be the union of the links of $v$ in each cube for which $v$ is a vertex; this is the \emph{disjoint} union of these links, quotiented by the identifications of simplicies induced by the gluing of the cubes that make up $X$. 
Recall that a simplicial complex is a {\it flag complex} if every pairwise adjacent $n+1$ vertices spans an $n$-simplex. A cube complex $X$ is {\it non-positively curved} if the link of every vertex is a flag complex. A {\it CAT(0) cube complex} is a cube complex that is simply connected and non-positively curved. This definition is equivalent to the CAT(0) property as a metric space, where each cube is equipped with the standard Euclidean metric ($L^2$-metric), \cite[Theorem 5.4, p.206]{BH_MetricSpBook}.

We often use $L^1$-metric on $X$ for which each edge has length one. We denote by $d_{(X,L^1)}$ and $d_{(X,L^2)}$, or simply $d_{L^1}$ and $d_{L^2}$ when $X$ doesn't need to be specified, the $L^1$ and $L^2$-metrics on $X$, respectively. If $\dim(X)<\infty$, then the two metrics are bi-Lipschitz equivalent.

\subsection*{Median graphs}
Let $G$ be a graph. Consider the combinatorial distance $d$ on $G$, where each edge has length one. Then for any $x,y\in \Vertex(G)$, $d(x,y)$ is the smallest number of edges in a path joining $x$ and $y$. We say that $G$ is a {\it median graph} if for each triple $x,y,z\in \Vertex(G)$, there exists a unique vertex $m=m(x,y,z)\in \Vertex(G)$, called the {\it median} of $x,y$, and $z$, such that $d(x,y)=d(x,m)+d(m,y)$, $d(y,z)=d(y,m)+d(m,z)$, and $d(z,x)=d(z,m)+d(m,x)$.
\begin{prop}[{\cite{Ger_CubeCplx_Median,Roller_CubeCplx_Median,Chepoi_CubeCplx_Median}}]\label{prop:MedianGraph}
    A graph $G$ is a median graph if and only if $G$ is the $1$-skeleton of a CAT(0) cube complex, which is uniquely determined as the cubical completion.
\end{prop}
For a median graph $G$, the {\it cubical completion of $G$} is a cube complex constructed as follows: fill each $(k-1)$-skeleton of a $k$-cube embedded in $X^{(k-1)}$ with a $k$-cube, inductively with $k\ge 1$ and $X^{(1)}:=G$.

\subsection*{Wall spaces and dual cube complexes}
Suppose that $S$ is a connected topological space. A {\it wall} $w$ is a connected subspace of $S$ such that $S\setminus w$ consists of two connected components. Let $\mathcal{W}$ be a collection of walls such that (a) (locally finite) for any $x,y\in S$, there exist only finitely many walls separating $x$ and $y$, and (b) (finite width) there exists $N>0$ such that for any collection of pairwise intersecting walls $\{W_i\}_{i\in I}$, we have $|I|<N$. We remark that the conditions (a) and (b) correspond to the local finiteness and finite width of pocets \cite{Sageev_CAT(0)CubeCplx}. Then, we refer to $(S,\mathcal{W})$ as a {\it wall space}. For each wall $w$, the closure of  a connected component of $\widetilde{S_{g,n}}\setminus w$ is a {\it half-space} of $w$. Every wall $w$ has two half-spaces, denoted by $\overleftarrow{w}$ and $\overrightarrow{w}$, such that $\overleftarrow{w}\cup\overrightarrow{w}=\widetilde{S_{g,n}}$ and $\overleftarrow{w}\cap\overrightarrow{w}=w$. Denote by $\overline{\mathcal{W}}$ the set of half-spaces of walls.

Consider a multicurve $Z$ in a surface $S_{g,n}$. Suppose that $p\from \widetilde{S_{g,n}}\to S_{g,n}$ is the universal covering. Define $\mathcal{W}_Z$ as the collection of connected components of $p^{-1}(Z)$. Each element of $\mathcal{W}_Z$ is a wall of $\widetilde{S_{g,n}}$. It is easy to show that $(\widetilde{S_{g,n}},\mathcal{W}_Z)$ satisfies conditions (a) and (b) in the previous paragraph, and thus it is a wall space. By the assumption $Z$ is in minimal position, the walls do not form bigons; that is, any two distinct walls intersect in at most one point \cite[Lemma 1.8]{FM_PrimerMCG}.

\begin{defn}[Wall orientation]\label{defn:Orientation_Walls}
    Consider a wall space $(S,\mathcal{W})$. An {\it orientation} of walls, or a {\it wall orientation}, is the choice of one half-space for each wall. We express an orientation as a function $v\from \mathcal{W}\to \overline{\mathcal{W}}$ such that $v(w)\in\{\overleftarrow{w}, \overrightarrow{w}\}$. An orientation of walls $v$ is {\it admissible} if
    \begin{enumerate}
        \item (Consistency) for every pair of walls $w_1$ and $w_2$, we have $v(w_1) \cap v(w_2)\neq \emptyset$, and
        \item (DCC) for any $x\in \widetilde{S_{g,n}}$, all but finitely many walls $w$ satisfy $x\in v(w)$.
    \end{enumerate}
    If $v$ satisfies (1) but not (2), then we say that $v$ is {\it consistently oriented toward infinity}.
\end{defn}

The condition (2) in Definition \ref{defn:Orientation_Walls} is called the {\it descending chain condition (DCC)} because it is equivalent to that every strictly nested sequence of half-spaces, $\overrightarrow{w_1}\supsetneq\overrightarrow{w_2}\supsetneq\dots$, is a finite sequence.

For a wall space $(S,\mathcal{W})$, we define a dual cube complex $\mathscr{X}$ as follows. First, the vertex set $\Vertex(\mathscr{X})$ is the set of admissible orientations of walls. We define an edge between two vertices $v_1,v_2$ if $v_1(w)=v_2(w)$ for all but one $w\in \mathcal{W}$. Then, for $k\ge2$, we inductively attach a $k$-cube to every $(k-1)$-dimensional subcomplex that is isomorphic to the $(k-1)$-skeleton of a $k$-cube. It is well-known that $\mathscr{X}$ is a CAT(0) cube complex \cite[Section 6.3]{Wise_Book}.

For a multicurve $Z$ in a surface $S_{g,n}$, we denote by $\mathscr{X}_Z$ the dual cube complex of $(\widetilde{S_{g,n}},\mathcal{W}_Z)$.

Suppose that $Z=\{\zeta\}$. Then $\mathcal{W}_Z$ is a disjoint union of walls. Hence, $\mathscr{X}_Z$ is isomorphic to the dual tree of $(\widetilde{S_{g,n}},\mathcal{W}_Z)$, whose vertices are connected components of $\widetilde{S_{g,n}} \setminus \mathcal{W}_Z$ and edges are defined for pairs of adjacent components.

Although we only focus on wall spaces for hyperbolic surfaces, the construction of dual cube complexes from wall spaces applies in more general settings; see \cite{Sageev_CAT(0)CubeCplx,Wise_Book} for introductory accounts.

\subsection*{Hyperplanes} Consider a $k$-cube $[0,1]^k=\{(x_1,x_2,\dots,x_k)~|~x_i\in [0,1]\}$. A {\it mid-cube} of $[0,1]^k$ is a $(k-1)$-cube embedded in $[0,1]^k$ defined by $x_i=1/2$ for some $i\in\{1,2,\dots,k\}$.

Suppose $X$ is a CAT(0) cube complex. For any $k$-cube $C_k$ in $X$, we define mid-cubes of $C_k$ as the pullback of the mid-cubes of $[0,1]^k$ via a homeomorphism $C_k \to [0,1]^k$. A {\it hyperplane} $H$ in $X$ is a connected union of mid-cubes such that for any cube $C\subset X$, $C\cap H$ is either  empty or a mid-cube of $C$. Hyperplanes in the dual cube complex $\mathscr{X}_Z$ of the wallspace $(\widetilde{S_{g,n}},\mathcal{W}_Z)$ are in one-to-one correspondence with walls. For any edge $e$ of $X$, there exists a unique hyperplane $H_e$ that passes through the midpoint of $e$. We say that $H_e$ is the {\it dual hyperplane} of $e$, and $e$ is a {\it dual edge} of $H_e$. The dual hyperplane is unique for each edge, whereas the dual edges of a hyperplane may not be unique.

Let $H$ be a hyperplane of $X$. The {\it carrier} of $H$, denoted by $N(H)$, is the subcomplex of $X$ consisting of cubes intersecting $H$. Then $N(H)$ is a convex subcomplex of $X$ \cite[Corollary 3.13]{Wise_Book} and isomorphic to $H\times [0,1]$ \cite[Corollary 3.4]{Wise_Book}, with both results relying on the CAT(0) property of $X$. Here, we see $H$ as a cube complex where each mid-cube of $X$ contained in $H$ is considered as a cube in $H$.

We say that a sequence, possibly infinite, of walls $w_1,w_2,\dots,w_k$ is {\it nested} if we can choose a half-space $\overrightarrow{w_i}$ for each $i\in \{1,2,\dots,k\}$ such that $\overrightarrow{w_i}\supset \overrightarrow{w_j}$ for every $i<j$. A sequence of hyperplanes in $\mathscr{X}_Z$ is nested if its corresponding sequence of walls is nested.

It follows from the local finiteness of walls---for any $x\neq y\in X$, there exist at most finitely many walls separating $x$ and $y$---that for any infinite nested sequence of walls $w_i$'s, we have $\bigcap_{i\ge1} \overrightarrow {w_i}=\emptyset$.

The next lemma follows from elementary properties of hyperbolic spaces.

\begin{lem}\label{lem:Limit_Seg_Geod}
    Fix a multicurve $Z$ on a surface $S_{g,n}$ with $\chi(S_{g,n})<0$. Suppose that $w_1,w_2,\dots$ is an infinite nested sequence of walls from $Z$, such that there exist half-spaces $\overrightarrow{w_i}$'s of $w_i$'s  with $\overrightarrow{w_i}\supset \overrightarrow{w}_j$ for any $i<j$. Then, there exists a unique point $\theta \in \partial \widetilde{S_{g,n}}$ such that $\overrightarrow{w_i}\to \{\theta\}$ with respect to the Hausdorff topology on $\widetilde{S_{g,n}} \cup \partial \widetilde{S_{g,n}}$.
\end{lem}

\subsection*{Hyperbolicity of CAT(0) cube complex}

\begin{thm}\label{thm:HypCubeCplx}
    For any multicurve $Z$ on $S_{g,n}$, the cube complex $\mathscr{X}_Z$ is Gromov hyperbolic.  
\end{thm}
\begin{proof}
    The fundamental group $\pi_1(S_{g,n})$ is Gromov hyperbolic, and its action on $\mathscr{X}_Z$ is cocompact. The stabilizer of each hyperplane is conjugate to the infinite cyclic subgroup generated by an element $\zeta\in Z$, and hence is a quasiconvex subgroup of $\pi_1(S_{g,n})$. Therefore, by \cite[Theorem~A and Corollary~B]{GM_Hyp_Gp_Act_Improper}, the cube complex $\mathscr{X}_Z$ is Gromov hyperbolic.
\end{proof}

The Gromov hyperbolicity of a CAT(0) cube complex can be characterized in terms of the intersection patterns of its hyperplanes. Consider a CAT(0) cube complex $X$. A {\it grid} $\mathcal{G}$ of hyperplanes is a pair $(\mathcal{H},\mathcal{V})$ of nested sequences, possibly infinite, of hyperplanes, $\mathcal{H}=(H_{h,1},H_{h,2},\dots,H_{h,k_h})$ and $\mathcal{V}=(H_{v,1},H_{h,2},\dots,H_{h,k_v})$, such that each hyperplane in $\mathcal{H}$ intersects every hyperplane in $\mathcal{V}$, and each hyperplane in $\mathcal{V}$ intersects every $\mathcal{H}$. Under the bijection between walls and hyperplanes, we similarly define grids of walls. We say that $X$ has {\it uniformly thin} grids of hyperplanes if there exists $C>0$ such that for every grid $(\mathcal{H},\mathcal{V})$, we have $\min(|\mathcal{H}|,|\mathcal{V}|)<C$, where $|S|$ denotes the cardinality of a set $S$.
\begin{lem}{\cite[Theorem 3.1]{Genev_HypCubCplx}}\label{lem:UnifThinGrids}
    Suppose that $X$ is a CAT(0) cube complex. Then, $X$ has uniformly thin grids of hyperplanes if and only if $X$ is Gromov hyperbolic. 
\end{lem}

\subsection{Tree coordinate system}\label{subsec:TreeCoordinateSystem} Recall that a continuous map $f\from X \to Y$ is {\it monotone} if $f^{-1}(y)$ is connected for any $y \in Y$. For cube complexes $X$ and $Y$, a continuous map $f\from X \to Y$ is a {\it cubical map} if each cube $C_X$ of $X$ is mapped onto a cube $C_Y$ of $Y$. For convenience, we also assume that it maps $C_X$ to $C_Y$ via a coordinate projection $p\from [0,1]^k\to [0,1]^l$ with $k\ge l$.

\begin{defn}[Tree coordinate system]\label{defn:TreeCoordiate}
    Let $X$ be a cube complex. A finite family of surjective cubical maps $\{\pi_i\from X \to X_i\}_{i=1,2,\dots,n}$ is called a {\it tree coordinate system of $X$} if the following properties hold.
    \begin{enumerate}
        \item $X_i$'s are trees.
        \item (Midpoint-monotonicity) For each $i\in\{1,2,\dots,n\}$, the fiber of any midpoint of edge in $X_i$ is connected.
        \item (Efficiency) For every edge $e$ of $X$, there exists a unique $i \in \{1,2,\dots,n\}$ such that $\pi_i(e)$ is an edge; for any $j\neq i$, $\pi_j(e)$ is a vertex of $X_j$.
    \end{enumerate}
When satisfying (2) (resp.\@ (3)), the family $\{\pi_i\from X \to X_i\}_{i=1,2,\dots,n}$ is called {\it midpoint-monotone} (resp.\@ {\it efficient}).
\end{defn}
\begin{prop}\label{prop:Partition_TreeSystem}
    For a multicurve $Z$ of $S_{g,n}$ in minimal position, consider the wallspace $(\widetilde{S_{g,n}},\mathcal{W}_Z)$. Suppose that $\mathcal{W}_Z=\mathcal{W}_1\sqcup\mathcal{W}_2\sqcup\dots\sqcup \mathcal{W}_k$ is a partition  $\mathcal{W}_Z$ such that for each $i\in \{1,2,\dots, k\}$, the walls in $\mathcal{W}_i$ are pairwise disjoint. Let $X$ be the dual cube complex of $(\widetilde{S_{g,n}},\mathcal{W}_Z)$, and let $X_i$ be the dual cube complex of $(\widetilde{S_{g,n}},\mathcal{W}_i)$ for each $i\in \{1,2,\dots,k\}$. Then there exists a naturally defined tree coordinate system $\{\pi_i\from X \to X_i\}$.
\end{prop}
\begin{proof}
    For each admissible orientation for $\mathcal{W}_Z$, its restriction to each $\mathcal{W}_i$ is also admissible. This gives rises to a family of surjective cubical maps $\{\pi_i\from X \to X_i\}_{i=1,2,\dots.k}$. This family is efficient because $\mathcal{W}_1\sqcup\mathcal{W}_2\sqcup\dots\sqcup \mathcal{W}_k$ is a partition of $\mathcal{W}$. The midpoint-monotonicity follows from the connectedness of each wall.
\end{proof}
\begin{defn}[Tree coordinate system of $\mathscr{X}_Z$]\label{defn:TreeCrdSysXscr}
Let $Z=\{\zeta_1,\zeta_2,\dots,\zeta_k\}$ be a multicurve of $S_{g,n}$ in minimal position. We define the {\it tree coordinate system of $\mathscr{X}_Z$} as the tree coordinate system $\{\pi_i\from \mathscr{X}_Z \to \mathscr{X}_{\zeta_i}\}_{i=1,2,\dots,k}$ defined by Proposition \ref{prop:Partition_TreeSystem} for the partition $\mathcal{W}_Z=p^{-1}(\zeta_1)\sqcup p^{-1}(\zeta_2)\sqcup \dots \sqcup p^{-1}(\zeta_k)$.
\end{defn}

For an edge-path $p\from I \to G$ of a graph $G$, we say that $p$ is {\it backtracking} if its two consecutive edges are the same edge. If $G$ is a tree, then $p$ is backtracking if and only if $p$ passes through an edge more than once.

\begin{lem}\label{lem:CubicalMaptoTreeMonotone}
    Suppose that $f\from X \to Y$ is a surjective cubical map where $X$ is a CAT(0) cube complex and $Y$ is a tree. Then, the following are equivalent.
    \begin{enumerate}
        \item (Midpoint-monotone) The map $f$ is monotone over midpoints of edges in $Y$.
        \item (Monotone) The map $f$ is monotone.
        \item (No backtracking) For any edge-path geodesic $\gamma$ in $X$, $f\circ \gamma$ is not backtracking.
    \end{enumerate}
    Moreover, if $f$ is monotone, then $f^{-1}(v)$ is a convex subcomplex for any $v\in \Vertex(Y)$ and $f^{-1}(y)$ is a hyperplane for any midpoint $y$ of edge in $Y$.
\end{lem}
\begin{proof}
    (Monotone)$\Rightarrow$(Midpoint-monotone): Trivial.
    
    (Midpoint-monotone)$\Rightarrow$(Monotone): Suppose that $f$ is monotone over midpoints of edges in $Y$. Let $y$ be a vertex of $Y$. Let $m$ be the midpoint of an edge in $Y$ incident to $y$. The preimage $H_m:=f^{-1}(m)$ is a hyperplane, such that $X \setminus H_m$ has two components $\overleftarrow{H_m}, \overrightarrow{H_m}$. Without loss of generality, suppose that $\overrightarrow{H_m}$ contains $f^{-1}(y)$. Denote by $M(\overrightarrow{H_m})$ the largest cube subcomplex of $X$ contained in $\overrightarrow{H_m}$. By using the fact that the carrier $N(H_m)$ is a convex subcomplex of $X$, one can show that $M(\overrightarrow{H_m})$ is a convex subcomplex of $X$ as well. Then $f^{-1}(y)$ is equal to the intersection of $M(\overrightarrow{H_m})$'s for the midpoints of all edges incident to $Y$, and hence it is a convex subcomplex of $X$. In particular, $f^{-1}(y)$ is connected.

    (Midpoint-monotone)$\Rightarrow$(No backtracking): The midpoint-monotonicity of $f$ implies that for any midpoint $y$ of an edge of $Y$, $f^{-1}(y)$ is a hyperplane. Hence, if $f\circ p$ is backtracking, where $p$ is an edge-path geodesic in $X$, then $p$ passes through the same hyperplane more than once. This contradicts \cite[Lemma 3.9]{Wise_Book}.

    (No backtracking)$\Rightarrow$(Midpoint-monotone): Suppose that there exists a midpoint $y$ of an edge $e_y$ of $Y$ such that $f^{-1}(y)$ contains two distinct hyperplanes $H$ and $H'$. Denote by $w_1$ and $w_2$ the endpoints of $e_y$ in $Y$. There exist edges $e$ and $e'$ of $X$ such that $f(e)=f(e')=e_y$, $e\cap H\neq \emptyset$, and $e'\cap H'\neq \emptyset$. Denote by $v_1$ and $v_2$ (resp.\@ $v_1'$ and $v_2'$) the endpoints of $e$ (resp.\@ $e'$) such that $f(v_1)=w_1$ and $f(v_2)=w_2$ (resp.\@ $f(v_1')=w_1$ and $f(v_2')=w_2$). 
    
    Consider the half-spaces $\overleftarrow{H},\overrightarrow{H}$ of $H$ such that $v_1\in \overleftarrow{H}$ and $v_2\in \overrightarrow{H}$. Since $e'$ does not intersects $H$, either $\overleftarrow{H}$ or $\overrightarrow{H}$ contains both $v_1'$ and $v_2'$. First, suppose $v_1',v_2'\in \overleftarrow{H}$. Consider an edge-path geodesic $p$ from $v_2'$ to $v_2$. Then $p$ has to cross $H$. Every edge of $X$ crossing $H$ is mapped by $f$ onto $e$. This implies that $f\circ p$ passes through $e$ while starting and ending at $w_2$. Hence $f\circ p$ is backtracking. If $v_1',v_2'\in \overrightarrow{H}$, then we consider an edge-path geodesic $p$ from $v_1'$ to $v_1$ and apply the similar argument.
\end{proof}

\begin{lem}\label{lem:ProdMapTrees}
Suppose that $\{\pi_i\from X \to X_i\}_{i=1}^n$ is a family of surjective cubical maps from a cube complex $X$ to trees $X_i$'s. Define a product map $\times_{i=1}^n \pi_i\from X \to \prod_{i=1}^n X_i$ by
    \begin{equation}\label{eqn:ProdMap}
        x \mapsto (\pi_1(x),\pi_2(x),\dots,\pi_n(x)).
    \end{equation}
Consider $\prod_{i=1}^n X_i$ as a cube complex. Then the following hold.
\begin{enumerate}
    \item The product map $\times_{i=1}^n \pi_i\from X \to \prod_{i=1}^n X_i$ is cubical if and only if for each edge $e$ of $X$, there exists at most one $i\in \{1,2,\dots,n\}$ such that $\pi_i(e)$ is an edge.
    \item The product map $\times_{i=1}^n \pi_i\from X \to \prod_{i=1}^n X_i$ is cubical and homeomorphic on each cube if and only if $\{\pi_i\from X \to X_i\}_{i=1,2,\dots,n}$ is efficient.
    \item Suppose that $\{\pi_i\from X \to X_i\}_{i=1,2,\dots,n}$ is efficient. The product map $\times_{i=1}^n \pi_i\from X \to \prod_{i=1}^n X_i$ is an embedding if and only if $\times_{i=1}^n\pi_i$ is injective on $\Vertex(X)$. In particular, if $X$ is CAT(0) and $\{\pi_i\from x\to X_i\}$ is efficient and monotone, then $\times_{i=1}^n\pi_i$ is an embedding.
\end{enumerate}
\end{lem}
\begin{proof}
    (1) If $\pi_1(e)$ and $\pi_2(e)$ are edges, then $\pi_1\times \pi_2\from X\to \prod_{i=1}^2 X_i$ maps $e$ to the diagonal of the cube $\pi_1(e)\times \pi_2(e)$. Thus $\times_{i=1}^n \pi_i$ is not cubical if there exists an edge $e$ such that $\pi_i(e)$ is an edge for more than one $i$. The converse follows similarly.

    (2) Suppose that there exists an edge $e$ such that $\pi_i(e)$ is a vertex for each $i\in\{1,2,\dots,n\}$. Then $\times_{i=1}^n\pi_i(e)$ is a vertex, and thus $\times_{i=1}^n\pi_i$ is not injective. The converse follows similarly.
    
    (3) The equivalence statement is straightforward from (2). Let us prove the other statement. For any ${\bf v}:=(v_1,v_2,\dots,v_n)\in \Vertex( \prod_{i=1}^n X_i)$, we have
    \[
        \left(\times_{i=1}^n \pi_i\right)^{-1}( {\bf v})=\bigcap_{i=1}^n \pi_i^{-1}(v_i).
    \]
    By Lemma \ref{lem:CubicalMaptoTreeMonotone}, $\bigcap_{i=1}^n \pi_i^{-1}(v_i)$ is either the empty set or a convex subcomplex. Suppose $\bigcap_{i=1}^n \pi_i^{-1}(v_i)\neq \emptyset$. If $\bigcap_{i=1}^n \pi_i^{-1}(v_i)$ contains an edge $e$, then $\pi_i(e)=v_i$ for each $i\in\{1,2,\dots,n\}$. It contradicts the efficiency of the tree coordinate system. Hence, $\times_{i=1}^n \pi_i( {\bf v})$ is a singleton.
\end{proof}
\begin{prop}\label{prop:TreeCoordiSystem}
    Let $X$ be a CAT(0) cube complex. Suppose that $\{\pi_i\from X \to X_i\}_{i=1,2, \dots, n}$ is a tree coordinate system for $X$. Then the following hold.
    \begin{enumerate}
        \item For every $k$-cube $C_k$ in $X$, there exists a subset $\{i_1,i_2,\dots,i_k\} \subset \{1,2,\dots,n\}$ of $k$ distinct elements such that for every $j\in \{1,2,\dots,k\}$, $\pi_{i_j}(C_k)=e_{i_j}$ for some edge $e_{i_j}$ in $X_{i_j}$. Thus, we have $\dim(X) \le n$.
        \item Every hyperplane is of the form $\pi_i^{-1}(x)$, where $x$ is the midpoint of an edge in $X_i$.
        \item (Systematic monotonicity) For any $\{i_1,i_2,\dots,i_k\}\subset\{1,2,\dots,n\}$ and any choice $x_{i_j}\in X_{i_j}$ for $1\le j\le k$, the intersection of fibers $ \bigcap_{j=1}^k \pi_{i_j}^{-1}(x_{i_j})$ is either connected or empty.
        \item For any $v,w\in \Vertex(X)$,
        \begin{equation}\label{eqn:L1distance equality}
            d_{(X,L^1)}(v,w) = \sum_{i=1}^n d_{(X_i,L^1)}(\pi_i(v),\pi_i(w)).
        \end{equation}
        Equivalently, an edge-path $\gamma$ in $X$ is a geodesic edge-path if and only if the projected edge-path $\pi_i\circ \gamma$ is not backtracking for any $i\in \{1,2,\dots,n\}$.
    \end{enumerate}    
\end{prop}
\begin{proof}
    (1) and (2) easily follow from the monotonicity and the efficiency of tree coordinate system (Definition \ref{defn:TreeCoordiate}), and (4) follows from Lemma \ref{lem:CubicalMaptoTreeMonotone}.

    Let us prove (3). Without loss of generality, we assume that $i_j=j$ and $x_{j}\in X_j$ is either a vertex or a midpoint of edge for $i \in \{1,2,\dots ,k\}$. If $x_i\in \Vertex(X_i)$ then $\pi_i^{-1}(x_i)$ is a convex subcomplex of $X$ by Lemma \ref{lem:CubicalMaptoTreeMonotone}. If $x_i$ is midpoint, then $\pi_i^{-1}(x_i)$ is a hyperplane, which is also convex subset because its carrier is a convex subcomplex. Hence the intersection $\bigcap_{i=1}^k \pi_i^{-1}(x_i)$ is either empty or convex.
\end{proof}

\subsection{Oriented cube complexes and the Roller boundary} Later, we will see that any $x$-generic measured foliation $\Fcal\in \MF(S_{g,n+1})$ induces orientations of edges in the cube complex $\mathscr{X}_Z$. Under the identification between edges of $\mathscr{X}_Z$ and walls in $\mathcal{W}_Z$, it gives rises to a consistent orientation of walls toward infinity (Definition \ref{defn:Orientation_Walls}). That is, $\Fcal$ corresponds to a unique element in the Roller boundary $\partial_R X$. In this section, we investigate general properties of directed edge-paths in oriented cube complexes and their relations with the Roller boundary.

Consider a $k$-cube $C=\{(x_1,x_2,\dots,x_k)~|~ x_i\in [0,1], 1\le i \le k\}$. Edges of $C$ that are of the form $(\epsilon_1,\epsilon_2,\dots,\epsilon_{i-1},[0,1],\epsilon_{i+1},\dots,\epsilon_k)$, where $\epsilon_j\in\{0,1\}$ for $j \in \{1,2,\dots,k\} \setminus \{i\}$, are called $i^{th}$-coordinate edges. We say that two edges of $C$ are {\it parallel} if both are $i^{th}$-coordinate edges for the same $i\in\{1,2,\dots,k\}$.

\begin{defn}[Edge orientations, mono-outgoing cubes]\label{defn:OrientedCubeCplx}
    Let $X$ be a cube complex. An {\it edge orientation} of $X$, is an orientation of edges of $X$ for which parallel edges of each cube $C=[0,1]^k$ are oriented parallelly, i.e., on each cube identified with $[0,1]^k$, the push-forwarded orientations from parallel edges via coordinate projection maps $[0,1]^k\to [0,1]$ coincide.
    
    We say that an edge orientation of $X$ has {\it mono-outgoing cubes} if, for each vertex $v$, all the outgoing edges of $v$, if exist, are contained in a single cube.
\end{defn}

\begin{defn}[Directed edge-paths]
    Let $X$ be a cube complex with an edge orientation. A {\it directed edge-path} of $X$ is a continuous map $p\from [0,n] \to X$ such that
    \begin{enumerate}
        \item $p$ maps integers to vertices of $X$,
        \item $p|_{[i,i+1]}$ is a homeomorphism onto an edge for any integer $0 \le i<n$, and
        \item the map $p$ is orientation-preserving where each edge $[i,i+1]$ is oriented from $i$ to $i+1$.
    \end{enumerate}
    An {\it infinite directed edge-path} is a continuous map $p\from [0,\infty) \to X$ with the same properties. A {\it reversely directed edge-path} is similarly defined with the orientation-preserving condition being replaced with orientation-reversing. We also refer to the sequence of the adjacent vertices $(p(i))_{i\ge0}$ as a directed edge-path. We define {\it directed edge-paths with layovers} by modifying (1) to
    \begin{enumerate}[resume]
        \item[(1')] $p|_{[i,i+1]}$ is either a homeomorphism onto an edge or a constant map to a vertex.
    \end{enumerate}
    A directed edge-path with layovers is {\it bounded} (resp.\@ {\it unbounded}) if its image consists of finitely (resp.\@ infinitely) many distinct edges.
\end{defn}

\begin{prop}\label{prop:DirectedPathGeodesic}
    Consider a CAT(0) cube complex $X$ with an edge orientation. Then, every directed edge-path $p\from [0,l]\to X$ is an edge-path geodesic.
\end{prop}
\begin{proof}
    Suppose that there exists a directed edge-path $p\from [0,l]\to X$ that is not an edge-path geodesic. By \cite[Lemma 3.9]{Wise_Book}, $p$ passes through a hyperplane $H$ at least twice. Since dual edges of $H$ are oriented in a parallel way along $H$, it contradicts the assumption that $p$ is a directed edge-path.
\end{proof}

\subsection*{Edge orientations vs. Wall orientations} Suppose that $X$ is a CAT(0) cube complex and $\Hcal_X$ is the set of hyperplanes. We can consider $(X,\Hcal_X)$ as a wall space. By a {\it wall orientation} of $X$, we mean the wall orientation of $(X,\Hcal_X)$. In this context, we use the terms walls and hyperplanes interchangeably.

Recall that the {\it Roller boundary $\partial_R X$} is defined as the set of orientations of walls $\zeta\from \mathcal{H}_X \to \overline{\mathcal{H}_X}$ that are consistently oriented toward infinity (Definition \ref{defn:Orientation_Walls}). Then, $\overline{X}^R:=\Vertex(X)\cup \partial_R X$ is embedded in $\{0,1\}^{\Hcal_X}$ as a totally disconnected compact Hausdorff space such that $\Vertex(X)$ is dense in $\overline{X}^{R}$ \cite[Proposition 6.11]{Roller_CubeCplx_Median}.

Note that every hyperplane intersects each cube along parallel edges. Hence, an edge orientation of $X$ determines a wall orientation, and vice versa.

\begin{prop}\label{prop:EdgeOri_WallOri}
    Suppose $X$ is a CAT(0) cube complex and $\Hcal_X$ is the set of hyperplanes. Consider $(X,\mathcal{H}_X)$ as a wall space. Then, there is a bijection
    \[
        \left\{\textnormal{edge orientations of }X\right\}
        \leftrightarrow
        \left\{\textnormal{wall orientations of }(X,\mathcal{H}_X)\right\},
    \]
    which restricts to the bijection
    \[
        \left\{\begin{array}{c} \textnormal{edge orientations of }X \\ \textnormal{with mono-outgoing cubes}\end{array}\right\}
        \leftrightarrow
        \left\{\begin{array}{c} \textnormal{wall orientations of }(X,\mathcal{H}_X) \\ \textnormal{that are consistently oriented}\end{array}\right\}.
    \]
    Moreover, the vertex set $\Vertex(X)$ is in bijection with
    \[
        \left\{\begin{array}{c} \textnormal{edge orientations of }X \\ \textnormal{with mono-outgoing cubes}\\ \textnormal{with no infinite directed edge-paths}\end{array}\right\}
        \leftrightarrow
        \left\{\begin{array}{c} \textnormal{wall orientations of }(X,\mathcal{H}_X) \\ \textnormal{that are consistently oriented}\\ \textnormal{with the descending chain condition}\end{array}\right\},
    \]
    and the Roller boundary $\partial_R X$ is in bijection with
    \[
        \left\{\begin{array}{c} \textnormal{edge orientations of }X \\ \textnormal{with mono-outgoing cubes}\\ \textnormal{and infinite directed edge-paths}\end{array}\right\}
        \leftrightarrow
        \left\{\begin{array}{c} \textnormal{wall orientations of }(X,\mathcal{H}_X) \\ \textnormal{that are consistently oriented}\\ \textnormal{toward infinity}\end{array}\right\}.
    \]
\end{prop}
\begin{proof}
    The first bijection is immediate. The existence of infinite directed edge-paths is equivalent not to satisfy the descending chain condition (Definition \ref{defn:Orientation_Walls}). Hence, the third and fourth bijections are obtained, assuming the second bijection. Below, we show the equivalence between having mono-outgoing cubes and being consistently oriented. 

    Consider a wall orientation $\zeta\from \Hcal_X\to \overline{\Hcal_X}$ and its corresponding edge orientation of $X$. Suppose that there is a vertex $v\in X$ such that it has two outgoing edges $e_1$ and $e_2$ that are not contained in the same cube. Then, their dual hyperplanes $H_1$ and $H_2$ are disjoint, and both $\zeta(H_1)$ and $\zeta(H_2)$ do not contain $v$. It follows that $\zeta(H_1) \cap \zeta(H_2)=\emptyset$, and thus $\zeta$ violates the consistency.

    Conversely, consider a wall orientation $\zeta\from \mathcal{H}_X \to \overline{\mathcal{H}_X}$ violating the consistency. Then, there exist disjoint hyperplanes $H_1, H_2\in \mathcal{H}_X$ with $\overleftarrow{H_1}:=\zeta(H_1)$ and $\overrightarrow{H_2}:=\zeta(H_2)$ such that $\overleftarrow{H_1} \cap \overrightarrow{H_2}=\emptyset$. By replacing $H_2$ with some hyperplane separating $H_1$ and $H_2$, we can assume that no hyperplanes separate $H_1$ and $H_2$. Choose a vertex $v$ in ${\rm int}(\overrightarrow{H_1} \cap \overleftarrow{H_2})$. Then, $v$ has two outgoing edges $e_1$ and $e_2$, which are dual to $H_1$ and $H_2$. Since $H_1$ and $H_2$ are disjoint, $e_1$ and $e_2$ are not contained in the same cube.
\end{proof}

\begin{defn}[Oriented CAT(0) cube complex]
    A CAT(0) cube complex $X$ is said to be {\it oriented} if it is endowed with either an edge orientation or a wall orientation---or with both, in such a way that they correspond via the bijection given in Proposition \ref{prop:EdgeOri_WallOri}. For any $\zeta \in \Vertex(X)\cup \partial_R X$, we say that $x$ is {\it oriented toward $\zeta$} if the orientation corresponds to $\zeta$ in Proposition \ref{prop:EdgeOri_WallOri}.
\end{defn}

\begin{lem}\label{lem:Inf_directed_edgepaths}
    Suppose that $X$ is a CAT(0) cube complex oriented toward $\zeta\in \partial_R X$. Then, for each $v\in \Vertex(X)$, there exists an infinite directed edge-path starting from $v$.
\end{lem}
\begin{proof}
      If every vertex has at least one outgoing edge, then, we can iteratively generate an infinite directed edge-path starting from any vertex. Thus, it suffices to show that every vertex has at least one outgoing edge.

      Suppose that there exists a vertex $w$ with no outgoing edges. Consider $\zeta$ as a function $\zeta \from \mathcal{H}_X \to \overline{\mathcal{H}_X}$.
      We first show that for any hyperplane $H\in \mathcal{H}_X$, we have $w\in\overrightarrow{H}:=\zeta(H)$. Suppose $w\notin\overrightarrow{H}$ for some $H$. If there exists a hyperplane $H'$ separating $w$ and $H$, then, by the consistency of $\zeta$, we also have $w \notin \overrightarrow{H'}:=\zeta(H')$. Hence, by replacing $H$ with some $H'$, we may assume that there is no hyperplane separating $H$ and $w$. It follows that there exists an edge $e$ incident to $w$ whose dual hyperplane is $H$. By our assumption on $w$, $e$ is an incoming edge of $w$. Hence, $w\in \overrightarrow{H}$. It is a contradiction. Hence, $w\in \zeta(H)$ for any hyperplane $H$.

    Next, note that $\zeta\in \partial_R X$ implies that there exists a sequence of hyperplanes $(H_i)_{i\ge1}$ such that $\overrightarrow{H_1}\supset\overrightarrow{H_2},\dots$ for $\overrightarrow{H_i}:=\zeta(H_i)$. By the local finiteness of wall space, for all but finitely many $H_i$'s, we have $w\notin \overrightarrow{H_i}$. This contradicts the conclusion of the previous paragraph.
\end{proof}

\subsection*{Roller boundary vs. Gromov boundary}
Recall that for a CAT(0) space $X$, the {\it visual boundary}, or the {\it boundary at infinity}, $\partial_\infty X$ is defined as the set of asymptotic equivalence classes of infinite geodesics. When $X$ is Gromov hyperbolic, the set of asymptotic equivalence classes of infinite geodesics is called the {\it Gromov boundary} $\partial_G X$. We note that Gromov hyperbolicity is invariant under quasi-isometries. Thus, Gromov hyperbolicity of a CAT(0) cube complex $X$ is a property independent of the choice of metric between $d_{L^1}$ and $d_{L^2}$, provided that $\dim(X)<\infty$. Therefore, for a Gromov hyperbolic CAT(0) cube complex $X$, the Gromov boundary $\partial_G X$, which is also equal to the boundary at infinity $\partial_\infty X$, can be identified with the asymptotic equivalence classes of infinite edge-path geodesics.

We will need the following lemma in the proof of Theorem \ref{thm:Roller-to-Gromov}. 
\begin{lem} \label{lem:consistent extension}
Suppose $X$ is a cube complex, and let $\Hcal$ be the set of hyperplanes of $X$. Suppose $\Hcal' \subset \Hcal$ and $\zeta\from \Hcal' \to \overline{\Hcal'}$ is a consistent orientation of this subset of hyperplanes. Then $\zeta$ can be extended to a consistent orientation of $\Hcal$.
\end{lem}
\begin{proof}
First, suppose that $\zeta$ is maximal, in the sense that it has no consistent extension to any $\Hcal'' \supsetneq \Hcal$. We will show that $\Hcal' = \Hcal$. If not, choose $H \in \Hcal \setminus \Hcal'$. For any $M \in \Hcal'$, if $H$ and $M$ are disjoint, and $H \not\subset \zeta(M)$, then we say $M$ \emph{forces} the orientation of $H$: any extension of $\zeta$ to $H$ must be the half-space that contains $M$.
If $M_1, M_2 \in \Hcal'$ both force the orientation of $H$, and they lie on opposite sides of $H$, then their orientations are inconsistent. Therefore they lie on the same side of $H$, and hence force the same orientation. We conclude that there exists a orientation of $H$ that is consistent with $\zeta$ (which is unique if and only if there is an $M \in \Hcal$ that forces $H$). Hence any maximal $\zeta$ has domain equal to $\Hcal$. 

By Zorn's Lemma, for any given $\zeta\from \Hcal' \to \overline{\Hcal'}$, there is a maximal consistent extension. By the previous paragraph, it is an extension $\Hcal$. 
\end{proof}

\begin{thm}\label{thm:Roller-to-Gromov}
    Let $X$ be a Gromov hyperbolic CAT(0) cube complex.
    \begin{enumerate}
        \item Fix $\zeta \in \partial_R X$ and consider the corresponding edge orientation of $X$. Then, every two infinite directed edge-paths are asymptotic to each other.
    \end{enumerate}
    As a result, a map $\partial_{RG}\from \partial_R X\to \partial_G X$ is well-defined that sends $\zeta\in \partial_R X$ to the asymptotic equivalence class of edge-path geodesics containing the infinite directed edge-paths with respect to the edge orientation of $X$ induced by $\zeta$.
    \begin{enumerate}[resume]
        \item The map $\partial_{RG}\from \partial_R X\to \partial_G X$ is a continuous surjection.
        \item The map $\partial_{RG}\from \partial_R X\to \partial_G X$ is a homeomorphism if and only if there exist no hyperplanes that intersect a nested sequence of hyperplanes $(H_n)_{n\ge1}$.
    \end{enumerate}
\end{thm}
\begin{proof}
    In this proof, we use the $L^1$-metric $d:=d_{L^1}$.
    
    (1) Suppose that $p_1,p_2\from [0,\infty)\to X$ are directed edge-paths. Let $x \in X$ be a vertex; the reader may want to imagine that $x$ is close to the midpoint of $p_1(0)$ and $p_2(0)$. We let $N = \max_{i = 1, 2} d(x, p_i(0))$.  For $i = 1, 2$, we let $q_i$  be a geodesic edge-path from $x$ to $p_i(0)$.

    For each $n \in \mathbb{N}$, define $y_n:=p_1(n)$ and $z_n:=p_2(n)$. Recall that, as a graph, $X^{(1)}$ is a median graph (Proposition \ref{prop:MedianGraph}). Denote by $w_n\in \Vertex(X)$ the median of $x, y_n$, and $z_n$. Choose edge-path geodesics $[w_n,y_n]$ between $w_n$ and $y_n$, $[w_n,z_n]$ between $w_n$ and $z_n$, and $[x,w_n]$ between $x$ and $w_n$. We note that the concatenation $[x,w_n]*[w_n,y_n]$ (resp.\@ $[x,w_n]*[w_n,z_n]$) is an edge-path geodesic between $x$ and $y_n$ (resp.\@ $x$ and $z_n$). See Figure \ref{fig:AsympRays}.

   We let $\Hcal_n$ be the set of the dual hyperplanes $H_e$ of edges $e$ in $[w_n, y_n]$ for which $H_e$ does not cross $q_1$. Since every $H \in \Hcal_n$ separates $x$ and $y_n$, $\Hcal_n$ is a nested sequence of hyperplanes, under a suitable order of its elements.
   It also follows that $q_1 * p_1([0, n])$ crosses each $H\in \Hcal_n$. Since $q_1$ does not cross $H$, by the definition of $\Hcal_n$, $p_1([0, n])$ does, exactly once. Hence, we obtain $|\Hcal_n| \ge d(w_n, y_n) - N$. We likewise let $\Vcal_n$ be the set of dual hyperplanes $H_e$ of edges $e$ in $[w_n, z_n]$ for which $H_e$ does not cross $q_2$, and observe the analogous properties.    

    \begin{figure}[h!]
	\centering
        \def\svgwidth{0.4\textwidth}
        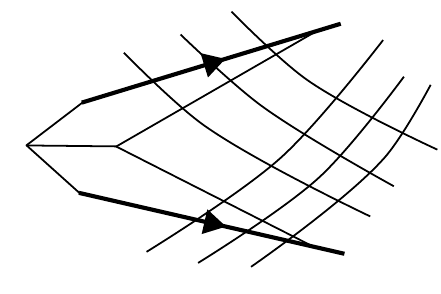
	\caption{}
    \label{fig:AsympRays}
    \end{figure}

    Next, we show that for any hyperplanes $H_1 \in \Hcal_n$ and $H_2 \in \Vcal_n$, $H_1$ and $H_2$ intersect. For the wall orientation $\zeta\from \mathcal{H}_X\to \overline{\mathcal{H}_X}$, the half spaces $\zeta(H_1)$ and $\zeta(H_2)$ do not contain the vertex $x$. Therefore, if $H_1$ and $H_2$ do not intersect, then it violates the consistency of wall orientation for $\zeta$.

    Since $[w_n,y_n]$ (resp.\@ $[w_n,z_n]$) is a geodesic edge-path, the hyperplanes dual to its edges are pairwise disjoint. Therefore $\mathcal{H}_n$ and $\mathcal{V}_n$ form a grid of hyperplanes. By Gromov hyperbolicity of $X$ (and Lemma \ref{lem:UnifThinGrids}), we conclude that $\min(|\Hcal_n|, |\Vcal_n|$) is uniformly bounded.

    We have  $d(x, y_n) \ge  n - N$ and $d(x, z_n) \le  n  + N$, and therefore
    \begin{align*} 
        d(w_n, y_n) &= \frac12(d(x, y_n) + d(y_n, z_n) - d(x, z_n)) \\
        &\ge \frac12 d(y_n, z_n) - 2N. 
    \end{align*}
    Hence $|\Hcal_n| \ge d(y_n, z_n)/2 - 3N$, and likewise for $\Vcal_n$. We conclude that $d(y_n, z_n)$ is bounded, and therefore $p_1$ and $p_2$ are asymptotic.
    
    (2) Let us prove the continuity first. The sets $P(\overrightarrow{H})$'s of asymptotic equivalence classes of infinite edge-path geodesics (eventually) contained in half-spaces $\overrightarrow{H}$'s form a subbasis of the topology of $\partial_G X$. For a half-space $\overrightarrow{H}$, the set $\partial_{RG}^{-1}(P(\overrightarrow{H}))$ is open (and closed), in the subspace topology of $2^{\mathcal{H}_X}$. Hence, $\partial_{RG}$ is continuous.

    Now we show the surjectivity of $\partial_{RG}$. Let $\gamma\from [0,\infty)\to X$ be an infinite edge-path geodesic. Let $H_1,H_2,\dots$ be the sequence of hyperplanes dual to the edges $\gamma([0,1]), \gamma([1,2]), \dots$. By \cite[Lemma 3.9]{Wise_Book}, the $H_i$'s are pairwise disjoint. Hence, the sequence $(H_n)_{n\ge1}$ is nested, i.e., we can choose a nested sequence of half-spaces $\overrightarrow{H_1}\supset\overrightarrow{H_2}\supset \dots$. Recall that by the local finiteness of walls, we have $\bigcap_{n\ge1}\overrightarrow{H_n}=\emptyset$. By Lemma \ref{lem:consistent extension}, there exists an consistent extension of this orientation of the $H_n$'s to all the hyperplanes of $X$; we have then obtained an element of $\partial_{RG}^{-1}(\{\gamma\})$.

       (3) For any edge-path geodesic $\gamma$ representing a point in $\partial_G X$, its fiber by $\partial_{RG}$ is a singleton if and only if there is a unique consistent extension of the orientation of the $H_n$ above. From the proof of Lemma \ref{lem:consistent extension} (and $\bigcap H_n = \emptyset)$, this occurs if and only if there is no hyperplane $H$ that intersects almost all the $H_n$. We therefore obtain (3).
\end{proof}

\begin{rem}
    A similar statement to Theorem \ref{thm:Roller-to-Gromov} is proven in \cite[Theorem D]{BF_CrossRatio_Cube_Hyp_Gp}, assuming that $X$ is uniformly locally finite.
    
    In \cite[Theorem D]{BF_CrossRatio_Cube_Hyp_Gp}, it is shown that $\partial_{RG}$ is finite-to-one, provided that $X$ is uniformly locally finite. In general, the map $\partial_{RG}$ may have an infinite fiber. Consider a starlike tree $S$, where infinitely many edges $\{E_i\}_{i \in I}$ are incident to the central vertex. Consider $(-\infty, \infty)$ as a $1$-dimensional cube complex with vertex set $\mathbb{Z}$. Then $X := S \times (-\infty, \infty)$ is a locally infinite Gromov hyperbolic CAT(0) cube complex. In this case, $\partial_R X = \Vertex(S) \times \{-\infty, \infty\}$, with the discrete topology on $\Vertex(S)$, and $\partial_G X = \{-\infty, \infty\}$, so that $\partial_{RG} \from \Vertex(S) \times \{-\infty, \infty\} \to \{-\infty, \infty\}$ is the projection map.
\end{rem}



\subsection{Oriented tree coordinate systems}\label{subsec:OrientedTreeCoordi}
\begin{defn}[Oriented tree coordinate system]
    For an oriented cube complex $X$, we say a tree coordinate system $\{\pi_i\from X\to X_i\}_{i=1,2,\dots,k}$ is {\it oriented} if each $X_i$ is oriented as a $1$-dimensional cube complex and $\pi_i$ is orientation preserving.
\end{defn}

The next lemma will be used when we show that an edge orientation defined by a measured foliation has mono-outgoing cubes.

\begin{lem}[Mono-outgoing cubes]\label{lem:OutgoingCube}
    Suppose that $X$ is an oriented {\rm CAT(0)} cube complex and $\{\pi_i\from X \to X_i\}_{i=1,2,\dots k}$ is an oriented tree coordinate system. Suppose that there exists an infinite directed edge-path $p\from [0,\infty)\to X$ such that
    \begin{enumerate}
        \item for any $i\in\{1,2,\dots,k\}$, $\pi_i \circ p$ is an unbounded infinite directed edge-path with layovers.
    \end{enumerate}
    Then, the property (1) holds for every infinite directed edge-path $p$.
    Moreover, in addition to the property (1), suppose that
    \begin{enumerate}[resume]
        \item each $X_i$ has mono-outgoing cubes, which are edges.
    \end{enumerate}
    Then, $X$ has mono-outgoing cubes.
\end{lem}
\begin{proof}
    Suppose (1) is satisfied for an infinite directed edge-path $p$. Consider an infinite directed edge-path $q\from [0,\infty) \to X$. By Theorem \ref{thm:Roller-to-Gromov}, there exists $C>0$ such that $d_{(X,L^1)}(p(t),q(t))<C$ for any $t>0$. Then, it follows from Equation \eqref{eqn:L1distance equality} that $\pi_i\circ q$ is unbounded for every $i\in \{1,2,\dots,k\}$.
    
    Next, let us suppose that $X$ also satisfies (2). Recall that since $X$ is a CAT(0) cube complex, it is a flag complex. Therefore, it suffices to show that for any vertex $v$, every pair of two outgoing edges $e_1$ and $e_2$ from $v$ lies in a common cube. 

    Suppose $e_1$ and $e_2$ are distinct outgoing edges from $v$ in $X$. We may reorder indices such that, for $i\in\{1,2\}$, $\pi_i$ projects $e_i$ onto an edge in $X_i$. Let $w$ be the endpoint of $e_1$ with $v\neq w$. Since an infinite directed edge-path exists, by Proposition \ref{prop:EdgeOri_WallOri}, $X$ is oriented toward a point in the Roller boundary. Then, by Lemma \ref{lem:Inf_directed_edgepaths} and the first part of this lemma, there exists an infinite directed edge-path $p$ starting from $w$ such that $\pi_1\circ p$ and $\pi_2 \circ p$ are unbounded.
    
    Denote by $H_1$ and $H_2$ the hyperplanes dual to $e_1$ and $e_2$, respectively. It suffices to show $H_1$ and $H_2$ intersect. Suppose that $H_1$ and $H_2$ are disjoint. By the assumption (1), $\pi_2\circ p$ is an unbounded directed edge-path in $X_2$. Thus, $\pi_2\circ p$ passes through $\pi_2(e_2)$, which is the unique outgoing edge from $\pi_2(v)$ in $X_2$. Hence, $p$ passes through $H_2$. Note that $v$, and thus $e_2$, is in different half-space of $H_1$ than $w$. It follows that if we denote by $\overrightarrow{H_1}$ the half-space containing $w$, then $H_2\subset \overleftarrow{H_1}$. Hence, $p$, which starts from $w$, must pass through $H_1$ before $H_2$, this implies that the directed edge-path $\pi_1\circ (e_1 * p)$ is backtracking, where $e_1*p$ is the concatenation of $e_1$ and $p$. This contradicts Lemma \ref{lem:CubicalMaptoTreeMonotone} and Proposition \ref{prop:DirectedPathGeodesic}.
\end{proof}

\section{Cube complexes of pre-homotopic curves/multicurves}\label{Sec:CubeCplxofCurves}

Suppose that $Z$ is a multicurve of $S_{g,n}$. In Section \ref{sec:CubeCplx}, we discussed the CAT(0) Gromov-hyperbolic cube complex $\mathscr{X}_Z$ constructed from the wall space $(\widetilde{S_{g,n}}, \mathcal{W}_Z)$. In this section, we construct a new cube complex $X_Z$ whose vertices are multicurves $\Gamma$ of $S_{g,n+1}$ that are homotopic to $Z$ in the puncture-filled surface $S_{g,n}$ with some additional properties; see Definition \ref{defn:CplxofPreHmpticMultiCurv}. Then, in Theorem \ref{thm:CubeCplxIso}, we show that there is an cubical isomorphism $\Omega\from X_Z \to \mathscr{X}_Z$.

Consider $S_{g,n+1}$ and its puncture-filled (or markedpoint-forgotten) surface $S_{g,n}$. Throughout, $x$ denotes the extra puncture we forget, by filling it. Let $i\from S_{g,n+1}\to S_{g,n}$ be the embedding whose images only misses $x$. It induces a map $i_*\from C(S_{g,n+1})\to C(S_{g,n})\cup\{o\}$, where $i_*([\gamma])=o$ for essential simple closed curves $\gamma$ in $S_{g,n+1}$ that are peripheral in $S_{g,n}$.

\begin{defn}[Pre-(adjective) curves]
Suppose that $\alpha$ and $\beta$ are essential simple closed curves in $S_{g,n+1}$. 
\begin{enumerate}
    \item $\alpha$ is {\em pre-peripheral} if $\alpha$ is peripheral in $S_{g,n}$.
    \item $\alpha$ is {\em pre-essential} if $\alpha$ is not pre-peripheral.
    \item $\alpha$ and $\beta$ are {\em pre-homotopic}, written as $\alpha \sim_{\pre} \beta$, if $i_*[\alpha]=i_*[\beta]$ in $S_{g,n}$.
\end{enumerate}

\end{defn}

\subsection{Curve complexes and trees of subsurface groups}

\subsubsection{Trees of groups for simple closed curves}\label{subsec:Trees of Groups}
We review some constructions of trees in \cite{KLS_TreesMCG}. The key idea is that the puncture $x$ to be forgotten from $S_{g,n+1}$ can be considered as the base point of the fundamental group $\pi_1(S_{g,n},x)$ of the $x$-filled surface $S_{g,n}$.

Let $CC(S_{g,n+1})$ denote the {\it curve complex} of $S_{g,n+1}$, whose vertex set is $C(S_{g,n+1})$ and $k$-simplices are defined for $(k+1)$-tuple of pairwise disjoint homotopy classes of simple closed curves \cite{MasurMinsky_CurveComplex1}. Since we assume $\chi(S_{g,n})<0$, the surface $S_{g,n+1}$ is not one of the exceptional cases $S_{1,1}$ or $S_{0,4}$, for which the adjacency relation in the curve complex is defined by intersection numbers $1$ and $2$, respectively, rather than by pairwise disjointness.

Suppose that $\gamma$ is a pre-essential simple closed curve in $S_{g,n+1}$ so that $i_*[\gamma]=[\zeta]$ for an essential simple closed curve $\zeta \subset S_{g,n}$.
Define $\mathscr{F}_{\zeta}$ as the subcomplex of $CC(S_{g,n+1})$ generated by $i_*^{-1}([\zeta])$. It is easy to show $\mathscr{F}_{\zeta}$ is one-dimensional. Assign $\pi_1(S_{g,n}\setminus \alpha,x)$ to every vertex $[\alpha]$ and $\pi_1(S_{g,n}\setminus (\alpha \cup \beta),x)$ to every edge between $[\alpha]$ and $[\beta]$, where $\alpha \cap \beta=\emptyset$. These are fundamental groups of the components of complements that contain $x$. In particular, the component of $S_{g,n}\setminus (\alpha \cup \beta)$ that contains $x$ is an annulus bounded by $\alpha$ and $\beta$. We denote the graph of groups by $(\mathscr{F}_{\zeta}, \pi_1(S_{g,n}\setminus -,x))$.

Let us define another graph of groups as follows. Let $\mathscr{D}_{\zeta}$ denote the decomposition of $\widetilde{S_{g,n}}$ by $p^{-1}(\zeta)$ where $p \from \widetilde{S_{g,n}} \to S_{g,n}$. That is, an (open) $2$-cell of $\mathscr{D}_{\zeta}$ means a connected component of  $\widetilde{S_{g,n}} \setminus p^{-1}(\zeta)$ and an (open) $1$-cell of $\mathscr{D}_{\zeta}$ means a component of $p^{-1}(\zeta)$. Define the dual graph $\mathscr{D}^*_{\zeta}$ in such a way that its vertices and edges correspond to $2$-cells and $1$-cells of $\mathscr{D}_{\zeta}$, respectively. To each vertex (resp.\@ edge) we assign the stabilizer of the dual $2$-cell (resp.\@ $1$-cell) of the $\pi_1(S_{g,n},x)$-action on $\widetilde{S_{g,n}}$ as deck transformations. We denote the graph of groups by $(\mathscr{D}^*_{\zeta}, \rm{Stab}(-))$. It is immediate from the construction that the underlying graph $\mathscr{D}^*_{\zeta}$ is a tree.

There is an isomorphism $(\mathscr{F}_{\zeta}, \pi_1(S_{g,n}\setminus -,x)) \to (\mathscr{D}^*_{\zeta}, \rm{Stab}(-))$ as graphs of groups sending $[\alpha]$ to $U^*$ where $U^*$ is the vertex dual to a $2$-cell $U$ so that $\pi_1(S_{g,n}\setminus \alpha,x)={\rm Stab}(U)$ \cite[Sections 3 and 7]{KLS_TreesMCG}. As a corollary, we see that $\mathscr{F}_{\zeta}$ is a tree.

The above summary is slightly different from how it is addressed in \cite{KLS_TreesMCG}, though they are essentially equivalent. The key step is to relate the fundamental group of subsurfaces $\pi_1(S_{g,n}\setminus-,x)$ to the stabilizers ${\rm Stab}(-)$, which is discussed in \cite[Section 3.2]{KLS_TreesMCG}

\subsection*{Pulls of simple closed curves along simple cords}

\begin{defn}[Simple cord]\label{defn:SimpleCord to Curve}
    Let $\gamma$ be a simple closed curve on $S_{g,n+1}$. A {\it simple cord} from $x$ to $\gamma$ is an embedding $\delta\from [0,1]\to S_{g,n}$ such that $\delta(0)=x$, $\delta(1)\in \gamma$, and $\delta([0,1)) \cap \gamma = \emptyset$. Here, we consider $S_{g,n}$ as a punctured surface, such that $\delta$ avoids the $n$ punctures.

    We say that two simple cords $\delta_0$ and $\delta_1$ are {\it homotopic relative to $x,\gamma$} if there exists a homotopy $(\delta_t)_{t\in[0,1]}$ in $S_{g,n}$ so that $\delta_t(0)=x$ and $\delta_t(1)\in \gamma$ for all $t\in [0,1]$.
\end{defn}

\begin{defn}[Simple pull]\label{defn:SimplePull of Curve}
    Let $\gamma$ be a simple closed curve on $S_{g,n+1}$ and $\delta$ be a simple cord from $x$ to $\gamma$. Then a {\it simple pull of $\gamma$ along $\delta$}, is a simple closed curve $\Pull_\delta(\gamma)$ that is obtained by pulling $\gamma$ along $\delta$. More precisely, let $U$ be a small neighborhood of $\delta$ so that $U \cap \gamma$ consists of only one connected component, which contains $\delta(1)$. See Figure \ref{fig:SimplePull}. The set $U$ and its boundary $\partial U$ are  decomposed by $\gamma$ into $\{U_1,\,U_2\}$ and $\{\partial_1 U,\,\partial_2 U\}$, where $\partial U_i =\partial_i U \cup (\gamma \cap U)$ for $i\in\{1,2\}$. Without loss of generality, suppose $x \in U_1$. Then $\Pull_\delta(\gamma)$ is defined from $\gamma$ by replacing $\gamma \cap U$ with $\partial_1 U$.

    \begin{figure}[h!]
	\centering
        \def\svgwidth{0.3\textwidth}
        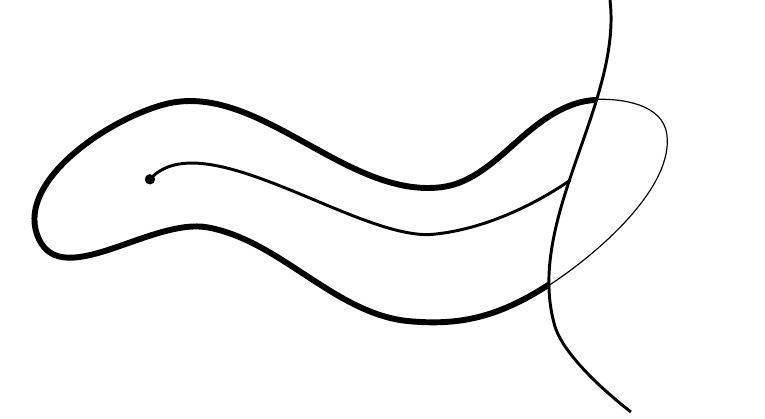
	\caption{Simple pull}
    \label{fig:SimplePull}
    \end{figure}

    Although the above construction of $\Pull_\delta(\gamma)$ depends on the choice of the neighborhood $U$, its homotopy class $[\Pull_\delta(\gamma)]$ is independent of $U$. Moreover, if two simple cords $\delta_1$ and $\delta_2$ are homotopic relative to $x,\gamma$, then  $[\Pull_{\delta_1}(\gamma)]=[\Pull_{\delta_2}(\gamma)]$.
    
    For homotopy classes, we say that $[\gamma_1]$ is a {\it simple pull} of $[\gamma_2]$ if there exist representatives $\gamma'_1$ and $\gamma'_2$ and a simple cord $\delta$ from $x$ to $\gamma'_2$ so that $\gamma'_1$ is a simple pull of $\gamma'_2$ along $\delta$. The existence of such a simple cord is independent of the choice of representatives of simple closed curves, which does not hold for multicurves; see Figure \ref{fig:DependenceRep}.
\end{defn}

\begin{prop}\label{prop:SPullCurvElemProp} Suppose that $\gamma_1,\gamma_2$ are pre-essential simple closed curves in $S_{g,n+1}$.
    \begin{enumerate}
        \item If $[\gamma_1]$ is a simple pull of $[\gamma_2]$, then they are pre-homotopic, i.e., $i_*[\gamma_1]=i_*[\gamma_2]$.
        \item If $[\gamma_2]$ is a simple pull of $[\gamma_1]$, then $[\gamma_1]$ is a simple pull of $[\gamma_2]$.
        \item Suppose $i_*[\gamma_2]=i_*[\gamma_1]=[\zeta]$ for a simple closed curve $\zeta$ on $S_{g,n}$. Then $[\gamma_1]$ and $[\gamma_2]$ are joined by an edge in $\mathscr{F}_{\zeta}$ if and only if $[\gamma_1]$ is a simple pull of $[\gamma_2]$.
    \end{enumerate}
\end{prop}
\begin{proof}
    The proofs are straightforward.
\end{proof}

\begin{defn}
    Suppose that $[\zeta]\in C(S_{g,n})$. Define a {\it complex of pre-homotopic curves} $X_{\zeta}$ as a graph such that
    \begin{enumerate}
        \item $\Vertex(X_{\zeta})=i_*^{-1}([\zeta])\subset C(S_{g,n+1})$, and
        \item two vertices $[\gamma_1]$ and $[\gamma_2]$ are joined by an edge iff $[\gamma_1]$ is a simple pull of $[\gamma_2]$.
    \end{enumerate}
\end{defn}

By Proposition \ref{prop:SPullCurvElemProp}, it is straightforward that $X_{\zeta}$ is isomorphic to $\mathscr{F}_{\zeta}$ as graphs. Hence, $X_\zeta$ is a tree. However, the construction using simple pulls can be extended to multicurves $Z$, assigning cube complexes $X_Z$.

\subsection{Complexes of pre-homotopic multicurves}

\begin{defn}[Pre-({\it adjective}) multicurves]
Suppose $\Gamma$ is a multicurve of $S_{g,n+1}$.
\begin{enumerate}
    \item $\Gamma$ is {\em pre-essential} if its every component is pre-essential.
    \item $\Gamma$ is {\em pre-non-parallel} if $i_*[\gamma_1]\neq i_*[\gamma_2]$ for any $\gamma_1 \neq \gamma_2\in \Gamma$.
    \item $\Gamma$ is {\em pre-minimally-self-intersecting} or, briefly, {\em pre-minimal} if
    \[
        \pair{[\Gamma]}{[\Gamma]}_{S_{g,n+1}}=\pair{i_*[\Gamma]}{i_*[\Gamma]}_{S_{g,n}}.
    \]
    \item A pre-non-parallel multicurve $\Gamma$ is {\em pre-filling} if $i_*[\Gamma]$ fills $S_{g,n}$.
    \item Two pre-essential and pre-non-parallel multicurves $\Gamma_1$ and $\Gamma_2$ are {\em pre-homotopic} if there is a bijection between components of $\Gamma_1$ and $\Gamma_2$ that sends each curve $\gamma_1 \in \Gamma_1$ to a pre-homotopic curve $\gamma_2 \in \Gamma_2$, i.e., $\gamma_1 \sim_{\pre} \gamma_2$. We write $\Gamma_1 \sim_{\pre} \Gamma_2$ if $\Gamma_1$ and $\Gamma_2$ are pre-homotopic.
\end{enumerate}
\end{defn}

\begin{defn}[Simple cords and simple pulls for multicurves]\label{defn:SimplePull_Multicurves}
    Let $\Gamma$ be a multicurve of $S_{g,n+1}$. Suppose that $\Gamma$ is in minimal position. We use Definitions \ref{defn:SimpleCord to Curve} and \ref{defn:SimplePull of Curve} with $\gamma$'s replaced by $\Gamma$'s to define {\it simple cords} from $x$ to $\Gamma$ and {\it simple pulls} of $\Gamma$ along simple cords.
    For a simple cord $\delta$ from $x$ to $\Gamma=\{\gamma_1,\gamma_2,\dots,\gamma_k\}$, suppose that $\delta(1) \in \gamma_1$. Then $\Pull_\delta(\Gamma)=\{\gamma_1',\gamma_2,\gamma_3,\dots,\gamma_k\}$ such that $\gamma_1'=\Pull_\delta(\gamma_1)$, where $\Pull_\delta(\gamma_1)$ is the simple pull of the simple closed curve $\gamma_1$ along $\delta$. We call $\gamma_1$ the {\it to-be-pulled component of $\Gamma$} and $\gamma_1 '$ the {\it pulled component of $\Pull_\delta(\Gamma)$} under the simple pull of $\Gamma$ along $\delta$.
\end{defn}

We use the term ``simple pull'' ambiguously to refer to both the operation and its result. For example, $\Pull_\delta(\Gamma)$ is a simple pull of $\Gamma$ obtained by the simple pull of $\Gamma$ along $\delta$.


\subsection*{Symmetry and non-symmetry of simple pulls for multicurves}
\begin{lem}\label{lem:Asymm of simplepull}
    Suppose that $[\Gamma']$ is a simple pull of $[\Gamma]$. Then, we have $\pair{[\Gamma']}{[\Gamma']}_{S_{g,n+1}}\le \pair{[\Gamma]}{[\Gamma]}_{S_{g,n+1}}$. The equality holds if $[\Gamma]$ is pre-minimal, i.e., the pre-minimality is preserved by simple pulls.
\end{lem}
\begin{proof}
    Suppose that a representative $\Gamma$ of $[\Gamma]$ is in minimal position such that $\Gamma'$ is a simple pull of $\Gamma$. Since simple pulls for representatives do not increase the intersection number, we have
    \[
        \pair{[\Gamma]}{[\Gamma]}_{S_{g,n+1}}=\pair{\Gamma}{\Gamma}_{S_{g,n+1}}=\pair{\Gamma'}{\Gamma'}_{S_{g,n+1}}\ge \pair{[\Gamma']}{[\Gamma']}_{S_{g,n+1}}.
    \]
    The inequality is strict if $\Gamma'$ is not in minimal position.
    If $[\Gamma]$ is pre-minimal, and $[\Gamma']$ is a simple pull of $[\Gamma']$, then
    \[
        \pair{i_*[\Gamma]}{i_*[\Gamma]}_{S_{g,n}}=\pair{[\Gamma]}{[\Gamma]}_{S_{g,n+1}} \ge \pair{[\Gamma']}{[\Gamma']}_{S_{g,n+1}}\ge \pair{i_*[\Gamma']}{i_*[\Gamma']}_{S_{g,n}}.
    \]
    Since $i_*[\Gamma]=i_*[\Gamma']$, the multicurve $[\Gamma']$ is also pre-minimal.
\end{proof}

Suppose that $[\Gamma']$ is a simple pull of $[\Gamma]$. If $[\Gamma]$ is pre-minimal and $\Gamma$ is in minimal position, then any simple pull $\Gamma'$ of $\Gamma$ is also in minimal position. Hence, $[\Gamma]$ is also a simple pull of $[\Gamma']$. 

If $[\Gamma]$ is not pre-minimal, the inequality in Lemma \ref{lem:Asymm of simplepull} may be strict, in which case $[\Gamma]$ is not a simple pull of $[\Gamma']$. A representative in $[\Gamma]$ may be obtained from a simple pull of a representative $\Gamma'$ of $[\Gamma']$ that is not in minimal position, but this is excluded from our definition of simple pulls between homotopy classes of multicurves.

\subsection*{(In)dependence of representatives of simple pulls}
If $[\alpha]$ is a simple pull of $[\gamma]$, then for any representative $\gamma_1$ of $[\gamma]$, there exists a simple cord $\delta$ from $x$ to $\gamma_1$ such that $[\Pull_\delta(\gamma_1)]=[\alpha]$. This may be false for multicurves. There may be two representatives of $\Gamma_1$ and $\Gamma_2$ of the same homotopy class $[\Gamma_1]=[\Gamma_2]$ such that another homotopy class of multicurve $[\Gamma]$ can be obtained a simple pull of $\Gamma_2$ but not of $\Gamma_1$ (Figure \ref{fig:DependenceRep}).

    \begin{figure}[h!]
	\centering
        \def\svgwidth{0.7\textwidth}
        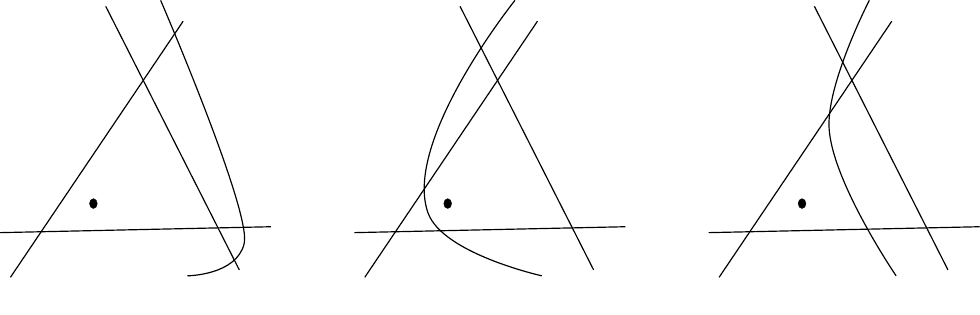
	\caption{Dependence of representatives of simple pulls}
    \label{fig:DependenceRep}
    \end{figure}





\subsection*{Commuting simple pulls and cube complexes of multicurves}
\begin{defn}[Commuting simple pulls]\label{defn:CommSimplePull}
    Let $\Gamma=\{\gamma_1,\gamma_2,\dots,\gamma_k\}$ be a multicurve of $S_{g,n+1}$. Fix $I\subset \{1,2,\dots,k\}$ with $|I|=m$. Suppose that $[\gamma_i']$ is a simple pull of $[\gamma_i]$ for each $i\in I$. For any $J\subset I$, we define a homotopy class of multicurve $[\Gamma_J]$ by
    \[
        [\Gamma_J]=\{[\gamma_i']~|~i\in J\}\cup \{[\gamma_i]~|~ i\in I\setminus J\}.
    \]
    We say that these $m$ simple pulls, changing $[\gamma_i]$'s into $[\gamma_i']$'s for $i\in I$, are an {\it $m$-tuple of commuting simple pulls} if $[\Gamma_{J\cup\{j\}}]$ is a simple pull of $[\Gamma_J]$, changing $[\gamma_{j}]$ to $[\gamma_{j}']$, for every $J\subsetneq I$ and $j\in I\setminus J$.  
\end{defn}

\begin{defn}[Complex of pre-homotopic multicurves]\label{defn:CplxofPreHmpticMultiCurv}
Let $Z$ be a multicurve of $S_{g,n}$. We define a cube complex $X_Z$ as follows:
\begin{enumerate}
    \item $\Vertex(X_Z)$ consists of the homotopy classes of multicurves $\Gamma$ of $S_{g,n+1}$ so that $\Gamma$ is pre-essential, pre-non-parallel, pre-minimal, and $i_*[\Gamma]=[Z]$, and
    \item $\Edge(X_Z)$ consists of simple pulls.
\end{enumerate}
The $1$-skeleton $X_Z^{(1)}$ has been defined. Then, 
\begin{enumerate}[resume]
    \item an $m$-cube for $m \ge 2$ is attached to each $1$-skeleton of an $m$-cube embedded in $X_Z^{(1)}$. By Lemma \ref{lem:cubes}, $m$-cubes correspond to $m$-tuples of commuting simple pulls.
\end{enumerate}
\end{defn}

\begin{lem}\label{lem:cubes}
    For a multicurve $Z$ of $S_{g,n}$, let $\Gamma=\{\gamma_1,\gamma_2,\dots,\gamma_k\}$ be a multicurve of $S_{g,n+1}$ with $[\Gamma] \in \Vertex(X_Z)$. Suppose that $[\Gamma]$ is a vertex of an $m$-cube $C_m$ in $X_Z$. Then, the following hold.
    \begin{enumerate}
        \item $m\le k$.
        \item The vertices of $C_m$ are obtained by an $m$-tuple of commuting simple pulls of $[\Gamma]$.
    \end{enumerate}
\end{lem}
\begin{proof}
    We first show (2) when $m=2$. Denote by $[\Gamma_1]$ and $[\Gamma_2]$ the vertices of $C_2$ adjacent to $[\Gamma]$, such that they are simple pulls of $[\Gamma]$. Denote by $[\Gamma_3]$ the other vertex of $C_2$.
    
    Suppose that $[\Gamma_1]$ and $[\Gamma_2]$ are obtained by simple pulls of different components of $[\Gamma]$. Without loss of generality, suppose that $[\gamma_i]$ of $[\Gamma]$ is simply pulled to $[\gamma_i']$ of $[\Gamma_i]$ for $i\in\{1,2\}$. Since $[\Gamma_3]$ is a common simple pull of $[\Gamma_1]$ and $[\Gamma_2]$, $[\Gamma_3]$ must be $\{[\gamma_1'],[\gamma_2'],[\gamma_3],\dots,[\gamma_m]\}$, which is the desired conclusion for $m=2$.
    
    Suppose that $[\Gamma_1]$ and $[\Gamma_2]$ are obtained by simple pulls of the same component $[\gamma_1]$ of $[\Gamma]$. Denote by $[\gamma_1']$ and $[\gamma_1'']$ the elements of $[\Gamma_1]$ and $[\Gamma_2]$ with $i_*[\gamma_1]=i_*[\gamma_1']=i_*[\gamma_1'']$. Since $[\Gamma_3]$ is a common simple pull of $[\Gamma_1]$ and $[\Gamma_2]$, there exists $[\gamma_1''']\in [\Gamma_3]$ that is a common simple pull of $[\gamma_1']$ and $[\gamma_1'']$. Then $[\gamma_1],[\gamma_1'],[\gamma_1''],$ and $[\gamma_1''']$ form a $4$-cycle in $X_{\zeta_1}$, which contradicts the fact that $X_{\zeta_1}$ is a tree.

    (1) Suppose $m>k$. There exist two vertices $[\Gamma_1]$ and $[\Gamma_2]$ that are adjacent to $[\Gamma]$ in $X_Z$ such that $[\Gamma_1]$ and $[\Gamma_2]$ are obtained by simple pulls of the same component of $[\Gamma]$. It contradicts (2) of the case $m=2$ because commuting simple pulls are obtained by simple pulls of different components.    

    (2) We use induction on $m$. The case $m=1$ is trivial and we have shown it when $m=2$. Suppose that the statement in (2) holds with $m$ replaced by $m-1$. Let $C_{m-1}$ be a $(m-1)$-cube such that $[\Gamma] \in  C_{m-1}\subset C_m$. By the induction hypothesis, after reordering indices, the vertices of $C_{m-1}$ can be obtained by simple pulls of $\gamma_1, \gamma_2, \dots \gamma_{m-1}$. For any subset $J$ of $\{1,2,\dots,m-1\}$, denote by $[\Gamma_J]$ the vertex in $C_{m-1}$ that is obtained by simple pulls of $[\gamma_j]\in [\Gamma]$ for all $j\in J$.
    
    Denote by $[\Gamma']$ a unique vertex in $C_m \setminus C_{m-1}$ that is adjacent to $[\Gamma]$. For each $j\in \{1,2,\dots,m-1\}$, consider the $2$-cube in $C_m$ containing $[\Gamma], [\Gamma'],[\Gamma_{j}]$. Applying (2) with $m=2$ to these $2$-cubes, we can show that $[\Gamma']$ is obtained by a simple pull of a component of $[\Gamma]$ that is not $[\gamma_j]$ with $1\le j\le m-1$. Without loss of generality, let $[\Gamma']$ be obtained from $[\Gamma]$ via a simple pull of $[\gamma_m]\in [\Gamma]$ to $[\gamma_k']\in [\Gamma']$. Define $[\Gamma_{\{m\}}]:=[\Gamma']$.

    \begin{figure}[h!]
	\centering
        \def\svgwidth{0.3\textwidth}
        \input{cube.pdf_tex}
	\caption{An $m$-cube $C_m$}
    \label{fig:cube}
    \end{figure}
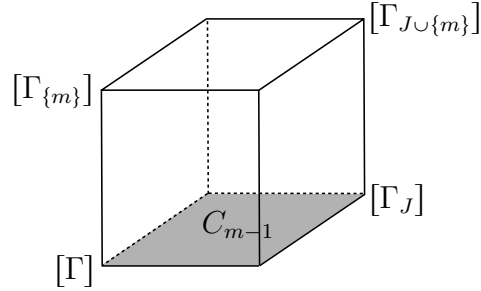
    
    Let us label vertices of $C_m$ as follows. We already labeled vertices in $C_{m-1}(\subset C_m)$ as $[\Gamma_J]$ for $J\subset \{1,2,\dots,m-1\}$. For any vertex $v\in C_m\setminus C_{m-1}$, there exists a unique $[\Gamma_J]\in C_{m-1}$ such that there is an edge between $v$ and $[\Gamma_J]$. Label $v$ as $[\Gamma_{J\cup \{m\}}]$. There is a bijection between vertices of $C_{m}$ and subsets of $\{1,2,\dots,m\}$. Recall that for any $J\subset \{1,2,\dots,m-1\}$, $[\Gamma_J]$ is obtained from $[\Gamma]$ via simple pulls of $\{[\gamma_j]\}_{j\in J}$ to $\{[\gamma_j']\}_{j'\in J}$; see Figure \ref{fig:cube}.

    Now, it suffices to show that for every $J\subset \{1,2,\dots,m-1\}$, (a) $[\Gamma_{J\cup \{m\}}]$ is a simple pull of $[\Gamma_J]$ changing $[\gamma_m]$ into $[\gamma_m']$ and (b) for any $j'\in \{1,2,\dots,m-1\}\setminus J$, $[\Gamma_{J \cup \{j',m\}}]$ is the simple pull of $[\Gamma_{J\cup \{m\}}]$ changing $[\gamma_{j'}]$ to $[\gamma_{j'}]$. We remark that (a) is about vertical edges and (b) is about horizontal edges in the upper $(m-1)$-cube in Figure \ref{fig:cube}.
    
    Let us first show (a). Note that (a) holds when $J=\emptyset$. Hence, it suffices to show that if the property is true for $J$, then it also holds for $J'$ with $J'=J \cup \{j'\}$ for any $j'\in \{1,2,\dots,m-1\}\setminus J$. Fix $J\subsetneq \{1,2,\dots,m-1\}$. Suppose that $[\Gamma_{J\cup\{m\}}]$ is the simple pull of $[\Gamma_J]$ changing $[\gamma_m]$ to $[\gamma_m']$. Fix $j'\in \{1,2,\dots,m-1\}\setminus J$, and define $J'=J\cup \{j'\}$. Applying (2) of $m=2$ for the $2$-cube containing $[\Gamma_J], [\Gamma_{J'}],[\Gamma_{J\cup\{m\}}]$, and $[\Gamma_{J'\cup \{m\}}]$, we have that $[\Gamma_{J'\cup \{m\}}]$ is the simple pull of $[\Gamma_{J'}]$ changing $[\gamma_m]$ to $[\gamma_m']$. It completes the proof of (a).

    Then, the proof of (b) also follows from the application of (2) when $m=2$, using the $2$-cube containing $[\Gamma_J], [\Gamma_{J'}],[\Gamma_{J\cup\{m\}}]$, and $[\Gamma_{J'\cup \{m\}}]$. This completes the proof of (2).
\end{proof}

\subsection*{Projection maps $p_i\from X_Z\to X_{\zeta_i}$}
Let $Z=\{\zeta_1,\zeta_2,\dots,\zeta_k\}$. For each $i\in\{1,2,\dots,k\}$, there exists a cubical map $p_i\from X_Z\to X_{\zeta_i}$ as follows: Suppose $[\Gamma]=\{[\gamma_1],[\gamma_2],\dots,[\gamma_k]\}$ is a vertex of $X_Z$ with $i_*[\gamma_i]=[\zeta_i]$ for each $i$. Define $p_i([\Gamma])$ as $[\gamma_i]$. Suppose $[\Gamma']$ is a vertex adjacent to $[\Gamma]$. Then there is a simple pull from $[\Gamma]$ to $[\Gamma']$, changing a component $[\gamma_j]$ to $[\gamma_j']$. They are mapped by $p_i$ to either adjacent vertices if $i=j$, or to the same vertices if $i\neq j$. Then $p_i$ extends to a cubical map $p_i\from X_Z\to X_{\zeta_i}$.

\begin{lem}\label{lem:p_i Surj, Efficent}
    The cubical map $p_i\from X_Z\to X_{\zeta_i}$ is surjective. Moreover, the family $\{p_i\from X_Z \to X_{\zeta_i}\}_{i=1,2,\dots,k}$ is efficient, in the sense of Definition \ref{defn:TreeCoordiate}.
\end{lem}
\begin{proof}
    We first show the surjectivity on vertices.
    Fix a simple closed curve $\gamma_i'$ on $S_{g,n+1}$ with $\gamma_i\sim_{\pre}\gamma_i'$. Define a multicurve $\Gamma'$ of $S_{g,n+1}$ by replacing $\gamma_i$ in $\Gamma$ with $\gamma_i'$. If $\Gamma'$ is in minimal position on $S_{g,n}$, then $[\Gamma']\in \Vertex(X_Z)$ and $p_i([\Gamma'])=[\gamma_i']$. Suppose $\Gamma'$ is not in minimal position on $S_{g,n}$. By using a standard argument of removing bigons, we can obtain $\Gamma''$ that is in minimal position and homotopic to $\Gamma'$ in $S_{g,n}$. When removing bigons, we can fix $\gamma_i'$, changing other components $\gamma_j$ with $j\neq i$ to $\gamma_j''$. We may assume that $\gamma_j''$ does not pass through $x$, though the homotopy from $\gamma_j$ to $\gamma_j''$ can pass over $x$. Then, we can consider $\Gamma''$ as a multicurve on $S_{g,n+1}$ that is pre-minimal and pre-homotopic to $\Gamma$ such that $p_i([\Gamma''])=[\gamma_i']$.

    The surjectivity on edges follows similarly. Consider two adjacent vertices $[\gamma_i]$ and $[\gamma_i']$ of $X_{\zeta_i}$. We can assume that $\gamma_i$ and $\gamma_i'$ bound a very thin annulus $A$ punctured at $x$. By applying the argument in the previous paragraph, we can find $\gamma_j$'s with $j\neq i$ so that $\Gamma=\{\gamma_j : j\neq i\} \cup \{\gamma_i\}$ and $\Gamma'=\{\gamma_j : j\neq i\} \cup \{\gamma_i'\}$ are pre-minimal. It follows that the connected component of $A\setminus \bigcup_{j\neq i} \gamma_j$ containing $x$ intersects both $\gamma_i$ and $\gamma_i'$. Hence, one can obtain $[\Gamma']$ as a simple pull of $[\Gamma]$. 

    The efficiency is straightforward by definition.
\end{proof}

Now, we have a family of surjective cubical maps onto trees, $\{p_i\from X_Z \to X_{\zeta_i}\}$. We will eventually show that it is a tree coordinate system of $X_Z$, as a corollary of Theorem \ref{thm:CubeCplxIso}. Note that, at this point, we have not even shown that $X_Z$ is connected.

\subsection{Isomorphisms between two cube complexes}
Suppose that $Z$ is a multicurve of $S_{g,n}$. In previous sections, we have constructed two cube complexes: the dual cube complex $\mathscr{X}_Z$ of the wall space $(\widetilde{S_{g,n}},\mathcal{W}_Z)$ and the cube complex $X_Z$ of pre-homotopic multicurves. In this section, we show that there is a naturally defined isomorphism between them.

\subsection*{Cubical map $\Omega\from X_Z \to \mathscr{X}_Z$} Consider $S_{g,n+1}$ as $S_{g,n}$ with a base point $x$. Fix a lift $\widetilde{x}\in \widetilde{S_{g,n}}$ of $x$. Suppose $[\Gamma] \in \Vertex(X_Z)$ such that $\Gamma$ is in minimal position on $S_{g,n+1}$.

Since $\Gamma$ is pre-minimal, $\Gamma$ is also in minimal position on $S_{g,n}$. Then we can construct a wall space $(\widetilde{S_{g.n}},\mathcal{W}_\Gamma)$ and its dual cube complex $\mathscr{X}_\Gamma$. Since $i_*[\Gamma]=[Z]$ on $S_{g,n}$, there is a natural bijection between the sets of walls $\mathcal{W}_\Gamma$ and $\mathcal{W}_Z$ under which walls with the same endpoints on the boundary $\partial \widetilde{S_{g,n}}$ are matched with each other. This yields a cubical isomorphism between $\mathscr{X}_\Gamma$ and $\mathscr{X}_Z$. 

Recall that an orientation of a wall $w\in \mathcal{W}$ is defined as a choice of one out of its two half-spaces $\overleftarrow{w},\overrightarrow{w}$. For each wall, we choose its half-space that contains $\tilde{x}$. This defines a vertex in $\mathscr{X}_\Gamma$. Since $\mathscr{X}_\Gamma$ is isomorphic to $\mathscr{X}_Z$, we obtain a vertex in $\mathscr{X}_Z$, which we denote by $\Omega([\Gamma])$. We emphasize that $\Omega([\Gamma])$ relies on the choice of the lift $\tilde{x}$ of $x$.

It is clear that if $[\Gamma']$ is a simple pull of $[\Gamma]$, then the orientations of all walls but one for $\Omega([\Gamma])$ and $\Omega([\Gamma'])$ coincide, and thus they are adjacent in $\mathscr{X}_Z$. Hence, $\Omega$ can be continuously extended to the $1$-skeletons, $\Omega^{(1)}\from X_Z^{(1)}\to \mathscr{X}_Z^{(1)}$.

By applying the above construction for each $\zeta_i \in Z$, we also obtain a cubical map $\Omega_i \from X_{\zeta_i}\to \mathscr{X}_{\zeta_i}$ between trees $X_{\zeta_i}$ and $\mathscr{X}_{\zeta_i}$. Note that as graphs, $X_{\zeta_i}$ is isomorphic to $\mathscr{F}_{\zeta_i}$ and $\mathscr{X}_{\zeta_i}$ is isomorphic to $\mathscr{D}^*_{\zeta_i}$, where $\mathscr{F}_{\zeta_i}$ and $\mathscr{D}^*_{\zeta_i}$ are trees defined in Section \ref{subsec:Trees of Groups}. Under the identification  $X_{\zeta_i}=\mathscr{F}_{\zeta_i}$ and $\mathscr{X}_{\zeta_i}=\mathscr{D}^*_{\zeta_i}$,  the map $\Omega_i$ coincides with the isomorphism between $\mathscr{F}_{\zeta_i}\to \mathscr{D}^*_{\zeta_i}$. Hence $\Omega_i \from X_{\zeta_i}\to \mathscr{X}_{\zeta_i}$ is an isomorphism of trees. The commutativity of the diagram \eqref{eqn:CommDiagTreeSystem} for $1$-skeletons is straightforward.

Recall that $\{\mathscr{p}_i\from \mathscr{X}_Z\to \mathscr{X}_{\zeta_i}\}_{i=1,2,\dots,k}$ is a tree coordinate system of $\mathscr{X}_Z$; see Definition \ref{defn:TreeCrdSysXscr}. 

\begin{thm}\label{thm:CubeCplxIso}
     The map $\Omega^{(1)}\from X_Z^{(1)} \to \mathscr{X}^{(1)}_Z$ extends to an isomorphism of cube complexes $\Omega\from X_Z \to \mathscr{X}_Z$ such that the following diagram commutes.
    \begin{equation}\label{eqn:CommDiagTreeSystem}
        \begin{tikzcd}
			X_Z \arrow{r}{\Omega} \arrow{d}{p_i}
			& \mathscr{X}_Z  \arrow{d}{\mathscr{p}_i}\\
			X_{\zeta_i} \arrow{r}{\Omega_i}
			& \mathscr{X}_{\zeta_i}
		\end{tikzcd}
    \end{equation}
\end{thm}

\begin{proof}[Proof of Theorem \ref{thm:CubeCplxIso}]
    Since the $1$-skeleton determines the higher-dimensional cubes of $X_Z$ and $\mathscr{X}_Z$, it suffices to show that $\Omega^{(1)}\colon X_Z^{(1)} \to \mathscr{X}_Z^{(1)}$ is a graph isomorphism. This follows from Lemmas \ref{lem:Inj.LocSurj_Imply_Bij}, \ref{lem:Inj.Omega}, and \ref{lem:LocSurf.Omega}, proved below.
\end{proof}

We say that a graph homomorphism $f\from G \to H$ is {\it locally surjective at $v\in \Vertex(G)$} if each edge in $H$ incident to $f(v)$ is the image under $f$ of an edge in $G$ incident to $G$. We say that $f$ is {\it locally surjective} if it is locally surjective at  each vertex of $G$. 

\begin{lem}\label{lem:Inj.LocSurj_Imply_Bij}
    Suppose that $f\from G \to H$ is a graph homomorphism such that $H$ is connected. If $f$ is injective and locally surjective, then $f$ is a graph isomorphism.
\end{lem}
\begin{proof}
    It suffices to show that $f$ is surjective.
    Suppose $f$ is not surjective, such that there exists a vertex $w$ in $H \setminus f(G)$. Since $H$ is connected, there exists an edge-path $e_1,e_2,\dots,e_k$ in $H$ starting from $w$ to $w'\in f(\Vertex(G))$ such that $e_i$'s are not in $f(G)$. Then, $f$ is not locally surjective at $v':= f^{-1}(w')$. 
\end{proof}

\begin{lem}\label{lem:Inj.Omega}
    $\Omega\from X_Z\to \mathscr{X}_Z$ is injective.
\end{lem}
\begin{proof}

Consider product maps $\times_{i=1}^k p_i\from X_Z \to \prod_{i=1}^k X_{\zeta_i}$ and $\times_{i=1}^k \mathscr{p}_i\from \mathscr{X}_Z \to \prod_{i=1}^k \mathscr{X}_{\zeta_i}$ defined by Equation \eqref{eqn:ProdMap}. Denote by $\times_{i=1}^k\Omega_i\from \prod_{i=1}^k X_{\zeta_i} \to \prod_{i=1}^k \mathscr{X}_{\zeta_i}$ the isomorphism of cube complexes obtained from the product of the isomorphisms $\Omega_i$'s. Then, we obtain the following commuting diagram.
\begin{equation*}
    \begin{tikzcd}
		X_Z \arrow{r}{\Omega} \arrow{d}{\times_{i=1}^k p_i}
		& \mathscr{X}_Z  \arrow{d}{\times_{i=1}^k \mathscr{p}_i}\\
		\prod_{i=1}^k X_{\zeta_i} \arrow{r}{\times_{i=1}^k\Omega_i}[swap]{\cong}
		& \prod_{i=1}^k \mathscr{X}_{\zeta_i}
	\end{tikzcd}
\end{equation*}
To show that $\Omega$ is injective, it suffices to show that $\times_{i=1}^k p_i$ is injective.

The injectivity of $\times_{i=1}^k p_i$ on vertices is straightforward; if $\times_{i=1}^k p_i([\Gamma])=\times_{i=1}^k p_i([\Gamma'])$, then the $i^{th}$-components of $[\Gamma]$ and $[\Gamma']$ coincide for each $i\in\{1,2,\dots,k\}$. It follows from Lemmas \ref{lem:ProdMapTrees} and \ref{lem:p_i Surj, Efficent} that $\times_{i=1}^k p_i$ is a cubical embedding.
\end{proof}

\begin{lem}\label{lem:LocSurf.Omega}
    $\Omega^{(1)}\from X_Z^{(1)}\to \mathscr{X}_Z^{(1)}$ is locally surjective.
\end{lem}
\begin{proof}
    Fix $[\Gamma]\in \Vertex(X_Z)$. Recall that we have $\mathscr{X}_Z\cong \mathscr{X}_\Gamma$, by considering $\Gamma$ as a representative of $[Z]$ in $S_{g,n}$. In this proof, we consider $\mathscr{X}_\Gamma$. Edges of $\mathscr{X}_\Gamma$ incident to $\Omega([\Gamma])$ are in a bijective correspondence with walls $w\in \mathcal{W}_\Gamma$ for which there exists no other wall $w'$ disjoint from $w$ that separates $w$ from $\tilde{x}$. We say that such walls $w$ are {\it incident} to $\tilde{x}$. If $w$ intersects the closure of the connected component of $\widetilde{S_{g,n}}\setminus p^{-1}(\Gamma)$ containing $\tilde{x}$, we say that $w$ is {\it exposed to $\tilde{x}$ with respect to $\mathcal{W}_\Gamma$}.
    
    If $w$ is exposed to $\tilde{x}$, then we can find a simple pull $[\Gamma']$ of $[\Gamma]$ such that $\Omega([\Gamma'])$ can be obtained from $\Omega([\Gamma])$ by flipping the orientation of the wall $w$. Therefore, it suffices to show that if $w$ is incident to $\tilde{x}$, then we can deform $\Gamma$ in $S_{g,n}$ so that $w$ will be exposed to $\tilde{x}$. Without loss of generality, suppose that $w\subset p^{-1}(\gamma_1)$.
    
    In (Step 1), we form an arc $\delta$ from $x$ to $\gamma_1$ with ${\rm int}(\delta)\cap \gamma_1=\emptyset$ and an annulus $A$ with $A\supset \delta$ and $\partial A=\gamma_1\cup \gamma_1'$, where $[\gamma_1']=\Pull_\delta[\gamma_1]$. Then, we  prove that each component $\Gamma \cap A$ is an arc joining the two different boundary components of $A$. In (Step 2), we describe how to deform $\Gamma$ within $A$ to be disjoint from the cord.

    {\bf (Step 1)} Let $\gamma_1$ be the component of $\Gamma$ such that $w\subset p^{-1}(\gamma_1)$. Let $\delta$ be a path from $x$ to $\gamma_1$ such that $\tilde{\delta}$, the lift of $\delta$ starting from $\tilde{x}$, ends in $w$. By Lemma \ref{lem:Path-Arc}, we may choose $\delta$ as an arc. In addition, we may assume that $\delta$ meets $\gamma_1$ only at $\partial \delta$; otherwise, $\delta$ and $\gamma_1$ would form a bigon, which can be removed inductively by an isotopy of $\delta$ rel $\partial \delta$.
    
    Choose a simple closed curve $\gamma_1'$ such that there is an annulus $A$ in $S_{g,n}$ with $\partial A=\gamma_1\cup \gamma_1'$ and $\delta \subset A$. Then $[\gamma_1']=\Pull_\delta([\gamma_1])$. As curves in $S_{g,n+1}$, $\gamma_1$ and $\gamma_1'$ are pre-homotopic. Given $\Gamma$, we can add $\gamma_1'$ without changing $\Gamma$ in its homotopy class such that $\Gamma \cup \{\gamma_1'\}$ is in minimal position; first choose any representative $\gamma_1'$ and then remove bigons with $\Gamma$ by deforming $\gamma_1'$ only.

    For each $\gamma_i\in \Gamma$ with $i\neq 1$, if $A \cap \gamma_i\neq \emptyset$, it consists of several arcs. We claim that each arc $\alpha$ joins a point in $\gamma_1$ to a point in $\gamma_1'$. There are two other possibilities: $\alpha$ joins two points in $\gamma_1$ or two points in $\gamma_1'$.
    
    If $\alpha$ joins two points in $\gamma_1$, then the two intersection points may be removed by a homotopy of $\gamma_i$ within $S_{g,n}$. Then it violates the assumption that $\Gamma$ is pre-minimal.
    
    \begin{figure}[h!]
	\centering
        \def\svgwidth{0.7\textwidth}
        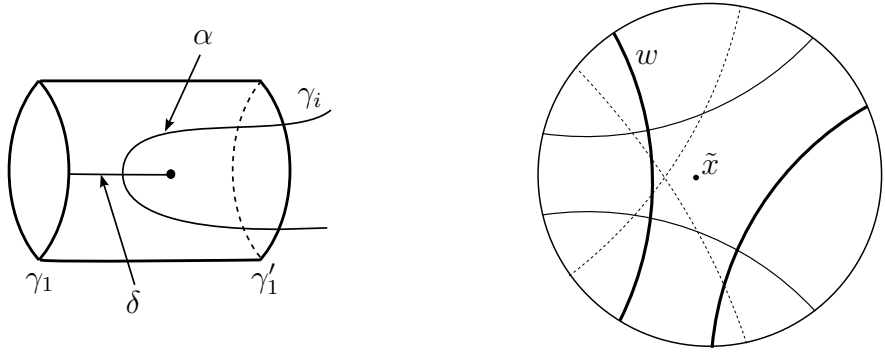
	\caption{In the right figure, the liftings of $\gamma_1$ are shown in bold, those of $\gamma_2$ (containing $\alpha$) in solid lines, and those of the remaining $\gamma_i$ in dotted lines. }
    \label{fig:AnnulusArc}
    \end{figure}
    
    Next, suppose that $\alpha$ joins two points on $\gamma_1'$. Without loss of generality, we may assume that $\alpha \subset \gamma_2$. Then 
    \[
        \pair{[\gamma_1']}{[\gamma_2]}_{S_{g,n+1}} > \pair{[\gamma_1]}{[\gamma_2]}_{S_{g,n+1}}.
    \]
    Let us focus on the walls $\mathcal{W}_{\{\gamma_1,\gamma_2\}} \subset \mathcal{W}_\Gamma$ that are liftings of $\gamma_1$ and $\gamma_2$.
    Since $\gamma_1$ and $\gamma_2$ are simple closed curves in minimal position and no other walls separate $w$ from $\tilde{x}$, the wall $w$ is exposed to $\tilde{x}$ with respect to $\mathcal{W}_{\{\gamma_1,\gamma_2\}}$. (We emphasize that this argument applies only when we restrict to two curves.) This implies that we can find a simple cord disjoint from $\gamma_2$ that is homotopic to $\delta$. Then we obtain
    \[
        \pair{[\gamma_1']}{[\gamma_2]}_{S_{g,n+1}} = \pair{[\gamma_1]}{[\gamma_2]}_{S_{g,n+1}},
    \]
    which is a contradiction. Hence, there exists no arc $\alpha$ joining two points on $\gamma_1'$.


    {\bf (Step 2)} Consider an annulus $A$ with a marked point $x \in {\rm int}(A)$. Let $\partial_1$ and $\partial_2$ denote the two boundary circles of $A$. Let $X$ be a set of arcs each of which joins one point in $\partial_1$ to a point in $\partial_2$ such that arcs in $X$ are in minimal position, i.e., no two of them form a bigon. Let $U$ be the component of $A \setminus X$ containing $x$. We show that there is a simple closed curve $\gamma$ homotopic to $\partial _1$ in $A$ relative to $x$ such that $\gamma$ intersects $U$ and $|\partial_1 \cap X|=|\gamma \cap X|=|X|$.
    \begin{figure}[h!]
	\centering
        \def\svgwidth{0.7\textwidth}
        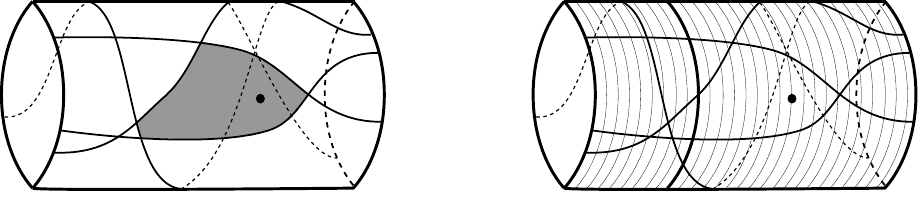
	\caption{Choice of $\gamma$ with $\gamma \sim \partial_1$. The solid dot denotes $x$, and the shaded region denotes $U$.}
    \label{fig:Annulus}
    \end{figure}
    
    See Figure \ref{fig:Annulus}. Let $\Fcal$ be a foliation on $A$ each leaf of which is homotopic to the core circle of $A$. First, we can assume that arcs in $X$ are transverse to $\Fcal$. More precisely, each arc of $X$ can be deformed in the homotopy class in $A$ relative to $\partial A\cup \{x\}$ such that it is transverse to $\Fcal$. Then we can remove bigons between arcs by taking homotopy along leaves of $\Fcal$. Then we choose $\gamma$ as a leaf of $\Fcal$ that is homotopic to $\partial_1$ in $A$ rel $\{x\}$ and passes through $U$.

    By (Step 1) and (Step 2), given $\Gamma$ and a wall $w$ incident to $\tilde{x}$, we can change $\Gamma$ within its homotopy class on $S_{g,n+1}$ for which the  wall $w$ is exposed to $\tilde{x}$.
\end{proof}

\begin{lem}[Path-to-arc transformation]\label{lem:Path-Arc}
    Suppose that $S$ is a surface and $\alpha\from [0,1] \to S$ is a path, with distinct endpoints. Then, there exists a homotopy $\left(\alpha_t\from [0,1]\to S\right)_{t\in[0,1]}$ rel $\{0,1\}$ such that $\alpha_0=\alpha$, $\alpha_1$ is an arc. Moreover, the homotopy $(\alpha_t)$ can be chosen to lie within an arbitrary small neighborhood of the path $\alpha$.
\end{lem}
We remark that the homotopy $(\alpha_t)$ passes over the initial point $\alpha(0)$, i.e., $\alpha_{t_0}(s_0)$ may be equal to $\alpha(0)$ for some $t_0,s_0\in(0,1)$.
\begin{proof}[Proof of Lemma \ref{lem:Path-Arc}]
    First, we may assume that $\alpha$ is in general position, such that it has only finitely many self-intersection points and at every self-intersection point only two segments of $\alpha$ intersect transversely.
    
    \begin{figure}[h!]
	\centering
        \def\svgwidth{0.8\textwidth}
        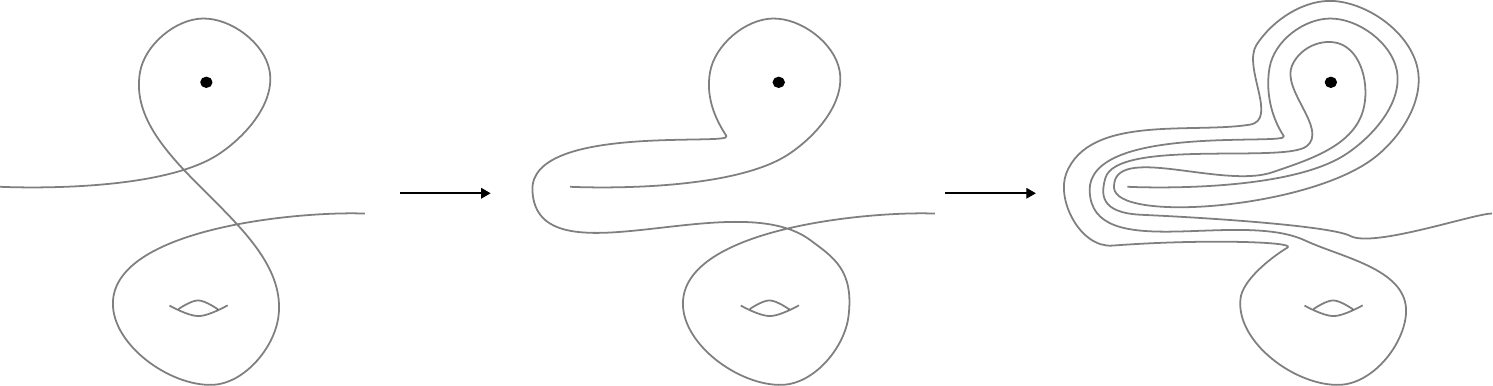
	\caption{Path-to-arc transformation}
    \label{fig:PathtoArc}
    \end{figure}
    
    See Figure \ref{fig:PathtoArc}. Suppose $0<t_0$ is the minimal number for which $\alpha(t_0)$ is a self-intersection point. Let $t_1$ be the number with $t_0<t_1$ and $\alpha(t_0)=\alpha(t_1)$. We perform the simple pull operation by considering $\alpha|_{[0,t_0]}$ as $\delta$ and $\alpha_{[t_1-\epsilon,t_1+\epsilon]}$ as the segment of $\gamma$ in Figure \ref{fig:SimplePull}, where $\epsilon>0$ is a small number. Note the simple pull can be defined as a homotopy within $U_1$ in Figure \ref{fig:SimplePull}. As a result, the first self-intersection point can be removed by applying the homotopy. By iterating this process, we ultimately obtain the desired homotopy $\alpha_t$.
\end{proof}

\section{Canonical simple pulls}\label{Sec:SimplePulls}

Recall that in Equation \eqref{eqn:PForget Functional}, for $\Fcal\in \MF(S_{g,n+1})$, we define a functional $\varphi(\Fcal)\from C(S_{g,n})\to \Rbb$ by
\[
    \varphi(\Fcal)([\zeta])=\inf_{i_*[\gamma]=[\zeta]} \pair{\Fcal}{[\gamma]}.
\]
In Section \ref{subsec:Curve_Reduction}, given $[\zeta]\in C(S_{g,n})$, we find an infinite edge-path $([\gamma_l])_{l\in \mathbb{N}}$ in $X_\zeta$ such that $\pair{\Fcal}{[\gamma_l]}\to \varphi(\Fcal)([\zeta])$ as $l\to \infty$. Then, in Section \ref{subsec:Multicurve_Reduction}, given a multicurve $Z$ of $S_{g,n}$, we find an edge-path $([\Gamma_l])_{l\ge0}$ in $X_Z$ such that $\pair{\Fcal}{[\Gamma_l]}\to \varphi(\Fcal)([Z])$ as $l\to \infty$. These constructions will be used in Section \ref{sec:ForgetMapMF} to show the existence of $\Phi(\Fcal)\in \MF(S_{g,n})$ with $\varphi(\Fcal)([\zeta])=\pair{\Phi(\Fcal)}{[\zeta]}$ for any $[\zeta]\in C(S_{g,n})$, which is a central part of Theorem \ref{theorem:MF}.

\begin{defn}[$x$-generic measured foliations]\label{defn:genericMF}
    Consider $S_{g,n+1}$ and $S_{g,n}$ as the surfaces of genus $g$ with $n+1$ and $n$ marked points, respectively. Let $\Fcal \in \MF(S_{g,n+1})$ and $x$ be the marked point of $S_{g,n+1}$ to be forgotten in $S_{g,n}$. Let $L_x$ be the leaf containing $x$. We say that $\Fcal$ is {\em $x$-generic} if $L_x$ is a half-infinite curve having $x$ as the endpoint.
\end{defn}

\subsection{Case 1: Single curves}\label{subsec:Curve_Reduction}

Recall that in a graph or a cube complex, we allow a sequence of consecutive vertices along an edge-path to represent the edge-path itself. Also, recall that two vertices $[\gamma_1]$ and $[\gamma_2]$ of $X_\zeta$ are connected by an edge if and only if $[\gamma_2]$ is a simple pull of $[\gamma_1]$ (and vice versa).

\begin{thm}\label{thm:CurveReduction}
    Suppose that $\zeta$ is an essential simple closed curve in $S_{g,n}$. Let $X_\zeta$ be the complex of pre-homotopic curves of $\zeta$. Suppose that $\Fcal$ is an $x$-generic minimal measured foliation on $S_{g,n+1}$.
    \begin{enumerate}
        \item Every simple pull strictly increases or decreases the intersection number with $\Fcal$.
        \item For any $[\gamma] \in \Vertex(X_\zeta)$, there exists a unique simple pull, called the canonical simple pull (Definition \ref{defn:CanonJoinPull}), of $[\gamma]$ that decreases the intersection number with $\Fcal$.
    \end{enumerate}
    Let us orient every edge of $X_\zeta$ toward the direction that the intersection number with $\Fcal$ decreases.
    \begin{enumerate}[resume]
        \item Let $(\gamma(l))_{l\ge0}$ be an infinite directed edge-path of $X_\zeta$. Considering it as a sequence in $\Vertex(CC(S_{g,n+1}))$, we have
        \[
            \gamma(l) \to \underline{\Fcal} \in \partial_\infty CC(S_{g,n+1})=\mathcal{EL}(S_{g,n+1})
        \]
        as $l\to \infty$, where $\underline{\Fcal}$ denotes the underlying minimal foliation of $\Fcal$ with the transverse measure forgotten.
        \item For any infinite directed edge-path $([\gamma(l)])_{l\ge0}$ of $X_\zeta$, we have
        \[
            \pair{\Fcal}{[\gamma(l)]} \to \varphi(\Fcal)([\zeta])
        \]
        as $l\to \infty$.
        \item (Convexity) For any triple $[\gamma_1], [\gamma_2], [\gamma_3]$ of consecutive vertices on a directed edge-path in $X_\zeta$, we have
        \[
            0<\pair{\Fcal}{[\gamma_2]}-\pair{\Fcal}{[\gamma_3]} \le \pair{\Fcal}{[\gamma_1]}-\pair{\Fcal}{[\gamma_2]}.
        \]
        \item For any infinite reversely-directed edge-path $([\gamma(l)])_{l=0,-1,-2,\dots}$, we have
        \[
            \pair{\Fcal}{[\gamma(l)]}\to 
            \infty
        \]
        as $l\to -\infty$.
        \end{enumerate}
        Recall that $X_\zeta$ is isomorphic to $\mathscr{D}^*_{\zeta}$ as trees. For $[\beta]\in C(S_{g,n+1})$ with $i_*[\beta]=[\zeta]$, let $D(\beta)$ denote the closure of the $2$-cell of $\mathscr{D}_{\zeta}$ that is dual to $[\beta] \in \Vertex(\mathscr{D}^*_{\zeta})$, where the closure is taken in the closed disk $\widetilde{S_{g,n}} \cup \partial_\infty \widetilde{S_{g,n}}$.
        \begin{enumerate}[resume]        
        \item There exists $\theta_\Fcal \in \partial_\infty \widetilde{S_{g,n}}$ such that for every infinite directed edge-path $(\gamma(l))_{l\ge0}$ of $X_\zeta$, we have
        \[
            D(\gamma(l)) \to \theta_\Fcal
        \]
        as $l\to \infty$, with respect to the Hausdorff topology on $\widetilde{S_{g,n}} \cup \partial_\infty \widetilde{S_{g,n}}$.
        \item The boundary point $\theta_\Fcal$ is independent of the choice of the essential curve $\zeta$.
    \end{enumerate}
\end{thm}

\begin{defn}[Canonical simple path and pulls]\label{defn:CanonJoinPull}
    Fix an $x$-generic measured foliation $\Fcal\in \MF(S_{g,n+1})$. Suppose that $\gamma$ is a pre-essential simple closed curve on $S_{g,n+1}$. Suppose that $\gamma$ and $\Fcal$ are transverse, $\Fcal \pitchfork \gamma$. We further assume that $\Fcal$ is represented in such a way that every singular point has degree $1$ or $\ge 3$. Consider the leaf $L_x$ containing $x$ as a ray starting from $x$. Let $\delta \subset L_x$ denote the segment $[x,y]$ between $x$ and the first intersection point $y$ of $L_x$ with $\gamma$. We call $\delta$ the {\em canonical simple cord} from $x$ to $\gamma$. The simple pull $\Pull_\delta(\gamma)$ along the canonical simple path is called the {\em canonical simple pull}. The homotopy class $[\Pull_\delta(\gamma)]$ is independent of the choice of representatives of $\gamma$ and $\Fcal$ (Lemma \ref{lem:IntNumbCanonPull}).
\end{defn}

\begin{defn}[Pulley and forgetting segment]\label{defn:PulleyForgetSegments}
Let $\Fcal, \gamma$ and $\delta=[x,y]$ be as defined in Definition \ref{defn:CanonJoinPull}. Let us consider a small closed neighborhood of $\delta$ that is foliated by parallel segments of leaves that do not contain any marked points and whose both endpoints are on $\gamma$. Being parallel means they are homotopic to each other with their endpoints staying within $\gamma$. The set, indexed by $\mathcal{A}$, of such closed neighborhoods $\{W_a\}_{a\in \mathcal{A}}$ is totally ordered by the set-inclusion. For each $a\in \mathcal{A}$, define $I_a:=\gamma \cap W_a$, which also follows the total order on $\mathcal{A}$. Denote by $W_{\Fcal,\gamma}$ and $I_{\Fcal,\gamma}$ the closures of the unions of $W_a$'s and $I_a$'s over $a\in \mathcal{A}$, respectively. By definition, the boundary of $W_{\Fcal,\gamma}$ must contain a singular point of $\Fcal$. We call $W_{\Fcal,\gamma}$ and $I_{\Fcal,\gamma}$ the {\em pulley} and the {\em forgetting segment} of (the canonical simple pull of) $\gamma$, respectively. We call the number $\pair{\Fcal}{I_{\Fcal,\gamma}}$ the {\em width} of the pulley. There are two possible cases; also see Figure \ref{fig:ForgetDomain}.

\begin{figure}[h!]
	\centering
    \hspace{10pt}
    \begin{subfigure}[b]{0.25\textwidth}
    \centering
    \def\svgwidth{\textwidth}
    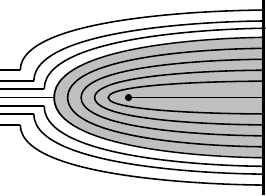
    \caption*{Case A}
    \end{subfigure}
    \hspace{40pt}
    \begin{subfigure}[b]{0.25\textwidth}
    \centering
    \def\svgwidth{\textwidth} 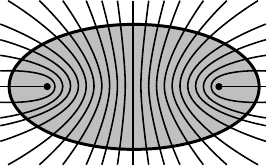
    \caption*{Case B}
    \end{subfigure}
	\caption{Pulleys are shaded.}
	\label{fig:ForgetDomain}
\end{figure}

\noindent {\it Case A: $I_{\Fcal,\gamma}$ is a closed interval without marked endpoints.}

Let $J$ be the closure of $\partial W_{\Fcal,\gamma} \setminus I_{\Fcal,\gamma}$. Then $J$ contains a singular point with valence $\ge 2$. For technical reasons, we choose a representative of $\Fcal$ such that $J$ contains exactly one singular point, whose valence is $3$; see Figure \ref{fig:Localmove}.

\noindent {\it Case B: $I_{\Fcal,\gamma}$ is a simple closed curve.}

In this case, $W_{\Fcal,\gamma}$ is a closed disk with $\partial W_{\Fcal,\gamma}=\gamma=I_{\Fcal,\gamma}$. We rule out this case because $\gamma$ is pre-peripheral.

\end{defn}

\begin{defn}[Quasi-transverse]\label{defn:Quasi-Trans}
    Suppose that $\gamma$ is a path or a closed curve on a surface $S$ and $\Fcal\in \MF(S)$. We say that $\gamma$ is {\it quasi-transverse} to $\Fcal$ if $\gamma$ is a concatenation of segments each of which is either (a) contained in a leaf and containing at least one singular point of $\Fcal$ or (b) transverse to $\Fcal$.
\end{defn}

Consider a measured foliation $\Fcal$ on $S_{g,n}$ and a path $\gamma\from I \to S_{g,n}$. We say that $\gamma$ and $\Fcal$ {\it form a bigon} if there exists a subarc $\alpha$ of $\gamma$ and an arc $\beta$ contained in a leaf of $\Fcal$ such that $\alpha \cup \beta$ bounds an open disk.

\begin{lem}\label{lem:QuasiTransverse}
    For any simple closed curve $\gamma$ and a measured foliation $\Fcal$ of a surface $S$, there is a quasi-transverse representative in $[\gamma]$. If $\gamma$ is quasi-transverse to $\Fcal$, then $\gamma$ realizes the intersection number with $\Fcal$, i.e., $\pair{\Fcal}{[\gamma]}=\pair{\Fcal}{\gamma}$. The same statement also works for paths $\delta$ joining two points $x,y\in S$ and its homotopy class fixing the endpoints $x$ and $y$.
\end{lem}
\begin{proof}
    The lemma is a consequence of the standard argument of removing bigons; see \cite[Expos\'e 5]{FLP_ThurstonSurface}.
\end{proof}

\begin{rem}
    In \cite{FLP_ThurstonSurface}, each segment of a quasi-transverse curve is additionally required to join singular points. For any quasi-transverse curve in Definition \ref{defn:Quasi-Trans}, we can slightly deform a segment contained in a leaf so as to satisfy the condition in \cite{FLP_ThurstonSurface}. See \cite[Figures 5.9 and 5.10]{FLP_ThurstonSurface} for the relevant deformation, which does not change the intersection numbers.
\end{rem}

Recall the notion of homotopy between simple cords introduced in Definition~\ref{defn:SimpleCord to Curve}.
\begin{lem}[Well-definedness and uniqueness of canonical pulls]\label{lem:IntNumbCanonPull}
    Let $\gamma$ be a pre-essential simple closed curve and $\Fcal\in \MF(S_{g,n+1})$ an $x$-generic measured foliation in $S_{g,n+1}$. Suppose that $\gamma$ and $\Fcal$ are transverse. Let $\delta\from [0,1]\to S_{g,n+1}$ be a simple cord from $x$ to $\gamma$. Then we have
    \[
        \begin{cases}
            \pair{\Fcal}{[\Pull_\delta(\gamma)]} = \pair{\Fcal}{[\gamma]} - \pair{\Fcal}{I_{\Fcal,\gamma}}
            &
            \begin{tabular}[c]{@{}l@{}}
                it $\delta$ is homotopic relative to $x,\gamma$\\
                to the canonical simple cord,
            \end{tabular}\\
            \pair{\Fcal}{[\Pull_\delta(\gamma)]} \ge \pair{\Fcal}{[\gamma]} + \pair{\Fcal}{I_{\Fcal,\gamma}} & \mathit{otherwise.}
        \end{cases}
    \]
    where $I_{\Fcal,\gamma}$ is the forgetting segment. In particular, the homotopy class of the simple pull $[\Pull_\delta(\gamma)]$ along the canonical simple cord $\delta$ is independent of the choices of representatives of $\gamma$ and $\Fcal$, while $\delta$ is dependent on these choices.
\end{lem}
\begin{proof}
In this proof, we consider punctures as marked points through which $\gamma$ can pass, but the homotopy is not allowed to pass over these marked points. Parametrize $\gamma\from [0,1] \to S_{g,n+1}$ and $\delta\from [0,1] \to S_{g,n+1}$ such that $\gamma(0)=\gamma(1)=\delta(1)$ and $\delta(0)=x$.

By Lemma~\ref{lem:QuasiTransverse}, we may deform $\delta$, via a homotopy relative to $x,\gamma$ so that it is quasi-transverse to $\Fcal$. Using a standard bigon-removal argument, we can further assume that $\gamma \cup \delta$ does not form a bigon with $\Fcal$.
Here a bigon is a closed embedded disk $D$ whose boundary $\partial D$ is a union of two arcs $\alpha_1$ and $\alpha_2$ such that $\alpha_1\subset \delta\cup \gamma$ and $\alpha_2$ is contained in a leaf of $\Fcal$. 
We briefly sketch the proof; see Figure \ref{fig:quasitransverse}. Since neither $\gamma$ nor $\delta$ forms a bigon $D$ with $\Fcal$, if $\gamma \cup \delta$ does form one, then the boundary $\partial D$ must intersect both $\gamma$ and $\delta$. Because $\delta$ intersects $\gamma$ only at $\gamma(1)$, there exists $t_0 \in [0,1]$ such that $\delta([t_0,1]) \cup \gamma$ forms a bigon with $\Fcal$. Let $t_1$ be the minimal such $t_0$. We then replace $\delta|_{[t_1,1]}$ with an arc contained in a union of leaves of $\Fcal$ that forms a bigon with $\delta|_{[t_1,1]} \cup \gamma$.
\begin{figure}
    \centering
    \def\svgwidth{0.5\textwidth}
    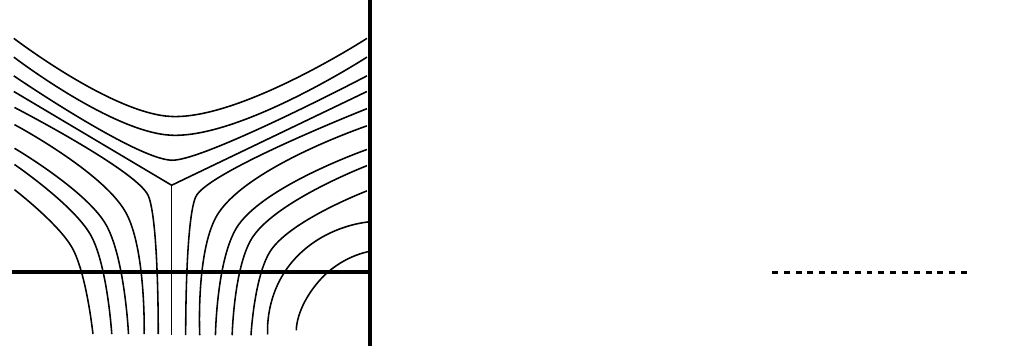
    \caption{}
    \label{fig:quasitransverse}
\end{figure}

If $\delta$ is the canonical simple cord, then it is straightforward that
\[
    \pair{\Fcal}{[\Pull_\delta(\gamma)]}=\pair{\Fcal}{[\gamma]}-\pair{\Fcal}{I_{\Fcal,\gamma}}.
\]

Suppose that $\delta$ is not the canonical simple cord. Consider the pulley $W_{\Fcal,\gamma}$ in Figure \ref{fig:ForgetDomain}-(Case A). Let $\delta'$  be the restriction of $\delta$ that is the segment from $x$ to the first intersection $y$ with $\partial W_{\Fcal,\gamma}$. If $y\in I_{\Fcal,\gamma}$, then $\delta'=\delta$, so $\delta$ is the canonical simple cord, contrary to our assumption. Hence, we have $y\notin I_{\Fcal,\gamma}$. It follows that $\delta'$ is transverse to all the leaves in $W_{\Fcal,\gamma}$, and thus
$\pair{\Fcal}{\delta}\ge \frac{1}{2}\pair{\Fcal}{I_{\Fcal,\gamma}}$. 

The closed curve $\delta * \gamma * \overline{\delta}$ does not form a bigon with $\Fcal$, and it can be approximated by simple closed curves in the homotopy class $[\Pull_\delta(\gamma)]$ within a small neighborhood of $\gamma \cup \delta$. Moreover $\delta * \gamma * \overline{\delta}$ is quasi-transverse to $\Fcal$. Hence, by Lemma \ref{lem:QuasiTransverse}, we have
\begin{equation}\label{eqn:Int_Num_Change}
    \pair{\Fcal}{[\Pull_\delta(\gamma)]}=\pair{\Fcal}{\gamma}+2\pair{\Fcal}{\delta}.
\end{equation}
This completes the proof.
\end{proof}

\begin{proof}[Proof of Theorem \ref{thm:CurveReduction}-(1,2,4,5)]
    The proof of (1),(2), and (5) is straightforward from Lemma \ref{lem:IntNumbCanonPull}.
    Let us consider the orientation on $\Edge(X_\zeta)$. By (1,2), every vertex has a unique outgoing edge. Since $X_\zeta$ is a tree, for any pair of vertices $[\gamma], [\gamma'] \in \Vertex(X_\zeta)$, the infinite directed edge-paths from $[\gamma]$ and $[\gamma']$ eventually coalesce. Then (4) follows.
\end{proof}

Recall that we denote by $\Proj\from \MF(S_{g,n+1})\to \PMF(S_{g,n+1})$ the projection map. By \cite{Kla_BdryCurveCplx}, the boundary $\partial_\infty CC(S_{g,n+1})$ is identified with the set of ending laminations $\mathcal{EL}(S_{g,n+1})$, which consists of equivalence classes of minimal foliations. For every minimal measured foliation $\Fcal$, denote by $\underline{\Fcal}$ its underlying minimal foliation, which is an element of $\mathcal{EL}(S_{g,n+1})=\partial_\infty CC(S_{g,n+1})$. 

We say that a sequence $(\Proj([\gamma_k])) \in \PMF(S)$ is {\it bounded} in $CC(S)$ if there is $C>0$ such that $d([\gamma_1],[\gamma_k])<C$ for any $k>1$, where $d$ denotes the edge-path metric on the $1$-skeleton of $CC(S)$.

\begin{proof}[Proof of Theorem \ref{thm:CurveReduction}-(3)]
Fix a hyperbolic metric $X$ on $S_{g,n+1}$. Since $(\gamma(l))_{l\ge0}$ is a sequence of distinct simple closed curves, we have $L_{\rm hyp}(\gamma(l);X) \to \infty$ as $l\to \infty$, where $L_{\rm hyp}(-;X)$ is the hyperbolic length function. It follows that 
\[
    \pair{\Fcal}{\frac{1}{L_{\rm hyp}(\gamma(l),X)} [\gamma(l)]} \to 0
\]
as $l\to \infty$.

Let $\Gcal$ be a limit point of the sequence $\left(\frac{1}{L_{\rm hyp}(\gamma(l),X)} [\gamma(l)]\right)_{l\ge1}$. It follows that \(L_{\mathrm{\rm hyp}}(\Gcal)=1\) and \(\pair{\Fcal}{\Gcal}=0\). Hence, by
\cite[Theorem~1.12]{Rees_MF}, the underlying foliations $\underline{\Fcal}$ and $\underline{\Gcal}$ of \(\Fcal\) and \(\Gcal\) are Whitehead equivalent. Since every convergent subsequence has this property,
\cite[Theorem~1.1]{Hamen_Bdry_CurveCplx} implies that the sequence
\([\gamma(l)]\) converges to
\(\underline{\Fcal}\in\mathcal{EL}(S_{g,n+1})\).
\end{proof}

\begin{proof}[Proof of Theorem \ref{thm:CurveReduction}-(6)]
    By Theorem \ref{thm:CurveReduction}-(5), if $([\gamma(l)])_{l=0,-1,-2,\dots}$ is an infinite reversely-directed edge-path in $X_\zeta$, then $\pair{\Fcal}{[\gamma(l)]}$ increases at least linearly fast as $l\to \infty$. Hence, we have $\pair{\Fcal}{[\gamma(l)]} \to \infty$ as $l\to \infty$.
\end{proof}


\begin{proof}[Proof of Theorem \ref{thm:CurveReduction}-(7)]
Any consecutive pair $D(\gamma(l))$ and $D(\gamma(l+1))$ shares a connected component of $p^{-1}(\zeta)$, say $\beta_l$. Let $U_l$ be the connected component of $\widetilde{S_{g,n}} \setminus \beta_l$ that contains $\interior(D(\gamma(l+1)))$. Then we have a nested sequence of compact sets $\overline{U_1} \supset \overline{U_2} \supset \dots$ where the closures are taken in the closed disk $\widetilde{S_{g,n}} \cup \partial_\infty \widetilde{S_{g,n}}$. Then, the existence of the limit $\theta_\Fcal$ of $(D(\gamma(l)))_{l\ge1}$ follows from Lemma \ref{lem:Limit_Seg_Geod}.
\end{proof}

We will prove Theorem \ref{thm:CurveReduction}-(8) in Section \ref{subsec:Multicurve_Reduction}.

\subsection{Case 2: Multicurves}\label{subsec:Multicurve_Reduction}

\subsection*{Canonical simple pulls for multicurves}

\begin{figure}[h!]
	\centering
    \begin{subfigure}[b]{0.35\textwidth}
    \centering
    \def\svgwidth{0.9\textwidth}
    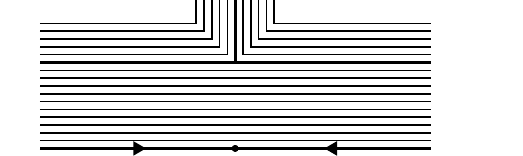
    \caption{}
    \end{subfigure}
    \begin{subfigure}[b]{0.24\textwidth}
    \centering
    \def\svgwidth{0.9\textwidth}
    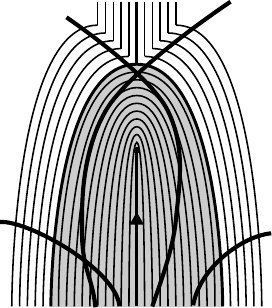
    \caption{}
    \end{subfigure}
    \begin{subfigure}[b]{0.3\textwidth}
    \centering
    \def\svgwidth{0.9\textwidth}
    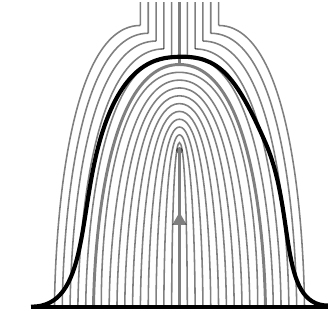
    \caption{}
    \end{subfigure}
	\caption{}
	\label{fig:Localmove}
\end{figure}
Suppose that $\Fcal \in \MF(S_{g,n+1})$ is an $x$-generic and minimal measured foliation. Consider a pre-minimal multicurve $\Gamma=\{\gamma_1, \gamma_2, \dots\gamma_k\}$ of $S_{g,n+1}$ that is transverse to $\Fcal$, denoted by $\Fcal \pitchfork \Gamma$. Let us look at a part of $\Fcal$ near the singular point $x$. Let $L_x$ be the singular trajectory starting from $x$. Let $y$ be the first intersection of $L_x$ with $\Gamma$. Without loss of generality, we assume $y\in \gamma_1$. Denote by $\delta:=[x,y]$ the segment of $L_x$ between $x$ and $y$. Consider the pulley $W_{\Fcal,{\gamma_1}}$ and the forgetting segment $I_{\Fcal,{\gamma_1}}$ as in Case-(A) in Definition \ref{defn:PulleyForgetSegments}. Since we consider a multicurve, the pulley $W_{\Fcal,\gamma_1}$ may intersect other components of $\Gamma$ as in Figure \ref{fig:Localmove}. We locally change $\gamma_1$ to $\gamma_1'$ as in Figure \ref{fig:Localmove}-(C). We note that $\gamma_1'$ is also transverse to $\Fcal$ and $[\gamma_1']$ is the canonical simple pull of $[\gamma_1]$. Define $\Gamma'=\{\gamma_1',\gamma_2,\dots,\gamma_k\}$. We can further assume the following property, which will be useful when proving that each component is simply pulled infinitely often through the iteration of this process:

\medskip
({\it Non-blocking first intersections}) $\gamma'_1$ may be chosen very close to $\partial W_{\Fcal,\gamma_1}$ so that it does not block the first intersection of $L_x$ with each $\gamma_i$ with $i\neq 1$. More precisely, let $y_i$ be the first intersection of $L_x$ with $\gamma_i$. Then $|[x,y_i] \cap \Gamma'|=|[x,y_i] \cap \Gamma|-1$.
\medskip

Let us discuss the change of $\Gamma$ to $\Gamma'$ in more detail. The local picture can be obtained from Figure \ref{fig:Localmove}-(A) by gluing the two sides on the bottom. If a curve $\gamma_2$ intersects the shaded region, there are five possibilities: $\gamma_2$ passes through (1) I and III, (2) I and IV, (3) II and III, (4) II and IV, and (5) III and IV. The type (5) can be excluded by assuming $\Gamma$ to be pre-minimal. One curve for each of types (1--4) is drawn in Figure \ref{fig:Localmove}-(B,C).

Consider how the local move affects the intersection of curves and $L_x$.

\begin{enumerate}
    \item The intersection of $L_x$ with curves other than $\gamma_1$ is unchanged.
    \item An intersection point of $L_x$ and $\gamma_1$ still appears as an intersection point of $L_x$ and $\gamma'$ unless it is on the side III or IV.
    \item All the intersection points between $L_x$ and $\gamma_1$ on the sides III and IV are eliminated.
    \item $\gamma_1'$ still has infinitely many intersections with $L_x$.
\end{enumerate}
We can choose $\gamma_1'$ in such a way that
\begin{enumerate}[resume]
    \item $\Fcal\pitchfork \Gamma'$.
\end{enumerate}
The intersection number with $\Fcal$ decreases
\begin{enumerate}[resume]
    \item
    $\pair{\Fcal}{[\Gamma]}-\pair{\Fcal}{[\Gamma']}=\pair{\Fcal}{\gamma_1}- \pair{\Fcal}{\gamma_1'} = \pair{\Fcal}{III \cup IV}>0$.
\end{enumerate}
\noindent We have $\pair{\Gamma}{\Gamma}=\pair{\Gamma'}{\Gamma'}$. Therefore, since $\Gamma$ is pre-minimal, $\Gamma'$ must also be pre-minimal. Hence, we have
\begin{enumerate}[resume]
    \item $[\Gamma]$ and $[\Gamma']$ are adjacent vertices in $X_Z$.
\end{enumerate}

We call $[\Gamma']$ a {\it canonical simple pull} of $[\Gamma]$. Unlike the case of simple closed curves, there may be more than one canonical simple pull of $[\Gamma]$. By Theorem \ref{thm:CurveReduction}-(2), there are at most $k(=|\Gamma|)$ simple pulls. In Theorem \ref{thm:MulticurveReduction}, we show that there is a unique cube in $X_Z$ containing the vertex $[\Gamma]$ such that the cube includes the edges corresponding to the canonical simple pulls of $[\Gamma]$.

\subsection*{Iteration of canonical simple pulls}
Recall that given representatives $\Fcal$ and $\Gamma$ with $\Fcal \pitchfork \Gamma$, we obtained a canonical simple pull $\Gamma'$. The choice of representatives $\Fcal$ and $\Gamma$ uniquely determines the homotopy class $[\Gamma']$ but not its representative $\Gamma'$. When iterating canonical simple pulls, we choose a representative $\Gamma'$ that satisfies the non-blocking first intersections, and then we apply the canonical simple pull to $\Gamma'$. In other words, we iterate canonical simple pulls for multicurves, not for their homotopy classes.

By iterating canonical simple pulls, we obtain a sequence of multicurves $(\Gamma(l))_{l\ge0}$ with $\Gamma(0)=\Gamma$. We remark that the sequence of even their homotopy classes $([\Gamma(l)])_{l\ge0}$ may depend on the choice of representative $\Gamma$ of $[\Gamma]$ and also on the choice of the canonical simple pull $\Gamma'$ of $\Gamma$. However, we will show in Theorem \ref{thm:MulticurveReduction} that all these sequences are (consecutive vertices on) geodesic edge-paths in $X_Z$ that are asymptotic to each other.

\begin{lem}\label{lem:Seq of Cannon Simple Pulls}
Suppose that $\Gamma=\{\gamma_1,\gamma_2,\dots,\gamma_k\}$ is a pre-minimal, pre-essential, and pre-non-parallel multicurve in $S_{g,n+1}$. Let $\Fcal$ be an $x$-generic minimal measured foliation on $S_{g,n+1}$. Let $\Gamma(l)=\{\gamma_1(l),\gamma_2(l),\dots,\gamma_k(l)\}$ for $k=0,1,2,\dots$ be a sequence of multicurves on $S_{g,n+1}$ obtained by iteration of canonical simple pulls from $\Gamma$ such that $\gamma_i(l)$ is pre-homotopic to $\gamma_i(0)$ for any $l\ge0$ and $i\in\{1,2,\dots,k\}$. Then the sequence $([\Gamma(l)])_{l\ge0}$ is a sequence of consecutive vertices in an unbounded directed edge-path $p$ in $X_Z$ such that for each $i\in \{1,2,\dots,k\}$, $\pi_i\circ p$ is an infinite directed edge-path with layovers in $X_{\zeta_i}$.
\end{lem}
\begin{proof}
    The only non-trivial statement is that $\pi_i\circ p$ is unbounded in $X_{\zeta_i}$. The canonical simple pull of a multicurve $\Gamma'$ only changes the component of $\Gamma'$ that contains the first intersection of $\Gamma'$ with $L_x$. By the minimality of $\Fcal$, $L_x$ intersects every simple closed curve that is transverse to $\Fcal$ infinitely often. By the (Non-blocking first intersection) property, if the first intersection $z$ of $\gamma_i(l)$ with $L_x$ is the $N$-th intersection of $\Gamma(l)$ with $L_x$, then $z$ is the first intersection of $\Gamma(l+N-1)$ with $L_x$. Hence $\pi_i \circ p$ is unbounded in $X_{\zeta_i}$.
\end{proof}

Two infinite edge-paths $p_1$ and $p_2$ in $X_Z$ are {\it asymptotic} each other if they are uniformly distant, i.e., there exists $D>0$ such that for any integer $n\ge 0$, we have
\[
    d_{(X_Z,L^1)}(p_1(n),p_2(n))\le D.
\]

\begin{thm}\label{thm:MulticurveReduction}
    Let $\Gamma=\{\gamma_1,\gamma_2,\dots,\gamma_k\}$ be a multicurve of $S_{g,n+1}$ that is pre-essential, pre-non-parallel, and pre-minimal. Let $i_*[\Gamma]=[Z]$ and $X_Z$ be the complex of pre-homotopic multicurves. Suppose that $\Fcal$ is an $x$-generic minimal measured foliation on $S_{g,n+1}$.
    \begin{enumerate}
        \item Every simple pull of $[\Gamma]$ strictly increases or decreases the intersection number of some $[\gamma]\in [\Gamma]$ with $\Fcal$. For any representatives $\Fcal$ and $\Gamma$ with $\Fcal \pitchfork \Gamma$, the canonical simple pull of $\Gamma$ strictly decreases the intersection number with $\Fcal$.
        \item For any $[\Gamma]\in \Vertex(X_Z)$ and any component $\gamma\in \Gamma$, there is at most one simple pull that decreases the intersection number with $\Fcal$ so that $\gamma$ is the to-be-pulled component.
    \end{enumerate}
    Let us orient every edge of $X_Z$ toward the direction that the intersection number with $\Fcal$ decreases.
    \begin{enumerate}[resume]
        \item $\{p_i\from X_Z\to X_{\zeta_i}\}_{i=1,2,\dots,k}$ is an oriented tree coordinate system of $X_Z$.
        \item The orientation on $X_Z$ defines a point $\zeta$ in the Roller boundary $\partial_R X_Z$, and thus $X_Z$ has mono-outgoing cubes.
        \item Every directed edge-path in $X_Z$ is a geodesic edge-path. Moreover, any two infinite directed edge-paths are asymptotic to each other.
    \end{enumerate}
    Let $([\Gamma(l)])_{l\ge0}$ be an infinite directed edge-path of $X_Z$, where $\Gamma(l)=\{\gamma_1(l),\gamma_2(l),\dots,\gamma_k(l)\}$, with $\gamma_i(l)$ pre-homotopic to $\gamma_i\in \Gamma$.
    \begin{enumerate}[resume]
        \item For any $i\in \{1,2,\dots,k\}$, the sequence $([\gamma_i(l)])_{l\ge0}$ defines an unbounded directed edge-path with layovers in $X_{\zeta_i}$. Hence, we have
        \[
            \pair{\Fcal}{[\gamma_i(l)]}\to \varphi(\Fcal)([\zeta_i]),
        \]
        as $l \to \infty$.
        \item There exist $[\Gamma'] \in \Vertex(X_Z)$ and a sequence of point-pushing mapping classes $([\psi_l])_{l\ge1}$ such that for any $\gamma' \in \Gamma'$
        \[
            \pair{\Fcal}{[\psi_l] \cdot [\gamma']} \to \varphi(\Fcal)(i_*[\gamma'])
        \]
        as $l\to \infty$.
    \end{enumerate}
\end{thm}

\begin{proof}[Proof of Theorem \ref{thm:MulticurveReduction}]
The statements (1) and (2) follow from the definition and Theorem \ref{thm:CurveReduction}.

For the oriented cube complex $X_Z$, the statement (3) is a corollary of Theorem \ref{thm:CubeCplxIso}. Then, (4) follows from Lemma \ref{lem:OutgoingCube}, Lemma \ref{lem:Seq of Cannon Simple Pulls}, and Proposition \ref{prop:EdgeOri_WallOri}.
Then, the first statement of (5) follows from
Proposition \ref{prop:TreeCoordiSystem}-(4). The second statement of (5) follows from Theorem \ref{thm:HypCubeCplx} and
Theorem \ref{thm:Roller-to-Gromov}.

The statement (6) follows from Lemma \ref{lem:Seq of Cannon Simple Pulls}, Lemma \ref{lem:OutgoingCube}, and Theorem \ref{thm:CurveReduction}-(4).

Since $X_Z/\pi_1(S_{g,n},x)$ is compact, the sequence $([\Gamma(l)])_{l\ge0}$ has an infinite intersection with an $\pi_1(S_{g,n},x)$-orbit. By choosing one $[\Gamma']$ from the orbit, we obtain (7).
\end{proof}

\begin{proof}[Proof of Theorem \ref{thm:CurveReduction}-(8)] Fix $[\zeta_1],[\zeta_2]\in C(S_{g,n})$ with $[\zeta_1]\neq[\zeta_2]$. Suppose $\gamma_1$ and $\gamma_2$ are simple closed curves in $S_{g,n+1}$ such that $\Gamma=\{\gamma_1,\gamma_2\}$ is pre-minimal (Lemma~\ref{lem:Make_pre-minimal}), $i_*[\gamma_1]=[\zeta_1]$, and $i_*[\gamma_2]=[\zeta_2]$. Denote by $\theta_{\Fcal,\gamma_1}$ and $\theta_{\Fcal,\gamma_2}$ the points on the boundary of $\mathbb{H}$ defined in Theorem \ref{thm:CurveReduction}-(7), by using the sequences of canonical simple pulls in $X_{\zeta_1}$ and $X_{\zeta_2}$. Let $Z=\{\zeta_1,\zeta_2\}$.

Suppose $\theta_{\Fcal,\gamma_1} \neq \theta_{\Fcal,\gamma_2}$. Denote by $p$ the directed edge-path in $X_Z$ along the sequence $([\Gamma(l)])_{l\ge0}$ with $\Gamma(0)=\Gamma$ in $X_Z$ defined in Lemma \ref{lem:Seq of Cannon Simple Pulls}. By Theorem \ref{thm:MulticurveReduction}-(6), for each $i\in \{1,2\}$, $p_i \circ p$ is an unbounded edge-path with layover. Denote by $L_i$ the image of $p_i \circ p$, which is homeomorphic to $[0,\infty)$.

Denote by $H_l$ the dual hyperplane of the edge between $[\Gamma(l)]$ and $[\Gamma(l+1)]$. Recall that the hyperplanes of $X_Z$ are in a bijective correspondence with the walls in $\mathcal{W}_Z$. Denote by $w_l$ the wall corresponding to $H_l$. The correspondence also preserves the intersection relationships for hyperplanes and walls, and for their half-spaces: $H_i \cap H_j \neq \emptyset$ if and only if $w_i \cap w_j\neq \emptyset$, and $\overrightarrow{H_i} \cap \overrightarrow{H_j} \neq \emptyset$ if and only if $\overrightarrow{w_i} \cap \overrightarrow{w_j}\neq \emptyset$. Let $\overrightarrow{H_l}$ be the half-space of $H_l$ that contains $[\Gamma(l+1)]$.

Define $I^1=\{l^1_1,l^1_2,\dots~|~l^1_i<l^1_{i+1}\}$ (resp.\@ $I^2=\{l^2_1,l^2_2,\dots~|~l^2_i<l^2_{i+1}\}$) as the set of non-negative integers $l^1_i$ for which the simple pull from $[\Gamma(l^1_i)]$ to $[\Gamma(l^1_i+1)]$ changes the curve pre-homotopic to $\gamma_1$ (resp.\@ $\gamma_2$). By the last sentence of the statement of Lemma \ref{lem:OutgoingCube}, both $I^1$ and $I^2$ are infinite sets.

Denote by $\overrightarrow{w_{l^1_i}}$ (resp.\@ $\overrightarrow{w_{l^2_i}}$) the half-space of $w_{l^1_i}$ (resp.\@ $w_{l^2_i}$) whose closure in $\widetilde{S_{g,n}}\cup \partial \widetilde{S_{g,n}}$ contains $\theta_{\Fcal,1}$ (resp.\@ $\theta_{\Fcal,2}$). By Theorem \ref{thm:CurveReduction}-(7) and $\theta_{\Fcal,\gamma_1} \neq \theta_{\Fcal,\gamma_2}$, $\overrightarrow{w_{l^1_i}} \cap \overrightarrow{w_{l^2_j}}=\emptyset$ for any sufficiently large $i,j>0$.

For $H_l=$ $w_{l^1_i}$ or $w_{l^2_i}$, let $\overrightarrow{H_l}$ be the half-space of $H_l$ that corresponds to $\overrightarrow{w_{l^1_i}}$ or $\overrightarrow{w_{l^2_i}}$. We claim that for any $l\neq l'\ge 0$, we have $\overrightarrow{H_l}\cap \overrightarrow{H_{l'}}\neq \emptyset$. Consider the oriented tree coordinate system $\{p_i\from X_Z \to X_{\zeta_i}\}_{i=1,2}$. For $v:=[\Gamma(0)]\in \Vertex(X_Z)$ and $i\in \{1,2\}$, let $L_i=\langle p_i(v)\rangle$ be the sub-directed graph of $X_{\zeta_i}$ generated by $p_i(v)$. By Theorem \ref{thm:CurveReduction}-(2), we can identify $L_i$ with $[0,\infty)$ such that integers are vertices. Then, for any $i\in\{1,2\}$ and $j\ge1$, we have $p_i(\Gamma(l^i_{j}))=j$ and $p_i(\Gamma(l^i_{j}+1))=j+1$. The half-spaces $\overrightarrow{H_l}$ are of the form $p_i^{-1}([n+0.5,\infty))$ for $i\in \{1,2\}$ and $n\in \mathbb{Z}_{\ge 0}$. The fact that $|I^1|=|I^2|=\infty$ implies $p_1^{-1}([r,\infty)) \cap p_2^{-1}([s,\infty))\neq\emptyset$ for any positive real numbers $r,s>0$.

Therefore, we obtain a contradiction because the half-spaces of walls are disjoint, whereas the half-spaces of hyperplanes intersect.
\end{proof}

\section{Construction of forgetting maps for measured foliations}\label{sec:ForgetMapMF}
In this section, we use theorems in Section \ref{Sec:SimplePulls} to show that for an $x$-generic measured foliation $\Fcal\in \MF(S_{g,n+1})$, there exists $\Phi(\Fcal)\in \MF(S_{g,n})$ such that $\pair{\Phi(\Fcal)}{-}=\varphi(\Fcal)$ as functionals $C(S_{g,n})\to \mathbb{R}$.

\subsection{Non-minimal measured foliations}\label{subsec:NonMinimal} In this subsection, we briefly review some properties of measured foliations.

\subsection*{Measured foliations on surface with boundary} Let us first discuss measured foliations on surfaces with non-empty boundary. Let $S_{g,n,b}$ denote surfaces with $n$ punctures and $b$ boundary circles. We allow marked points to be on boundary components. In this article, we mainly care about measured foliations $\Fcal$ on $S_{g,n,b}$ satisfying the following boundary conditions \cite[Chapter 11.2.2]{FM_PrimerMCG}:
\begin{enumerate}
    \item Every boundary component of $S_{g,n,b}$ contains at least one singular point.
    \item Every boundary component is a union of leaves joining singular points. 
\end{enumerate}
With this boundary condition, by collapsing each boundary component to a puncture, we obtain a measured foliation $\overline{\Fcal}$ on a closed surface $S_{g,n+b}$ with $n+b$ punctures.

\subsection*{Decomposition of measured foliations} Suppose $\Fcal$ is a (possibly non-minimal) measured foliation on a surface $S$. We can think of $\Fcal$ as a union $\Fcal_1\cup \Fcal_2\cup \dots \cup \Fcal_k$ of partial measured foliations $\Fcal_i$ of $S$ such that each $\Fcal_i$ is supported in an immersed subsurface $S_i$ of $S$. More precisely, there exists a compact surface $S_i$ and a continuous map $\iota_i\from S_i\to S$ such that $\iota_i$ is an embedding in ${\rm int}(S_i)$. The measured foliation $\Fcal_i$ lifts to a measured foliation, which we also denote by $\Fcal_i$ by abusing notation, on $S_i$. The interiors ${\rm int}(S_i)$'s are disjointly embedded in $S$, but
\[
    \bigcup_{i} \iota_i(\partial S_i)= \bigcup\{{\rm leaves~joining~singular~points~of~}\Fcal\}.
\]
Moreover, each $\Fcal_i$ satisfies one of the following: (a) $\Fcal_i$ is a minimal measured foliation of $S_i$, or (b) $(S_i, \Fcal_i)$ is an annulus foliated by circles.
We refer to $\Fcal_i$ in case (a) as a {\it minimal component} and in case (b) as an {\it annulus component} of $\Fcal$. We call $S_i$ the support of $\Fcal_i$, denoted also by $\supp(\Fcal_i)$. It is worth noting that it may happen that $k=1$ but ${\rm int}(S_1) \subsetneq S$---this occurs, for instance, when $\Fcal$ is not minimal but has a dense leaf.

\begin{lem}\label{lem:minimal component of x}
    Suppose that $\Fcal$ is an $x$-generic measured foliation on $S_{g,n+1}$. Let $\Fcal=\bigcup_{i=1}^k \Fcal_i$ be the decomposition of $\Fcal$ into minimal and annulus components. Then, there exists a minimal component $\Fcal_j$ such that $x$ lies in ${\rm int}(S_j)$, where $S_j$ is the support of $\Fcal_j$.
\end{lem}
\begin{proof}
    If $x$ is in $\iota_i(\partial S_i)$, then $x$ is in a leaf joining singular points. This contradicts the $x$-genericity of $\Fcal$.
\end{proof}

\subsection{Non-generic measured foliations}\label{subsec:Non-generic}

Suppose that $\Fcal$ is an $x$-non-generic measured foliation on $S_{g,n+1}$; see Definition~\ref{defn:genericMF}. Then we define $\Phi(\Fcal)\in \MF(S_{g,n})$ in a natural way. There are three cases as illustrated in Figure \ref{fig:Forget_nondeg}.

\begin{figure}[h!]
    \centering
    \begin{subfigure}[t]{0.20\textwidth}
        \centering
        \def\svgwidth{\textwidth}
        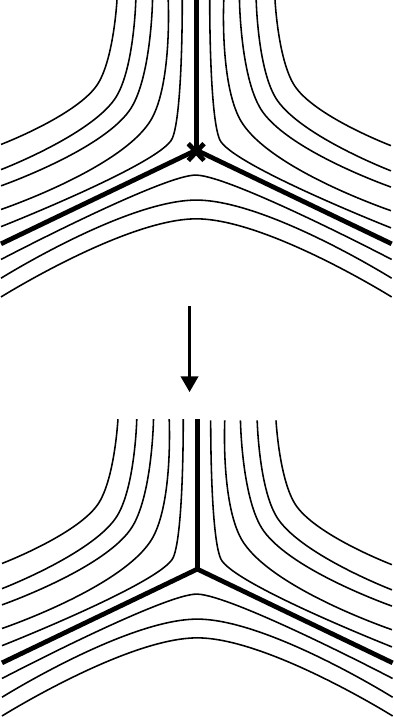
        \caption{}
    \end{subfigure}\hspace{20pt}
    \begin{subfigure}[t]{0.20\textwidth}
        \centering
        \def\svgwidth{\textwidth}
        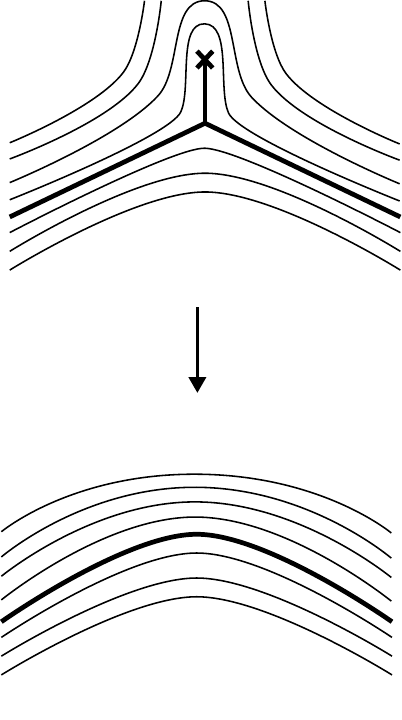
        \caption{}
    \end{subfigure}\hspace{20pt}
    \begin{subfigure}[t]{0.20\textwidth}
        \centering
        \def\svgwidth{\textwidth}
        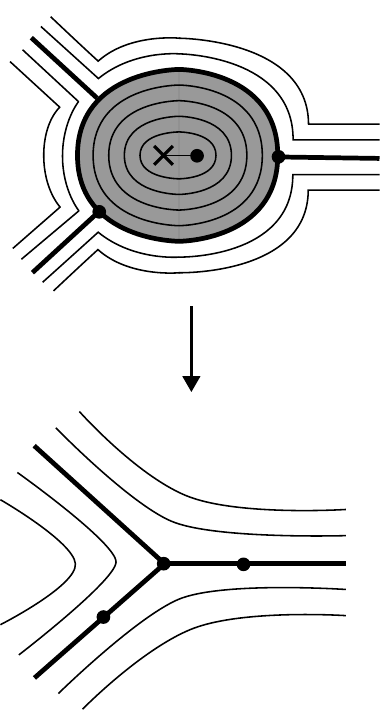
        \caption{}
    \end{subfigure}
    \caption{(A) and (B) are of non-annular type, and (C) is of annular type.}
    \label{fig:Forget_nondeg}
\end{figure}
\begin{enumerate}
    \item[(A)] $\deg_\Fcal(x)>1$.\indent  We define $\Phi(\Fcal)$ by literally removing $x$ from the set of marked points of $\Fcal$.
\end{enumerate}
Suppose that $\deg_\Fcal(x)=1$. Since $\Fcal$ is $x$-non-generic, the leaf $L_x$ is a saddle connection $[x,y]$, where $y$ is other singular point of $\Fcal$. 
\begin{enumerate}
    \item[(B)] $\deg_\Fcal(x)=1<\deg_\Fcal(y)$.\indent We define $\Phi(\Fcal)$ by collapsing $[x,y]$ to a point $y$ and forget $x$ from $\Fcal$.
    \item[(C)] $\deg_\Fcal(x)=\deg_\Fcal(y)=1$.\indent The leaf $[x,y]$ is surrounded by parallel circular leaves. Let $U$ be the largest saturated neighborhood of $[x,y]$ foliated by these circular leaves together with $[x,y]$. Then $\partial U$ must contain at least one singular point of $\Fcal$. Note that a singular point of $\Fcal$ is not necessarily a marked point. Let $k$ denote the number of singular points on $\partial U$. We then forget $x$ and collapse $\overline{U}$ to a starlike tree with $k$ ends, whose central vertex corresponds to $y$ and whose endpoints correspond to the singular points of $U$. Figure~\ref{fig:Forget_nondeg}~(C) illustrates the case where $\partial U$ contains three singular points, two of which are marked.
\end{enumerate}
We refer to Cases (A) and (B) as {\it non-annular type}, and to Case (C) as {\it annular type}.

\subsection*{Properties of $x$-non-generic measured foliations relative to simple closed curves}
We suppose that $\Fcal\in \MF(S_{g,n+1})$ is $x$-non-generic and $\pair{\Fcal}{[\gamma]} > 0$, where $[\gamma] \in C(S_{g,n+1})$.

We say that $\Fcal$ is \emph{immediately} $x$-non-generic \emph{with respect to (wrt)} $\gamma$ if $\Fcal$ and $\gamma$ can be represented transversely so that (i) $\deg_\Fcal(x) \ge 2$ if $\Fcal$ is of non-annular type, or (ii) $[x,y]\cap \gamma=\emptyset$ if $\Fcal$ is of annular type and $[x,y]$ is the saddle connection. (Here, being transverse should imply that $\gamma$ does not pass through any singular points of $\Fcal$.) Otherwise, we say that $\Fcal$ is \emph{eventually} $x$-non-generic \emph{wrt} $\gamma$.

Suppose that $\Fcal$ is eventually $x$-non-generic wrt $\gamma$. By following the same argument as in the $x$-generic case, one can show that there exists a unique simple pull $[\gamma']$ of $[\gamma]$ such that $\pair{\Fcal}{[\gamma]}>\pair{\Fcal}{[\gamma']}$. We call $[\gamma']$ the {\it canonical simple pull} of $[\gamma]$.

Recall the pulley $W_{\Fcal,\gamma}$ and the forgetting interval $I_{\Fcal,\gamma}$ of $\Fcal$ wrt $\gamma$ (Definition \ref{defn:PulleyForgetSegments}).

\begin{defn}[Good/bad normalized representative]
    Suppose that $\Fcal$ is an eventually $x$-non-generic measured foliation type wrt $[\gamma]\in C(S_{g,n+1})$. A representative $\Fcal$ is said to be {\it normalized} wrt $\gamma$ if $\Fcal\pitchfork \gamma$, $\deg_\Fcal(z)=1$ for all marked points $z$, $\deg_\Fcal(y)=3$ for all other singular points $y$ of $\Fcal$, and $\partial W_{\Fcal,\gamma}$ contains only one singular point.
    
    When $\Fcal$ is normalized, the leaf $L_x$ of $x$ is a saddle connection $[x,y]$ for some singular point $y$. We say that $\Fcal$ is a {\it bad} normalized representative if $y\in\partial W_{\Fcal,\gamma}$ so that near $y$, $L_x$ approaches $y$ along the prong not contained in $\partial W_{\Fcal,\gamma}$. Otherwise, $\Fcal$ is a {\it good} normalized representative. If $\Fcal$ is of annular type, then every normalized representative $\Fcal$ is good.
\end{defn}

If $\Fcal$ is not normalized, then we can normalize it via Whitehead moves along $\partial W_{\Fcal,\gamma}$. This normalizing process preserves the number $|\gamma \cap L_x|$.

\begin{defn}[Clean long pulley]\label{defn:CleanLongPulley}
    Suppose that $\Fcal$ is a normalized representative wrt $\gamma$. We say that $\Fcal$ has a {\it clean long pulley} with respect to $\gamma$ if the pulley $W_{\Fcal,\gamma}$ can be extended along the initial segment $L'_x$ of $L_x$ that contains $\gamma \cap L_x$. In other words, the forgetting segment $I_{\Fcal,\gamma}$ can be slid along $L_x$ over the intersection $\gamma \cap L_x$. See Figure~\ref{fig:CleanLongPulleys}.
\end{defn}
\begin{figure}[h!]
	\centering
    \begin{subfigure}[b]{0.4\textwidth}
    \centering
    \def\svgwidth{\textwidth}
    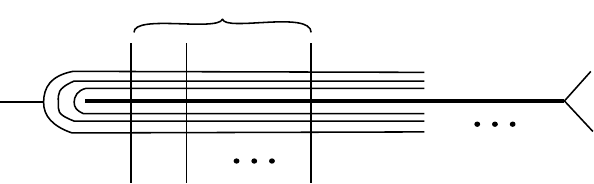
    \caption{$L_x\neq [x,y]$}
    \end{subfigure}
    \hspace{20pt}
    \begin{subfigure}[b]{0.4\textwidth}
    \centering
    \def\svgwidth{\textwidth} 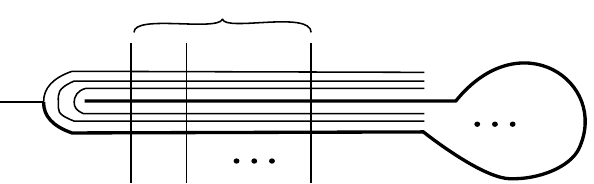
    \caption{$L_x=[x,y]$}
    \end{subfigure}
	\caption{Clean long pulleys}
	\label{fig:CleanLongPulleys}
\end{figure}

Denote by $n(\mathcal{F}, [\gamma])$ the number of canonical simple pulls of $[\gamma]$ that can be performed wrt $\mathcal{F}$. That is, $\mathcal{F}$ is immediately $x$-non-generic wrt the $n(\mathcal{F}, [\gamma])^{th}$-canonical simple pull $\gamma_{n(\mathcal{F}, [\gamma])}$ of $[\gamma]$ wrt $\mathcal{F}$. We define $n(\mathcal{F}, [\gamma])=0$ if $\Fcal$ is immediately $x$-non-generic wrt $\gamma$ and $n(\mathcal{F}, [\gamma])=\infty$ if $\Fcal$ is $x$-generic. In Lemma \ref{lem:non-gen_FiniteCanSimPulls}, we prove that $n(\mathcal{F}, [\gamma])<\infty$ if $\Fcal$ is $x$-non-generic.

\begin{lem}\label{lem:non-gen_FiniteCanSimPulls}
    Suppose that $\Fcal$ is eventually  $x$-non-generic wrt $\gamma$ and $\Fcal$ is a normalized representative wrt $\gamma$. Denote by $L_x$ the leaf of $x$. Then
    \[
        1\le |L_x \cap \gamma|\le 2\, n(\Fcal,\gamma)<\infty.
    \]
    Moreover, suppose that $\Fcal$ is a good normalized representative that has a clean long pulley wrt $\gamma$. If $L_x\neq[x,y]$, then
    \[
        n(\Fcal,\gamma)=|L_x \cap \gamma|,
    \]
    whereas if $L_x=[x,y]$, then
    \[
        n(\Fcal,\gamma)=\frac{1}{2}|L_x \cap \gamma|.
    \] 
\end{lem}
\begin{proof}
    We define a representative $\gamma_1$ of the canonical simple pull of $[\gamma]$ so that $\Fcal \pitchfork \gamma_1$, with the modification from $\gamma$ to $\gamma_1$ occurring only in a small neighborhood of the pulley (see Figure~\ref{fig:Localmove}). Let $y$ denote the singular point on the boundary of the pulley.
    If $\Fcal$ is bad wrt $\gamma_1$, then
    \[
        |\gamma_1 \cap L_x| = |\gamma \cap L_x|.
    \]
    If $\Fcal$ is good wrt $\gamma_1$ and $L_x\neq[x,y]$, then
    \[
        |\gamma_1 \cap L_x| = |\gamma \cap L_x| - 1.
    \]
    If $\Fcal$ is good wrt $\gamma_1$ and $L_x=[x,y]$, then
    \[
        |\gamma_1 \cap L_x| = |\gamma \cap L_x| - 2.
    \]

    If $\Fcal$ is not normalized wrt $\gamma_1$, then we normalize it via Whitehead moves without changing $|\gamma_1 \cap L_x|$. We then iterate the canonical simple pull and let $\gamma_i$ denote the $i^{th}$-canonical simple pull of $\gamma$. It follows that $1\le |L_x \cap \gamma|\le 2n(\Fcal,\gamma)$.
    
    To show $n(\Fcal,\gamma)<\infty$, it suffices to prove that $\Fcal$ cannot be bad wrt $\gamma_i$ for all sufficiently large $i$.
    For contradiction, without loss of generality, suppose that $L_x$ is bad wrt $\gamma_i$ for all $i\ge1$.
    Since $\Fcal$ is bad with respect to $\gamma_2$, the endpoint $y$ of $L_x$ remains the only singular point on $\partial W_{\Fcal,\gamma_2}$. This implies that the widths of $\partial W_{\Fcal,\gamma_1}$ and $\partial W_{\Fcal,\gamma_2}$ are equal. By induction, it follows that the width of the pulley $W_{\Fcal,\gamma_i}$ is independent of $i$. In other words, $\langle \Fcal, \gamma_i \rangle - \langle \Fcal, \gamma_{i+1} \rangle > 0$ is a constant independent of $i$. Consequently, $\langle \Fcal, \gamma_i \rangle < 0$ for all sufficiently large $i$, which is a contradiction.

    Suppose $\mathcal{F}$ is a good normalized representative with a clean long pulley. See Figure~\ref{fig:CleanLongPulleys}. Then, for every $i \le n(\mathcal{F},\gamma)-1$, the foliation $\mathcal{F}$ is good wrt $\gamma_i$, and $y$ is the singular point on the boundary of the pulley wrt $\gamma_i$. Hence, if $L_x \neq [x,y]$, we have
    \[
        |\gamma_{i}\cap L_x| - |\gamma_{i+1} \cap L_x| = 1 \quad \text{for all}~ i \le n(\mathcal{F},\gamma)-1,
    \]
    which implies that $n(\mathcal{F},\gamma) = |L_x \cap \gamma|$. The case $L_x = [x,y]$ follows similarly.
\end{proof}

Suppose that $\Fcal$ is immediately $x$-non-generic wrt $\gamma$. Then $\Fcal$ and $\gamma$ can be transversely represented such that $\deg_{\Fcal}(x)\ge 2$.

In this case, we defined $\Phi(\Fcal)$ as the measured foliation obtained from $\Fcal$ by simply removing $x$ from the set of marked points. Hence, we have
\begin{equation}\label{eqn:non-gen}
    \pair{\Fcal}{[\gamma]}_{S_{g,n+1}}
    =\pair{\Fcal}{\gamma}_{S_{g,n+1}}
    =\pair{\Fcal}{\gamma}_{S_{g,n}}
    =\pair{\Phi(\Fcal)}{i_*[\gamma]}_{S_{g,n}},
\end{equation}
where $\pair{\Fcal}{\gamma}_{S_{g,n}}$ is computed on the puncture-filled surface $S_{g,n}$ using the same representatives of $\Fcal$ and $\gamma$, which remain transverse in $S_{g,n}$.
    
In this case, we can still define simple pulls along leaves emanating from $x$; there are $\deg_{\Fcal}(x)$ such leaves.  As in the $x$-generic case, these simple pulls are precisely the ones that preserve the intersection number with $\Fcal$, whereas all other simple pulls strictly increase the intersection number with $\Fcal$.

\begin{prop}\label{prop:PForget NonGeneric}
    Let $\Fcal$ be an $x$-non-generic measured foliation on $S_{g,n+1}$. Then, there exists $\Phi(\Fcal)\in \MF(S_{g,n})$ realizing the functional $\varphi(\Fcal)$, i.e., 
    $\pair{\Phi(\Fcal)}{[\zeta]}=\varphi(\Fcal)([\zeta])$ for all $[\zeta]\in C(S_{g,n})$. Moreover, for each pre-essential simple closed curve $\gamma$ on $S_{g,n+1}$, there exists $\gamma'\sim_{\pre}\gamma$ such that $\pair{\Fcal}{[\gamma']}_{S_{g,n+1}}=\pair{\Phi(\Fcal)}{i_*[\gamma]}_{S_{g,n}}$.
\end{prop}
\begin{proof}
    Suppose $\gamma$ is a pre-essential simple closed curve. By Lemma \ref{lem:non-gen_FiniteCanSimPulls}, after finitely many canonical simple pulls, we obtain $\gamma'$ with $\gamma'\sim_{\pre}\gamma$ such that $\Fcal$ is immediately $x$-non-generic wrt $\gamma'$. It follows from Equation \eqref{eqn:non-gen} that $\pair{\Fcal}{[\gamma']}_{S_{g,n+1}}=\pair{\Phi(\Fcal)}{i_*[\gamma]}_{S_{g,n}}$.
\end{proof}
    
\subsection{Generic measured foliations}
\begin{defn}
        Let $\Gamma$ be a multicurve of $S_{g,n}$. We say that $\Gamma$ {\em determines measured foliations} of $S_{g,n}$ if the intersection numbers with the elements of $\Gamma$ uniquely determine foliations, i.e., the map $\MF(S_{g,n}) \to (\Rbb_{\ge 0})^{|\Gamma|}$ sending $\Gcal \in \MF(S_{g,n})$ to $\left(\pair{[\gamma]}{\Gcal}\right)_{\gamma\in \Gamma}$ is injective.
\end{defn}

The existence of a multicurve determining measured foliations follows from the Dehn-Thurston coordinate \cite[Expos\'{e} 6.4]{FLP_ThurstonSurface}.

Recall that a multicurve $\Gamma$ on $S_{g,n}$ is filling if $\pair{[\Gamma]}{[\alpha]}>0$ for every essential simple closed curve $\alpha\subset S_{g,n}$. It is easy to show that if $\Gamma$ determines measured foliations, then $\Gamma$ is filling. 

\begin{lem}\label{lem:FillMF nonzero intersection}
    Suppose that $\Gamma$ is a filling multicurve and $\Fcal$ is a non-zero measured foliation on $S_{g,n}$. Then $\pair{\Fcal}{[\Gamma]}>0$.
\end{lem}
\begin{proof}
    If $\Fcal$ is minimal, then we obviously have $\pair{\Fcal}{[\Gamma]}>0$. In general, we consider the decomposition of $\Fcal$ into annulus and minimal components $\Fcal_1\cup \Fcal_2\cup \dots \cup \Fcal_k$. If there exists an annulus component $\Fcal_i$, then we have $\pair{\Fcal_i}{[\Gamma]}>0$ by the filling property of $\Gamma$. Suppose that there exists a minimal component $\Fcal_i$ supported in a subsurface $S_i\subset S$. Since it is filling, $\Gamma$ must intersect $S_i$. It follows from the minimality of $\Fcal_i$ in $S_i$, we obtain $\pair{\Fcal_i}{[\Gamma]}>0$.
\end{proof}

For $r>0$, we define a subset $\MF_{[\Gamma], B}(S_{g,n})$ of $\MF(S_{g,n})$ by
\[
    \MF_{[\Gamma], B}(S_{g,n}):= \{ \Fcal \in \MF(S_{g,n}) : \pair{[\Gamma]}{\Fcal}\le r\}.
\]
Recall that our definition $\MF(S_{g,n})$ contains the zero element $0$.
\begin{lem}\label{lem:MF_Gamma,B cpt}
    The set $\MF_{[\Gamma], B}(S_{g,n})$ is compact.
\end{lem}
\begin{proof}
    Suppose $\MF_{[\Gamma], B}(S_{g,n})$ is not compact. There exists a sequence $(\Fcal_k)$ in $\MF_{[\Gamma], B}(S_{g,n})$ such that $\Fcal_k \to \infty$, which is equivalent to that there exists $[\alpha]\in C(S_{g,n})$ such that
    \begin{equation}\label{eqn:MFtoInfty}
        \pair{\Fcal_k}{[\alpha]}\to \infty.
    \end{equation}
    Recall that $\Proj\from \MF(\hat{\Cbb}, P_f)\setminus \{0\} \to \PMF(\hat{\Cbb}, P_f)$ denotes the projectivization map. By substituting with a subsequence, the sequence $(\Proj(\Fcal_k))$ converges to $\Proj(\Gcal)$ for some $\Gcal\in \MF(S_{g,n})\setminus \{0\}$, and thus there exists a sequence $a_k>0$ so that $a_k\cdot \Fcal_k\to \Gcal$. By \eqref{eqn:MFtoInfty}, we have $a_k\to 0$ as $k\to \infty$. It follows that
    \[
        \pair{a_k\cdot \Fcal_k}{[\Gamma]}=a_k\pair{\Fcal_k}{[\Gamma]}\le a_kB|\Gamma|,
    \]
    where $|\Gamma|$ is the number of components in $\Gamma$. Letting $k\to \infty$ we obtain $\pair{\Gcal}{[\Gamma]}=0$, which contradicts Lemma \ref{lem:FillMF nonzero intersection}.
\end{proof}

Consider $\Gamma'$ and $([\psi_l])_{l\ge1}$ in Theorem \ref{thm:MulticurveReduction}-(7). Suppose further that $\Gamma'$ fills $S_{g,n+1}$. For any $\gamma'\in \Gamma'$, we have 
\[
    \pair{[\psi_k]^{-1}\cdot \Fcal}{[\gamma']} \to \varphi(\Fcal)([\zeta]).
\]
In particular, $\pair{[\psi_k]^{-1}\cdot \Fcal}{[\Gamma']}$ is uniformly bounded. It follows from Lemma \ref{lem:MF_Gamma,B cpt} that, after replaced by an appropriate subsequence, $[\psi_k]^{-1}\cdot \Fcal$ converges to a measured foliation $\Gcal$ with respect to the weak$^*$-topology. One can show that $\Gcal$ is $x$-non-generic; see Lemma \ref{lem:xNonGeneric}. Then we will define $\Phi(\Fcal)$ as $\Phi(\Gcal)$, which is defined by Proposition \ref{prop:PForget NonGeneric}.

The argument in the previous paragraph involves three issues. First, we need to find a pre-essential, pre-non-parallel, and pre-minimal multicurve $\Gamma$ of $S_{g,n+1}$ that fills $S_{g,n+1}$. Second, the multicurve $\Gamma'$ in Theorem \ref{thm:MulticurveReduction}-(7) may be different from $\Gamma$. Third, to show that $\phi(\Fcal)$ enjoys the desired property, $\pair{\Phi(\Fcal)}{[\zeta]}=\varphi(\Fcal)[\zeta]$ for all $[\zeta]\in C(S_{g,n})$, we need $i_*[\Gamma]$ determines measured foliations of $S_{g,n}$. We resolve these issues in Proposition \ref{prop:LimitMFfromIntReduction} by using Lemma \ref{lem:Interction_bound_pre-parallel} and  Proposition \ref{prop:Multicurv_Det_MF}.

\begin{lem}\label{lem:Interction_bound_pre-parallel}
    Fix an $x$-non-generic measured foliation $\Fcal\in \MF(S_{g,n+1})$. Suppose $\Gamma$ is a pre-essential and pre-non-parallel multicurve on $S_{g,n+1}$ such that $i_*[\Gamma]$ fills $S_{g,n}$. Then the following hold.
    \begin{enumerate}
        \item For any $[\Gamma']$ that is a simple pull of $[\Gamma]$, we have $\pair{\Fcal}{[\Gamma']}\le 3\pair{\Fcal}{[\Gamma]}$.
        \item Suppose further that $\Gamma$ is pre-minimal. Let $\gamma'$ be a simple closed curve on $S_{g,n+1}$ such that $\gamma \sim_{\pre} \gamma'$ for some $\gamma\in \Gamma$. Then there exists $C=C([\gamma'],[\Gamma])>0$ such that for all $[\psi] \in \Push(\pi_1(S_{g,n},x))$, we have
        \[
            \pair{[\psi] \cdot \Fcal}{[\gamma']} \le C \pair{[\psi] \cdot \Fcal}{[\Gamma]}.
        \]
        The same statement holds with $\gamma'$ being replaced by a multicurve $\Gamma'$ that is pre-homotopic to $\Gamma$.
    \end{enumerate}
\end{lem}
\begin{proof}
    (1) Let $\delta$ denote a simple cord from $x$ to $\Gamma$ such that $[\Pull_\delta(\Gamma)]=[\Gamma']$. There exist representatives of $\Gamma,\Gamma'$, and $\delta$ in their homotopy classes such that 
    \[
        2\pair{\Fcal}{\delta}=\left|\pair{\Fcal}{[\Gamma]}-\pair{\Fcal}{[\Gamma']}\right|.
    \]
    Since $i_*[\Gamma]$ fills $S_{g,n}$, the connected component $U$ of $S_{g,n}\setminus \Gamma$ containing $x$ and the interior of $\delta$ is a disk. We may assume that $\Fcal\pitchfork\Gamma$ and $\Fcal\pitchfork\delta$. Then $\overline{U}$ contains the pulley $W_{\Fcal,\gamma}$, where $\gamma$ is the component of $\Gamma$ that is simply pulled. It follows that
    \[
        \pair{\Fcal}{\delta}\le\pair{\Fcal}{\partial U} \le  \pair{\Fcal}{\Gamma}.
    \]
    Hence, by Equation \eqref{eqn:Int_Num_Change} and Lemma \ref{lem:IntNumbCanonPull}, we have $\pair{\Fcal}{[\Gamma']}\le 3\pair{\Fcal}{[\Gamma]}$.

    (2) Let $[Z]=i_*[\Gamma]$. Recall that $p_i\from X_Z \to X_{\zeta_i}$ is surjective (Lemma \ref{lem:p_i Surj, Efficent}). Hence, there exists a pre-minimal multicurve $\Gamma'$ on $S_{g,n+1}$ such that $\Gamma\sim_{\pre}\Gamma'$ and $\gamma'\in \Gamma'$. Consider $[\Gamma]$ and $[\Gamma']$ as vertices of $X_Z$, and define $N:=d_{X_Z^{(1)}}([\Gamma],[\Gamma'])$. Since $\pi_1(S_{g,n},x)$ acts on $X_Z$ as cubical automorphisms, we have
    \[
        \pair{\Fcal}{[\phi]^{-1} \cdot [\gamma']}\le \pair{\Fcal}{[\phi]^{-1}\cdot [\Gamma']} \le 3^N \pair{\Fcal}{[\phi]^{-1}\cdot [\Gamma]}.
    \]
    We use (1) for the second inequality. The statement with $\gamma'$ being replaced by $\Gamma'$ also follows by the same argument.
\end{proof}

\begin{prop}\label{prop:Multicurv_Det_MF}
    There exist a pre-essential, pre-non-parallel, and pre-minimal multicurve $\Gamma$ on $S_{g,n+1}$ such that $i_*[\Gamma]$ determines measured foliations of $S_{g,n}$.
\end{prop}
\begin{proof}
    By using a multicurve on $S_{g,n}$ that determines measured foliations on $S_{g,n}$, we can obtain a multicurve $\Gamma$ on $S_{g,n+1}$ that is pre-essential and pre-non-parallel such that $i_*[\Gamma]$ determines measured foliations of $S_{g,n}$. Then, the desired conclusion follows from Lemma~\ref{lem:Make_pre-minimal}. 
\end{proof}

\begin{lem}\label{lem:Make_pre-minimal}
    Suppose that $\Gamma$ is a pre-essential and pre-non-parallel multicurve on $S_{g,n+1}$. Then, there is a multicurve $\Gamma'$ pre-homotopic to $\Gamma$ such that $\Gamma'$ is pre-minimal.
\end{lem}
\begin{proof}
    The desired multicurve $\Gamma'$ can be obtained by taking a homotopy of $\Gamma$ in the puncture-filled surface $S_{g,n}$ to remove all bigons.
\end{proof}

\begin{defn}[Doubly-filling multicurve]
    A multicurve $Z$ on $S_{g,n}$ is {\it doubly-filling} if $\pair{[\alpha]}{[Z]}\ge 2$ for every essential arc $\alpha$ in $S_{g,n}$. 
\end{defn}

\begin{lem}
    There exists a doubly-filling multicurve $Z$ on $S_{g,n}$.
\end{lem}
\begin{proof}
    Fix a filling multicurve $Z_0$ of $S_{g,n}$. By adding more simple closed curves to $Z_0$ if necessary, we can assume that for any distinct marked points $y_1$ and $y_2$, the multicurve $Z_0$ contains at least two simple closed curves separating $y_1$ and $y_2$. Then, we have $\pair{[Z]}{[\alpha]}\ge 2$ for every arc $\alpha$ joining $y_1$ and $y_2$.

    Next, we consider arcs whose endpoints are the same. Let $y$ be a marked point, by adding more simple closed curves to $Z_0$ if necessary, we may assume that there exists a representative $Z_0'$ of $[Z_0]$ such that the component of $S_{g,n}\setminus Z_0'$ containing $y$ is a bigon, formed by $\zeta_1,\zeta_2 \in Z_0'$. We have $\zeta_1\neq \zeta_2$ because they are simple closed curves. Then, for every arc $\alpha$ joining $y$ to itself, we have $\pair{[\alpha]}{[\zeta_1]\cup [\zeta_2]}\ge 2$.
\end{proof}

\begin{lem}\label{lem:Isolated non-filling vertices}
    Let $Z$ be a doubly-filling multicurve in $S_{g,n}$. Suppose that $[\Gamma]$ is a vertex in $X_Z$ that is not filling in $S_{g,n+1}$. Then, all its adjacent vertices in $X_Z$ are filling multicurves.
\end{lem}
\begin{proof}
    Suppose $Z$ is a doubly-filling multicurve in $S_{g,n}$. Let $[\Gamma'] \in \Vertex(X_Z)$. Since $\Gamma'$ is pre-filling, the connected component of $S_{g,n+1} \setminus \Gamma'$ containing $x$ is either (a) a once-punctured disk, in which case $\Gamma'$ is filling, or (b) a twice-punctured disk, in which case $\Gamma'$ is not filling.

    Suppose $[\Gamma_1]$ and $[\Gamma_2]$ are adjacent vertices in $X_Z$ both of which are not filling. We may assume that $\Gamma_1$ and $\Gamma_2$ are representatives of $[\Gamma_1]$ and $[\Gamma_2]$ such that $\Gamma_2$ is a simple pull $\Pull_\delta(\Gamma_1)$ of $\Gamma_1$ along a cord $\delta$. For each $i\in \{1,2\}$, let $y_i(\neq x)$ be the marked points in the component of $S_{g,n+1}\setminus \Gamma_i$ containing $x$. Then, the simple cord $\delta$ can be extended to an arc $\alpha$ joining $y_1$ and $y_2$ such that $\pair{[\alpha]}{[\Gamma_1]}=1$. It contradicts that $Z$ is doubly-filling.
\end{proof}

\begin{lem}\label{lem:xNonGeneric}
    Let $\Gcal$ be a measured foliation on $S_{g,n+1}$. If there exists a pre-filling multicurve $\Gamma$ of $S_{g,n+1}$ such that
    \[
        \pair{\Gcal}{[\gamma]}=\inf_{\gamma \sim_{\pre} \alpha} \pair{\Gcal}{[\alpha]}
    \]
    for any $\gamma \in \Gamma$, then $\Gcal$ is $x$-non-generic. Moreover, if $\Gamma$ is filling, then $\Gcal$ is $x$-non-generic of non-annular type.
\end{lem}
\begin{proof}
    Suppose that $\Gcal$ is $x$-generic. Let $\Gcal'$ be the minimal component of $\Gcal$ such that $S':={\rm supp}(\Gcal')$ contains $x$ its interior with $\iota\from S'\to S_{g,n+1}$; see Lemma \ref{lem:minimal component of x}.  Either (1) there exists $\gamma \in \Gamma$ such that $\gamma$ intersects an essential simple closed curve $\alpha$ in $S'$ or (2) $S'$ is disjoint from $\Gamma$.

    If case (1) holds, then $\pair{\Gcal}{[\gamma]} > \inf_{\gamma \sim_{\pre} \alpha} \pair{\Gcal}{[\alpha]}$, which is a contradiction.
    
    Since $\Gamma$ is pre-filling, (2) implies that the image $\iota(S')$ in $S$ is a disk with two marked points. Then, $\Gcal'$ is an annulus component, which looks like Figure \ref{fig:Forget_nondeg}-(C). Then, $\Gcal$ is not $x$-generic.

    If $\Gcal$ is $x$-non-generic of annular type, then $\Gcal'$ is the annulus component described in the previous paragraph. In this case, there exists a leaf $[x,y]$ joining two marked points such that $[x,y]$ does not intersect $\Gamma$, which contradicts the assumption that $\Gamma$ is filling. 
\end{proof}

\begin{prop}\label{prop:LimitMFfromIntReduction}
    Let $\Fcal$ be an $x$-generic minimal measured foliation on $S_{g,n+1}$. Suppose that $\Gamma$ is a pre-essential, pre-non-parallel, and pre-minimal multicurve of $S_{g,n+1}$ such that $[Z]:=i_*[\Gamma]$ is doubly-filling in $S_{g,n}$. Then, there exist a filling multicurve $\Gamma'$ with $\Gamma\sim_{\pre}\Gamma'$ and a sequence of point-pushing mapping classes $([\psi_l])_{l\ge 1}$ satisfying the following.
    \begin{enumerate}
        \item $[\psi_l]\cdot \Fcal \to \Gcal$ in $\MF(S_{g,n+1})$ and $\Gcal$ is $x$-non-generic of non-annular type.
        \item $\pair{[\gamma_i]}{\Gcal} = \inf\limits_{\gamma_i \sim_{\pre} \alpha} \pair{[\alpha]}{\Gcal}=\inf\limits_{\gamma_i \sim_{\pre} \alpha} \pair{[\alpha]}{\Fcal}$ for every $\gamma_i\in\Gamma'$.
    \end{enumerate}
    Moreover, if $i_*[\Gamma]$ determines measured foliations of $S_{g,n}$, then for any pre-essential simple closed curve $\gamma\in C(S_{g,n+1})$, we have
    \[
        \pair{\Phi(\Gcal)}{i_*[\gamma]}=\inf_{\gamma\sim_{\pre}\alpha}\pair{\Fcal}{[\alpha]},
    \]
    where $\Phi(\Gcal)$ is a measured foliation on $S_{g,n}$ in Proposition \ref{prop:PForget NonGeneric}.
\end{prop}
\begin{proof}
    For notational simplicity, we use $\varphi$ defined in Equation \eqref{eqn:PForget Functional} by
    \[
        \varphi(\Fcal)([\alpha]):=\inf_{i_*[\gamma]=[\alpha]} \pair{\Fcal}{[\gamma]}.
    \]
    
    By applying Theorem \ref{thm:MulticurveReduction}-(7) for $\Gamma$, we obtain a sequence of point-pushing mapping classes $([\psi_k])_{l\ge1}$ and $\Gamma'$ with $\Gamma\sim_{\pre} \Gamma'$ such that for each $\gamma \in \Gamma'$, we have
    \[
        \pair{[\psi_l] \cdot \Fcal}{[\gamma]}\to \varphi(\Fcal)(i_*[\gamma]),
    \]
    as $l \to \infty$.
    In particular, we may assume that $\Gamma'$ is filling. Indeed, since $Z$ is doubly filling, Lemma \ref{lem:Isolated non-filling vertices} implies that the projection of the infinite directed edge-path $[\Gamma(l)]$, in Theorem~\ref{thm:MulticurveReduction}-(7), to the compact quotient $X_Z/\pi_1(S_{g,n},x)$ passes through a vertex of the form $[\Gamma']\cdot \pi_1(S_{g,n})$ infinitely often, where $\Gamma'$ is filling.
    It follows that $\pair{[\psi_l] \cdot \Fcal}{[\Gamma']}$ is uniformly bounded, and so is $\pair{[\psi_l] \cdot \Fcal}{[\Gamma]}$ by Lemma \ref{lem:Interction_bound_pre-parallel}. Then, by Lemma \ref{lem:MF_Gamma,B cpt}, after passing to a subsequence of $(\psi_l)$, we obtain $[\psi_l]\cdot \Fcal\to \Gcal$ for some $\Gcal\in \MF(S_{g,n+1})$.

    Next, we show that $\Gcal$ is $x$-non-generic of non-annular type. For every $\gamma_i \in \Gamma'$, we have
    \begin{align*}
        \varphi(\Gcal)(i_*[\gamma_i])&=\inf_{\alpha\sim_{\pre}\gamma_i} \pair{\Gcal}{[\alpha]}\\
        &= \inf_{\alpha\sim_{\pre}\gamma_i} \lim_{l\to \infty} \pair{[\psi_l]\cdot \Fcal}{[\alpha]}\\
        &= \inf_{\alpha\sim_{\pre}\gamma_i} \lim_{l\to \infty} \pair{\Fcal}{[\psi_l]^{-1}\cdot[\alpha]}\\
        &\ge\varphi(\Fcal)(i_*[\gamma_i]),
    \end{align*}
    and
    \begin{align*}
        \varphi(\Gcal)(i_*[\gamma_i])&=\inf_{\alpha\sim_{\pre}\gamma_i} \pair{\Gcal}{[\alpha]}\\
        &\le \pair{\Gcal}{[\gamma_i]}\\
        &=\lim_{l\to \infty} \pair{[\psi_l]\cdot \Fcal}{[\gamma_i]}\\
        &=\varphi(\Fcal)(i_*[\gamma_i]).
    \end{align*}
    Therefore, we obtain
    \[
        \varphi(\Gcal)(i_*[\gamma_i])=\pair{\Gcal}{[\gamma_i]}=\varphi(\Fcal)(i_*[\gamma]).
    \]
    Then, since $\Gamma'$ is filling, Lemma \ref{lem:xNonGeneric} implies $\Gcal$ is $x$-non-generic of non-annular type.
    
    Suppose further that $i_*[\Gamma]$ determines measured foliations of $S_{g,n}$. Fix a pre-essential simple closed curve $\gamma'$ on $S_{g,n+1}$ with $i_*[\gamma'] \not\subset i_*[\Gamma]$, and define $\Gamma_1 := \Gamma \cup \{\gamma'\}$. We can modify $\gamma'$, keeping $i_*[\gamma']$ fixed, so that $\Gamma_1$ becomes pre-minimal; more precisely, we remove bigons between $\gamma'$ and $\Gamma$ in $S_{g,n}$ by changing $\gamma'$ only. Applying the previous argument (applied to $\Gamma$) to $\Gamma_1$, we obtain $\Gcal_1$ and $\Gamma_1'$ such that for any $\gamma \in \Gamma_1'$,
    \[
        \varphi(\Gcal_1)(i_*[\gamma]) = \pair{\Gcal_1}{[\gamma]} = \varphi(\Fcal)(i_*[\gamma]).
    \]
    Since $i_*[\Gamma]$, which is a sub-multicurve of $i_*[\Gamma_1']$, determines measured foliations of $S_{g,n}$, we conclude that $\Phi(\Gcal) = \Phi(\Gcal_1)$. In particular, we have
    \[
        \varphi(\Gcal)(i_*[\gamma']) = \varphi(\Gcal_1)(i_*[\gamma']) = \varphi(\Fcal)(i_*[\gamma']). \qedhere
    \]
\end{proof}

\subsection{Proof of the first part of Theorem \ref{theorem:MF}:  Realization, (1), and (2)}\label{subsec:Proof of Thm1 I}

\subsubsection{Proof for the case when $\Fcal$ is minimal and $x$-generic}
Consider $\Gcal$ in Proposition \ref{prop:LimitMFfromIntReduction}. By defining $\Phi(\Fcal)$ as $\Phi(\Gcal)\in \MF(S_{g,n})$, we obtain the equation $\pair{\Phi(\Fcal)}{[\zeta]}=\varphi(\Fcal)([\zeta])(:=\inf_{i_*[\gamma]=[\zeta]} \pair{\Fcal}{[\gamma]})$ for every $[\zeta]\in C(S_{g,n})$.

The map $\varphi(-)([\zeta])\from \MF(S_{g,n+1})\to \mathbb{R}$ is upper-semi continuous because it is an infimum of continuous functions
\[
    \left\{\pair{-}{[\gamma]}\from \MF(S_{g,n+1})\to \mathbb{R}\right\}_{[\gamma]\in i_*^{-1}([\zeta])}.
\]

Let us prove (2). We first show that the unstable measured foliations of point-pushing pseudo-Anosov mapping classes lie in $\Phi^{-1}(0)$. 
Let $[\psi]$ be a point-pushing pseudo-Anosov mapping class, and let $\Fcal_\psi$ denote its unstable measured foliation satisfying $[\psi]^{-1} \cdot \Fcal_\psi = \lambda^{-1} \cdot \Fcal_\psi$ for some $\lambda > 1$. For every $[\gamma] \in C(S_{g,n+1})$, we have
\[
    \pair{\Fcal_\psi}{[\psi]^k \cdot [\gamma]}
    = \pair{[\psi]^{-k} \cdot \Fcal_\psi}{[\gamma]}
    = \lambda^{-k} \pair{\Fcal_\psi}{[\gamma]}
    \to 0
\]
as $k \to \infty$.  Note that $\psi^k \circ \gamma$ is pre-homotopic to $\gamma$ for all $k \in \mathbb{Z}$.  Since the stable measured foliation of $\psi$ is the unstable measured foliation of $\psi^{-1}$, we have that all invariant measured foliations of point-pushing pseudo-Anosov mapping classes are in $\Phi^{-1}(0)$.

It is well-known that stable measured foliations of point-pushing pseudo-Anosov mapping classes are dense in $\MF(S_{g,n+1})$. A proof is as follows. The mapping class group $\MCG^+(S_{g,n+1})$ action on the space of projective measured foliations $\PMF$ is minimal \cite[Theorem 6.19]{FLP_ThurstonSurface},
so every point in $\PMF$ has a dense orbit. 
Let $[\psi]\in \Push(\pi_1(S_{g,n},x))$ be such that $[\psi]$ is pseudo-Anosov. 
For all $[\phi]\in \MCG^+(S_{g,n+1})$,
we have $[\phi]\cdot \Fcal_\psi=\Fcal_{\phi\circ \psi\circ\phi^{-1}}$ and 
$[\phi\circ \psi\circ\phi^{-1}]$ is a point-pushing pseudo-Anosov mapping class.

To complete the proof of (2), it suffices to show that $\Phi^{-1}(0)$ is a countable intersection of open dense subsets. Fix a multicurve $Z$ of $S_{g,n}$ that determines measured foliations of $S_{g,n}$. We have
\[
    \Phi^{-1}(0)=\bigcap_{m\in \mathbb{Z}^+,\,[\zeta]\in Z} \left\{\Fcal~:~\varphi(\Fcal)([\zeta])<1/m\right\}.
\]
Hence, $\Phi^{-1}(0)$ is a countable intersection of open subsets. \qed

\subsubsection{Reduction of the non-minimal case to the minimal case}\label{subsubsec:nonminimal_reduction}
Fix an $x$-generic measured foliation $\Fcal\in \MF(S_{g,n+1})$. Consider its decomposition $\Fcal_1\cup \Fcal_2\cup \dots \cup \Fcal_k$ by minimal and annulus components and their supports $S_i={\rm supp}(\Fcal_i)$'s. Suppose that $\gamma$ is an essential simple closed curve in $S_{g,n+1}$ that is in a minimal position with respect to $\Fcal$. For each $i$, we define a multi-arc $\gamma_i:= \gamma \cap S_i$ on $S_i$, which is a collection of arcs joining boundary components of $S_i$. The homotopy classes of $[\gamma_i]$'s are well-defined. We have
\[
    \pair{\Fcal}{[\gamma]}_S=\sum_{i=1}^k\pair{\Fcal_i}{[\gamma_i]}_{S_i},
\]
where the subscripts $S$ and $S_i$ indicate the surfaces where the intersection numbers are considered.

Without loss of generality, assume that $x \in {\rm int}(S_1)$; see Lemma~\ref{lem:minimal component of x}. We say that $\gamma$ is {\it combinatorially reduced with respect to $(\Fcal, x)$} if ${\rm int}(S_1) \setminus \gamma_1$ has no component that is a once–punctured disk punctured at $x$. If $\gamma$ is not combinatorially reduced, then $\gamma_1$ contains an arc component $\alpha$ whose endpoints lie on the same boundary component of $S_1$, such that $\alpha$ together with an arc $\beta \subset \partial S_1$ bounds a disk punctured at $x$. Replacing $\alpha$ in $\gamma$ by $\beta$ is a simple pull, and this operation removes the arc component $\alpha$ from $\gamma_1$.
By applying this simple pull finitely many times, we can reduce any $\gamma \in C(S_{g,n+1})$ to a combinatorially reduced form.

Suppose that $\gamma$ is combinatorially reduced. Then, we have
\begin{equation}\label{eqn:non-minimal}
    \inf_{\gamma\sim_{\pre} \gamma'} \pair{\Fcal}{[\gamma']}_{S_{g,n+1}}=\inf_{\gamma_1\sim_{\pre}\gamma_1'} \pair{\Fcal_1}{[\gamma_1']}_{S_1}+\sum_{i=2}^k\pair{\Fcal_i}{[\gamma_i]}_{S_i}.
\end{equation}
We sketch a proof of \eqref{eqn:non-minimal}.
The sequence of canonical simple pulls starting from $\gamma$ affects only $\gamma_1$, leaving $\gamma_i$ fixed for all $i \neq 1$, and the changes to $\gamma_1$ coincide with those obtained by applying canonical simple pulls directly to $\gamma_1$.
Although we have not rigorously discussed canonical reductions for properly embedded arcs in surfaces with non-empty boundary, the argument can be carried out in a manner analogous to the case of simple closed curves, and we omit the details.
Alternatively, one may collapse boundary components to punctures and consider simple closed curves bounding tubular neighborhoods of the arcs; see the related discussion in the next paragraph.

We define
    \[
        \Phi(\Fcal):=\Phi(\Fcal_1)\cup \Fcal_2\cup \dots \cup \Fcal_k.
    \]
Consider the punctured surface $\widehat{S}_1$ obtained by collapsing each boundary component of $S_1$ to a puncture. 
Then $\Fcal_1$ induces an $x$-generic minimal measured foliation $\widehat{\Fcal}_1$, and $\gamma_1$ induces a multi-arc $\widehat{\gamma}_1$ joining punctures in $\widehat{S_1}$. 
Hence it suffices to show the following: for an $x$-generic minimal $\Fcal \in \MF(S_{g,n+1})$ and a pre-essential arc $\alpha$, we have 
\begin{equation}\label{eqn:Phi Arcs}
    \Phi(\Fcal)(i_*[\alpha]) = \inf_{\alpha' \sim_{\pre} \alpha} \pair{\Fcal}{[\alpha']}.
\end{equation}
An arc joining punctures is called {\it pre-essential} if it does not bound a disk in $S_{g,n}$. 
If $\alpha$ joins distinct punctures, let $\gamma$ denote the boundary of a small tubular neighborhood of $\alpha$. 
Then, for all $\Fcal \in \MF(S)$, we have 
\[
    2\pair{\Fcal}{[\alpha]} = \pair{\Fcal}{[\gamma]}.
\]
When $\alpha$ joins the same puncture, we take one boundary component $\gamma$ of a tubular neighborhood of $\alpha$, and then 
\[
    \pair{\Fcal}{[\alpha]} = \pair{\Fcal}{[\gamma]}.
\]
Hence Equation~\eqref{eqn:Phi Arcs} for arcs $\alpha$ follows from the case of simple closed curves.

\subsection{Forgetting more than one puncture}
The following proposition implies that an $m$-puncture forgetting map $\MF(S_{g,n+m})\to \MF(S_{g,n})$ is well-defined by the $m$-iteration of one-puncture forgetting maps for measured foliations and independent of the order of forgetting.

\begin{prop}\label{prop:m Punctures}
    Consider the case when we fill $m>1$ punctures via an embedding $i\from S_{g,n+m}\to S_{g,n}$. Label by $x_1,x_2,\dots,x_m$ the $m$ punctures that are filled. Denote by $i_k\from S_{g,n+m-k+1}\to S_{g,n+m-k}$ the one-puncture-filling map, filling $x_k$, for $1\le k\le m$, such that $i=i_m\circ i_{m-1}\circ \dots \circ i_1$. Denote by $\Phi_k\from \MF(S_{g,n+m-k+1})\to \MF(S_{g,n+m-k})$ the puncture forgetting map induced by $i_k$ for $1\le k\le m$. Define $\Phi\from \MF(S_{g,n+m})\to \MF(S_{g,n})$ by $\Phi:=\Phi_m\circ \Phi_{m-1}\circ \dots \circ \Phi_1$. Then, for any $\Fcal\in \MF(S_{g,n+m})$ and $\gamma\in C(S_{g,n})$, we have
    \[
        \pair{\Phi(\Fcal)}{[\gamma]}=\inf_{ \substack{i_*[\alpha]=[\gamma] \\ [\alpha] \in C(S_{g,n+m})}}  \pair{\Fcal}{[\alpha]} 
    \]
\end{prop}
The proof is left to the reader.


\section{Extremal lengths and forgetting maps for measured foliations}
Throughout, we denote by $\pi\from \Tcal_{g,n+1} \to \Tcal_{g,n}$ the puncture forgetting map between Teichm\"uller spaces. In this section, we prove Theorems \ref{thm:InequalityEL} and \ref{thm:LowBddTeichDist}, which establish Theorem \ref{theorem:MF}-(3,4). We also prove Theorem 
\ref{thm:Im_horospheres_}, which classifies the $\pi$-images of horospheres.

\begin{thm}\label{thm:InequalityEL}
Consider $X \in \Tcal_{g,n+1}$ and non-zero $\Fcal \in \MF(S_{g,n+1})$. Then, we have
    \[
        \EL(\Phi(\Fcal);\pi(X))=\inf_{X'\in \pi^{-1}(\pi(X))}\EL(\Fcal;X').
    \]
The infimum is attained as the minimum if and only if $\Fcal$ is $x$-non-generic of non-annular type.
\end{thm}

Recall the definitions of horoballs and horospheres in Equation \eqref{eqn:horo_ball_sphere}. As a corollary of Theorem \ref{thm:InequalityEL}, we obtain the following classification of $\pi$-images of horospheres.
\begin{thm}\label{thm:Im_horospheres_}
       For all $r>0$ and non-zero $\Fcal\in \MF(S_{g,n+1})$, the following hold.
       \begin{enumerate}
            \item If $\Phi(\Fcal)= 0$, then
                \[
                    \pi\left(S_h(\Fcal,r)\right)=\Tcal_{g,n}.
                \]
            \item If $\Phi(\Fcal)\neq 0$ and $\Fcal$ is $x$-generic or $x$-non-generic of annular type, then
                \[
                    \pi\left(S_h(\Fcal,r)\right)=B_h\left(\Phi(\Fcal),r\right).
                \]
            \item If $\Phi(\Fcal)\neq0$ and $\Fcal$ is $x$-non-generic of non-annular type, then 
                \[
                    \pi\left(S_h(\Fcal,r)\right)=\overline{B_h\left(\Phi(\Fcal),r\right)}.
                \]
       \end{enumerate}
\end{thm}

\begin{proof}
    It is straightforward to show that for all $Y\in \Tcal_{g,n}$, we have 
    \[
        \sup_{X \in \pi^{-1}(Y)} \EL(\mathcal{F}; X) = \infty.
    \]
    Hence, $Y\in \pi(S_h(\Fcal,r))$ if and only if there exists $X\in \pi^{-1}(Y)$ with $\EL(\Fcal;X)\le r$. Then the desired conclusion follows from Theorem \ref{thm:InequalityEL}.  
\end{proof}

Suppose that $X$ is a Riemann surface and $\Fcal$ is a measured foliation on $X$. Fix a conformal metric $\sigma$ on $X$. The length of $\Fcal$ with respect to $(X,\sigma)$ is defined by
\[
    L(\Fcal;(X,\sigma)):=\inf_{\Fcal\sim \Fcal'} \int_X \rho\times \alpha_F,
\]
where $\alpha_\Fcal$ denotes the transverse measure of $\Fcal$ and $\rho\times \alpha_F$ is locally the product of $\alpha_\Fcal$ and the $\sigma$-length measure along leaves of $\Fcal$. The infimum is taken over the equivalence class of $\Fcal$. The {\it extremal length} of $\Fcal$ with respect to $X$ is defined by
\[
    \EL(\Fcal;X):= \sup_{\sigma} \frac{L(\Fcal;(X,\sigma))^2}{\Area(X,\sigma)},
\]
where the supremum is taken over conformal metrics $\sigma$ on $X$ and $\Area(X,\sigma)$ is the area of $X$ with respect to $\sigma$.

Denote by $\mathcal{QT}_{g,n}$ the holomorphic vector bundle over $\Tcal_{g,n}$ whose fiber over $X\in\Tcal_{g,n}$ is the set $Q(X)$ of quadratic differentials on $X$. Let $\mathcal{QT}_{g,n}^\times$ be the collection of non-zero quadratic differentials. By \cite{HubbardMasur_QD}, we have a homeomorphism $\MF(S_{g,n})\times \Tcal_{g,n}\to \mathcal{QT}_{g,n}^\times$ sending $(\Fcal,X)$ to $q_{\Fcal,X}$ such that the horizontal foliation of $q_{\Fcal,X}$ is equivalent to $\Fcal$. The singular Euclidean metric $\sigma_{\Fcal,X}$ induced from $q_{\Fcal,X}$ is an {\it extremal metric} of $\Fcal$, i.e., 
\[
    \EL(\Fcal;X)=\frac{L(\Fcal;(X,\sigma_{\Fcal,X}))^2}{\Area(X,\sigma_{F,X})}
\]

The following lemma is well-known. See \cite{GM_EL_Teich} and \cite[Lemma 5.1]{Minsky_TeichGeod_EndsHyp3}.

\begin{lem}\label{lem:IneqBtwELInt}
    For any measured foliations $\Fcal$ and $\Gcal$ on $S_{g,n}$ and any $X\in \Tcal_{g,n}$, we have
    \[
        \EL(\Fcal;X) \EL(\Gcal;X) \ge \pair{\Fcal}{\Gcal}^2.
    \]
    The equality holds if and only if $\Fcal$ and $\Gcal$ are the horizontal and the vertical foliations of some quadratic differential $q\in Q(X)$, respectively.
\end{lem}


Recall that there are three possibilities (A), (B), and (C) for $x$-non-genericity, as illustrated in Figure~\ref{fig:Forget_nondeg}, such that (A) and (B) are of non-annular type and (C) is of annular type.

\begin{lem}[Theorem \ref{thm:InequalityEL} for non-generic case]\label{lem:InequalityELNonGeneric}
    Let $\Fcal$ be an $x$-non-generic measured foliation on $S_{g,n+1}$. For each $X\in \Tcal_{g,n+1}$, we have
    \[
        \EL(\Phi(\Fcal);\pi(X))=\inf_{X'=\pi^{-1}(\pi(X))} \EL(\Fcal;X')
    \]
    Moreover, there exists $X'\in \pi^{-1}(\pi(X))$ such that $\EL(\Fcal;X')=\EL(\Phi(\Fcal);\pi(X))$ if and only if $\Fcal$ is $x$-non-generic of non-annular type.
\end{lem}
\begin{proof} We analyze the two types separately.

    {\it Case 1: $\Fcal$ is $x$-non-generic of non-annular type.}\quad  Since $X$ and $\pi(X)$ are the same Riemann surfaces with different sets of marked points, they share the same set of conformal metrics. For each conformal metric $\sigma$, we have
    \[
        L(\Fcal; (X,\sigma)) \ge L(\Phi(\Fcal); (\pi(X),\sigma)),
    \]
    since the infimum is taken over a larger homotopy class for $\pi(X)$ than for $X$.
    Hence, we obtain $\EL(\Fcal;X)\ge \EL(\Phi(\Fcal);\pi(X))$. 
    
    It is easy to show that the equality $\EL(\Fcal; X') = \EL(\Phi(\Fcal);\pi(X))$ holds if and only if the extra point $x$ on $X'$ is not a pole of $q_{X',\Fcal}$ (Case~(A)). We can show the existence of such an $X'$ as follows.
    
    Consider the horizontal measured foliations $\Fcal_1$ of $q_{\Fcal,X}$ and $\Fcal_2$ of $q_{\Phi(\Fcal),\pi(X)}$. If the extra marked point $x$ is not a one-prong singularity of $\Fcal_1$ (Case (A) in Figure \ref{fig:Forget_nondeg}), then we take $X'=X$. Suppose $x$ is a one-prong singularity (Case (B) in Figure \ref{fig:Forget_nondeg}). Then $x$ is an endpoint of a segment leaf $[x,y]$ where $y$ is a singular point with $\deg_\Fcal(y)\ge2$. Let $\Fcal_1'$ be the measured foliation obtained by collapsing $[x,y]$ from $\Fcal_1$. Then $\Fcal_1'$ with $x$ forgotten and $\Fcal_2$ are equivalent measured foliations in the puncture-filled surface $S_{g,n}$. Hence, there is a continuous deformation from $\Fcal_1'$ to $\Fcal_2$. Let us consider a path $(x(t))_{t\in[0,1]}$ that traces the location of the point $x$ on $\Fcal_1$ as it undergoes two consecutive continuous deformations, $\Fcal_1 \rightsquigarrow \Fcal_1' \rightsquigarrow \Fcal_2$. The deformation $\Fcal_1' \rightsquigarrow \Fcal_2$ may involve Whitehead moves related to $x$. In this case, there is more than one way to choose the path $x(t)$, and we take any of them; see Figure \ref{fig:WHmove}.
    \begin{figure}[h!]
	   \centering
        \def\svgwidth{0.7\textwidth}
        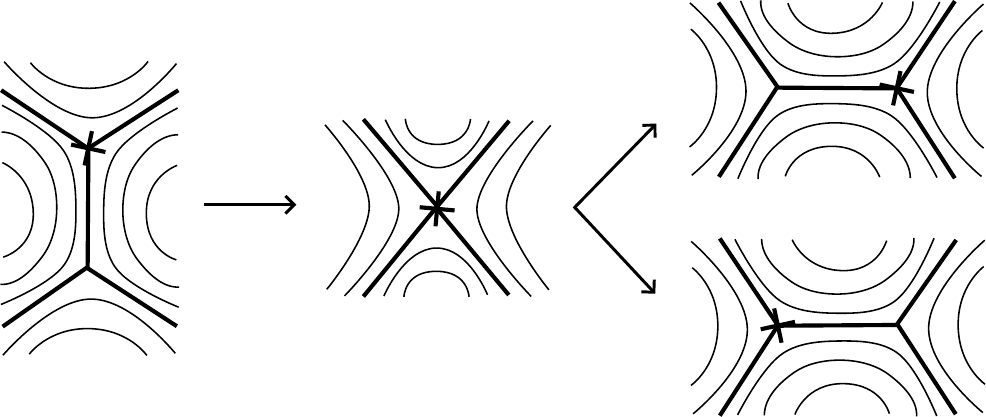
	\caption{}
    \label{fig:WHmove}
    \end{figure}
    
    Then we obtain $X'$ by moving the extra marked point $x$ on $X$ along the path $(x(t))$. It is possible that the path $(x(t))$ passes through or ends at a marked point. In this case, we take a small detour to make the path $x(t)$ avoid all the marked points. For this $X'$, $q_{\Fcal,X'}$ does not have a pole at the extra marked point $x$, and it coincides with $q_{\phi(\Fcal),\pi(X)}$ as quadratic differentials on $\pi(X)$.

    {\bf (Step 2)} Suppose that $\Fcal$ is $x$-non-generic of annular type.
    
    Then $\Fcal=\Fcal_1\cup \Fcal_2$ such that $\Fcal_2$ is the annulus component around $x$ and $y$, where $[x,y]$ is a saddle connection. Note that $\Fcal_1$ itself can be considered as an $x$-non-generic measured foliation of non-annular type. Then, we have $\Phi(\Fcal)=\Phi(\Fcal_1)$. Hence,
    \[
        \EL(\Fcal;X)\ge \EL(\Fcal_1;X)+\EL(\Fcal_2;X)>\EL(\Fcal_1;X)\ge \EL(\Phi(\Fcal);\pi(X)).
    \]
    Therefore, $\inf_{X'\in \pi^{-1}(\pi(X'))} \EL(\Fcal;X')\ge \EL(\Phi(\Fcal);\pi(X))$ and there is no $X'$ with $\EL(\Fcal;X')= \EL(\Phi(\Fcal);\pi(X))$.

    Next, we show that $\EL(\Phi(\Fcal);\pi(X)) \ge \inf_{X' \in \pi^{-1}(\pi(X))} \EL(\Fcal; X')$. Fix $X_0\in \pi^{-1}(\pi(X))$. For the quadratic differential $q_{\Fcal_1,X_0}$, let $\sigma_1$ be its singular Euclidean structure and $\Hcal$ its horizontal measured foliation. Denote by $x_0$ the marked point we forget in $X_0$. We may choose $X_0$ so that $x_0$ lies sufficiently close to $y$, ensuring that $\Hcal$ admits a saddle connection $[x_0,y]$. Define a path $(x_t)_{t\in[0,1)}$ supported in $[x_0,y]$ such that $x_t\to y$ as $t\to 1$, and let $(X_t)$ be the corresponding path in $\pi^{-1}(\pi(X))$. Then, we have $\EL(\Fcal_1;X_t) = \EL(\Phi(\Fcal);\pi(X))$ for each $t \in [0,1)$, and $\EL(\Fcal_2; X_t) \to 0$ as $t \to 1$. Hence, by the convexity of $\sqrt{\EL}$, i.e., 
    \[
        \sqrt{\EL(\Fcal;X_t)}\le \sqrt{\EL(\Fcal_1;X_t)} +\sqrt{\EL(\Fcal_2;X_t)},
    \]
    we obtain $\inf_{X' \in \pi^{-1}(\pi(X))} \EL(\Fcal; X') \le \EL(\Phi(\Fcal);\pi(X))$.
\end{proof}

The gradient descent flow $(X_t)_{t\ge0}$ of the extremal length function
$\EL(\Fcal;\cdot)\colon \pi^{-1}(Y)\to\mathbb{R}$, with initial condition
$X_0=X$, is independent of the choice of Riemannian metric within the conformal
class of the Riemann surface $Y$, where $\pi^{-1}(Y)$ is identified with the
universal cover $\widetilde{Y}$. Hence, we may assume that $Y$ is equipped with
the hyperbolic metric in its conformal structure.

\begin{prop}\label{prop:GradDescFlow}
    Let $\Fcal$ be an $x$-generic measured foliation on $S_{g,n+1}$. Consider $X\in \Tcal_{g,n+1}$ and $Y:=\pi(X)\in \Tcal_{g,n}$. Suppose that $(X_t)_{t\ge0}$ is the trajectory of the gradient descent flow of the extremal length function $\EL(\Fcal;-)\colon \pi^{-1}(Y)\to\mathbb{R}$,
    with $X_0=X$, parametrized to have unit speed with respect to the hyperbolic metric on $Y$. Consider $(X_t)$ as a path $x(t)\from \Rbb_{\ge0} \to Y$ of the extra marked point in $Y$. Then the following hold.
    \begin{enumerate}
        \item $x'(t)$ is tangent to the leaf of $q_{\Fcal,X_t}\in Q(X_t)$ at $x(t)$, which is a simple pole of $q_{\Fcal,X_t}$. 
        \item We have
        \begin{equation}\label{eqn:PunctMoving_Res}
            \frac{d}{dt}\log(\EL(\Fcal;X_t))=-2\pi {\rm Res}_{x(t)}(q_{\Fcal,X_t}\cdot x'(t))<0.
        \end{equation}
        \item Let $p\colon \widetilde{Y}\to Y$ be the universal covering map, and identify $\pi^{-1}(Y)$ with $\widetilde{Y}$. Then there exist $\epsilon>0$ and a compact subsurface $Y'\subset Y$ such that the set
        \[
            \{t\in[0,\infty): p\circ X_t\in Y'\}
        \]
        contains infinitely many pairwise disjoint intervals of length $\epsilon$.
    \end{enumerate}
    
\end{prop}
\begin{proof}
    Let $\mu_t\in T_{X_t}\Tcal_{g,n}$ be the tangent vector of the path $X_t$ at $t$. For any $q\in Q(X_t)$, we have
    \[
        \pair{\mu_t}{q}=-\pi {\rm Re}({\rm Res}_{x(t)}(q\cdot x'(t))),
    \]
    where $x'(t)$ is considered as a holomorphic vector field near $x(t)$ extending the tangent vector $x'(t)$ at $x(t)$ \cite[Lemma 8.2]{Maxime_HoloCouch}. By Gardiner's formula \cite[Theorem 8]{Gardiner_MFQD}, the directional derivative of $\log(\EL(\Fcal;-))$ along $\mu$ is given by
    \[
        \nabla_{\mu_t} \log(\EL(\Fcal;-))=2\pair{\mu_t}{q_{\Fcal,X_t}}.
    \]
    Thus, we have
    \[
        \nabla_{\mu_t} \log(\EL(\Fcal;-))=-2\pi {\rm Re}({\rm Res}_{x(t)}(q_{\Fcal,X_t}\cdot x'(t))).
    \]
    Consider a local chart near $x(t)$ for which we have $q_{\Fcal,X_t}(z)=\frac{1}{z} dz^2$ with $0=x(t)$. Then the leaf of the horizontal foliation of $q$ starting from $0=x(t)$ is the non-negative real line, along which direction $x'(t)$ maximizes ${\rm Re}({\rm Res}_{x(t)}(q\cdot x'(t)))$ and ${\rm Res}_{x(t)}(q\cdot x'(t))$ is a non-negative real number. This completes the proofs of (1) and (2).

    Suppose, toward a contradiction, that $x(t)$ escapes into a cusp of $Y$. It follows that there exists a simple closed curve $\gamma$ of $S_{g,n+1}$ bounding a disk with two punctures one of which is $x$ such that $\EL(\gamma;X_t)\to 0$. Let $y$ be the other puncture. We have $\pair{\Fcal}{[\gamma]}>0$; otherwise, $\Fcal$ would necessarily contain a saddle connection joining $x$ and $y$, contradicting the assumption that it is $x$-generic. By Lemma \ref{lem:IneqBtwELInt} with $n$ replaced by $n+1$, we have $\EL(\Fcal;X_t)\to \infty$. This contradicts that $(X_t)$ is a gradient descent flow of $\EL(\Fcal;-)$.

    Now, we have that $(p\circ X_t)$ is recurrent to a compact subsurface
    $Y'\subset Y$. Recall that $(X_t)$ has unit speed with respect to the hyperbolic metric of $Y$. There exists a compact surface $Y''$ with $Y'\subset Y'' \subset Y$ such that whenever $(p\circ X_t)$ enters $Y'$, it remains in $Y''$ for at least
    time $\epsilon$. By replacing $Y'$ with $Y''$, we obtain the statement (3).
\end{proof}

Let us explain how the deck transformation action of $\Push(\pi_1(S_{g,n},x))$ on $\pi^{-1}(Y)$ can be interpreted as point-pushing. Fix $X_0\in \pi^{-1}(Y)$. Choose a homeomorphism $f\colon S_{g,n+1}\to X_0$ representing $X_0$ as a point of $\Tcal_{g,n+1}$. Define $x_0:=f(x)$, which we regard as the extra marked point of $X_0$. We identify $\pi_1(S_{g,n},x)$ with $\pi_1(X_0,x_0)$ via
\[
    f_*\colon \pi_1(S_{g,n},x)\to \pi_1(X_0,x_0).
\]
Every point $X'\in\pi^{-1}(Y)$ is represented by a homotopy class of paths $\alpha\colon [0,1]\to Y$ with $\alpha(0)=x_0$ and $\alpha(1)=x'$, where $x'$ is the extra marked point of $X'$. Equivalently, $X'$ is obtained from $X_0$ by pushing the extra marked point $x$ along $\alpha$. For every $[\gamma]\in\pi_1(Y,x_0)$, the image of its action $[\gamma]\cdot X'$ is represented by the homotopy class of the concatenated path $[\gamma\cdot\alpha]$, along which the extra marked point is pushed. The groups $\pi_1(Y,x')$ and $\pi_1(Y,x_0)$ are identified via conjugation by $\alpha$, namely,
\[
    \pi_1(Y,x_0)=\alpha\cdot\pi_1(Y,x')\cdot\overline{\alpha}.
\]
On the right-hand side of the equation
\[
    [\gamma\cdot\alpha]=[\alpha\cdot\overline{\alpha}\cdot\gamma\cdot\alpha],
\]
the initial path $\alpha$ represents $X'$, while the loop $\overline{\alpha}\cdot\gamma\cdot\alpha$ represents the element of $\pi_1(Y,x')$ corresponding to $\gamma\in\pi_1(Y,x_0)$.

\begin{lem}\label{lem:Limit MF of Grad Flow}
    In the setting of Proposition \ref{prop:GradDescFlow}, there exists an increasing sequence $(t_l)_{l\in \mathbb{N}}$ of positive real numbers and a sequence $([\psi_l])_{l\ge0}$ in $\Push(\pi_1(S_{g,n},x))$ such that
    \[
        [\psi_l]\cdot \Fcal \to \Hcal
        \quad\text{and}\quad
        [\psi_l]\cdot X_{t_l} \to X_\infty
        \quad\text{as}~
        l\to\infty,
    \]
    where $X_\infty \in \pi^{-1}(Y)$ and $\Hcal\in \MF(S_{g,n+1})$ is either the zero measured foliation or an $x$-non-generic measured foliation.
\end{lem}
\begin{proof}
    Recall that $p\from \widetilde{Y}\to Y$ is the universal cover and $\widetilde{Y}$ can identified with the fiber $\pi^{-1}(Y)$ of $\pi\from \Tcal_{g,n+1}\to \Tcal_{g,n}$.
    
    We first consider the case $\lim_{t\to\infty}\EL(\Fcal,X_t)> 0$.
    By Proposition \ref{prop:GradDescFlow}, there exists a sequence of disjoint intervals $I_k:=[t_k, t_k+\epsilon]$ for $k\ge1$ so that
    \[
        \int_{I_k}-2\pi {\rm Res}_{x(t)}(q_{\Fcal,X_t}\cdot x'(t))\,dt=\log \EL(\Fcal;X_{t_k+\epsilon})-\log \EL(\Fcal;X_{t_k}),
    \]
    which tends to $0$ as $k\to \infty$. Therefore, 
    there exists a sequence $(t_k)_{k=1}^\infty$ with $t_k \in I_k$ for each $k$ such that
    \begin{equation}\label{eqn:Res_to_zero}
        \lim_{k\to \infty}{\rm Res}_{x(t_k)}(q_{\Fcal,X_{t_k}}\cdot x'(t_k)) = 0.
    \end{equation}
    
    By passing to a subsequence of $(t_k)$, we may assume that $(p\circ X_{t_k})_{k\ge1}$ converges to a point in the compact subsurface $Y'$ of $Y$ introduced in Proposition \ref{prop:GradDescFlow}. Therefore, there exists a sequence $([\psi_k])$ in $\Push(\pi_1(S_{g,n},x))$ such that
    \[
        [\psi_k]\cdot X_{t_k}\to X_\infty
        \quad\text{in}~
        \pi^{-1}(Y).
    \]

    Note that
    \[
        \EL(\Fcal; X_{t_k})=\EL([\psi_k]\cdot \Fcal; [\psi_k] \cdot X_{t_k})
        \,\le\,
        \EL(\Fcal;X).
    \]
    By the continuity of $\EL\from \MF(S_{g,n+1})\times \Tcal_{g,n+1}\to \Rbb$ and the compactness of $\PMF(S_{g,n+1})$, by passing to a subsequence of $(\psi_k)$, we have $[\psi_k]\cdot \Fcal\to \Hcal$ for some $\Hcal\in \MF(S_{g,n+1})$. Moreover, we have
    \[
        [\psi_k]\cdot q_{\Fcal,X_{t_k}}=q_{[\psi_k]\cdot \Fcal,[\psi_k]\cdot X_{t_k}}\to q_{\Hcal,X_\infty}.
    \]
    
    Recall that $(X_t)$ has unit speed, which means $\|x'(t)\|=1$. By Equation \eqref{eqn:Res_to_zero}, the quadratic differential $q_{\Hcal,X_\infty}$ does not have a simple pole at the marked point $x$. Therefore, $\Hcal$ is $x$-non-generic.

    When $\lim_{t\to\infty}\EL(\Fcal,X_t)=0$, the preceding argument still works without requiring Equation~\eqref{eqn:Res_to_zero}, and consequently we obtain $\Hcal=0$.
\end{proof}

\begin{lem}\label{lem:IneqEL}
    Let $\Fcal$ and $\Fcal'$ be measured foliations of $\MF(S_{g,n})$. Suppose that $\pair{\Fcal'}{[\gamma]} \ge \pair{\Fcal}{[\gamma]}$ for every simple closed curve (hence every measured foliation) $\gamma$ on $S_{g,n+1}$. Then $\EL(\Fcal';X) \ge \EL(\Fcal;X)$ for any $X\in \Tcal_{g,n}$.
\end{lem}
\begin{proof}
    Let $q_{\Fcal,X}$ be the quadratic differential of $X$ whose horizontal measured foliation is $\Fcal$. Denote by $\Gcal$ the vertical measured foliation of $q_{\Fcal,X}$. By Lemma \ref{lem:IneqBtwELInt}, we have
    \[
        \EL(\Fcal';X)\EL(\Gcal;X) \ge \pair{\Fcal'}{\Gcal}^2 \ge \pair{\Fcal}{\Gcal}^2 = \EL(\Fcal;X) \EL(\Gcal;X).\qedhere
    \]
\end{proof}

\begin{lem}\label{lem:InequELTwoNonGeneric}
Let $\Fcal$ be an $x$-generic measured foliation on $S_{g,n+1}$ and $X\in \Tcal_{g,n+1}$. Let $\Gcal$ and $\Hcal$ be $x$-non-generic measured foliations obtained by Proposition \ref{prop:LimitMFfromIntReduction} and Lemma \ref{lem:Limit MF of Grad Flow}, respectively, which may be the zero element in $\MF(S_{g,n})$. For every $[\alpha] \in C(S_{g,n})$, we have
    \[
        \pair{\Phi(\Gcal)}{[\alpha]} \le \pair{\Phi(\Hcal)}{[\alpha]}.
    \]
    Then, by Lemma \ref{lem:IneqEL}, we have
    \[
        \EL(\Phi(\Gcal); \pi(X)) \le \EL(\Phi(\Hcal); \pi(X)).
    \]    
\end{lem}
\begin{proof}
    Since $\Hcal$ is $x$-non-generic, by Proposition \ref{prop:PForget NonGeneric} there is $[\gamma]\in C(S_{g,n+1})$ such that $i_*[\gamma]=[\alpha]$ and $\pair{\Hcal}{[\gamma]}=\pair{\Phi(\Hcal)}{i_*[\gamma]}$. Recall that $\Hcal$ is obtained as the limit of a sequence of point-pushing orbits, i.e., there exists a sequence $\left( [\psi_l]\right)_{l\ge1}$ of point-pushing mapping classes such that $[\psi_l] \cdot \Fcal \to \Hcal$ as $l\to \infty$. For each $l\ge 1$, we have
    \[
        \pair{\Phi(\Gcal)}{i_*[\gamma]}=\inf_{\gamma \sim_{\pre} \gamma'} \pair{\Fcal}{[\gamma']}\le \pair{\Fcal}{[\psi_l]^{-1} \cdot [\gamma]}=\pair{[\psi_l]\cdot\Fcal}{[\gamma]}
    \]
    By taking the limit $\lim_{l\to \infty}$, we obtain
    \[
        \pair{\Phi(\Gcal)}{i_*[\gamma]}\le \pair{\Hcal}{[\gamma]}=\pair{\Phi(\Hcal)}{i_*[\gamma]}.\qedhere
    \]
\end{proof}

Now we are ready to prove Theorem \ref{thm:InequalityEL}.

\begin{proof}[Proof of Theorem \ref{thm:InequalityEL}] By Lemma~\ref{lem:InequalityELNonGeneric}, it suffices to consider the case where $\Fcal$ is $x$-generic. We further assume that $\Fcal$ is minimal, as the non-minimal case reduces to the minimal one by the argument in Section~\ref{subsubsec:nonminimal_reduction}.

Let $\Gcal$ be the $x$-non-generic measured foliation obtained by Proposition \ref{prop:LimitMFfromIntReduction}.  Recall that we define $\Phi(\Fcal):=\Phi(\Gcal)$; see Section \ref{subsec:Proof of Thm1 I}.

For any $X'\in \pi^{-1}(\pi(X))$, we obtain
\[
    \EL(\Fcal;X') \ge \EL(\Phi(\Fcal); \pi(X')) = \EL(\Phi(\Fcal); \pi(X))
\]
from the following: We let $\Hcal$ and $X'_\infty$ be the $x$-non-generic measured foliation and the element in $\pi^{-1}(\pi(X))$ obtained by Lemma \ref{lem:Limit MF of Grad Flow} with $X$ being replaced with $X'$. Then, we have
    \begin{align*}
        \EL(\Fcal;X') & \ge \EL(\Hcal;X'_\infty) & &\\
            & \ge \EL(\Phi(\Hcal);\pi(X'_\infty))=\EL(\Phi(\Hcal);\pi(X)) & \text{(by  Lemma \ref{lem:InequalityELNonGeneric}})&\\
            & \ge \EL(\Phi(\Gcal); \pi(X)) & \text{(by Lemma \ref{lem:InequELTwoNonGeneric})}&\\
            & = \EL(\Phi(\Fcal); \pi(X)). & &
    \end{align*}
    
On the other hand, since $\Gcal$ is $x$-non-generic of non-annular type (Proposition \ref{prop:LimitMFfromIntReduction}), Lemma \ref{lem:InequalityELNonGeneric} implies that there exists $X'\in \pi^{-1}(\pi(X))$ such that $\EL(\Gcal;X')=\EL(\Phi(\Fcal);\pi(X))$. Let $([\psi_l])_{l\ge1}$ be the sequence of point-pushing mapping classes in Proposition \ref{prop:LimitMFfromIntReduction} such that $[\psi_l]\cdot \Fcal \to \Gcal$ as $l\to \infty$. Then, for each $l\ge0$, we have
    \[
        \inf_{X''\in \pi^{-1}(\pi(X))} \EL(\Fcal;X'')\le \EL(\Fcal;[\psi_l]^{-1}\cdot X')=\EL([\psi_l]\cdot \Fcal;X')
    \]
    and
    \[
        \EL([\psi_l]\cdot \Fcal;X')\to \EL(\Gcal;X')
    \]
    as $l\to \infty$.
\end{proof}

\begin{thm}\label{thm:LowBddTeichDist}
Suppose $X\in \Tcal_{g,n+1}$ and $\Fcal\in \MF(S_{g,n+1})$. Let $(X_t)_{t\ge0}$ be the Teichm\"uller geodesic generated by $\Fcal$ at $X_0=X$ with $d_{\Teich}(X_0,X_t)=t$. If $\Phi(\Fcal)\neq 0$, then for any $t>0$ we have
\[
    t-\frac{1}{2} \log \frac{\EL(\Fcal;X_0)}{\EL(\Phi(\Fcal);\pi(X_0))} \le d_{\Teich}(\pi(X_0),\pi(X_t)) \le t.
\]
This implies that
\[
    \int_0^\infty 1-||D\pi(\Dot{X_t}) ||_{\Teich}~dt \le \frac{1}{2} \log \frac{\EL(\Fcal;X_0)}{\EL(\Phi(\Fcal);\pi(X_0))}< \infty
\]
\end{thm}
\begin{proof}
The inequality $d_{\Teich}(\pi(X_0),\pi(X_t)) \le t$ follows from the facts that $\pi$ is holomorphic and that holomorphic maps do not increase Teichm\"uller distances.

Let $q_0\in Q(x_0)$ be the cotangent vector of the Teichm\"uller geodesic $(X_t)$ such that $\Hor(q_t)=e^{t}\Hor(q_0)$ and $\Hor(q_0)=\Fcal$, where $\Hor(q)$ is the horizontal measured foliation of the quadratic differential $q\in Q(Y)$ of a Riemann surface $Y$. We have
\begin{align*}
    \EL(\Fcal;X_0)=\EL(\Hor(q_0);X_0) & = \EL(\Hor(q_t);X_t)\\
    &\ge \EL(\Phi(\Hor(q_t));\pi(X_t)) \\
    &= \EL(\Phi(e^{t} \Hor(q_0));\pi(X_t))\\
    &= e^{2t} \EL(\Phi(\Fcal);\pi(X_t)).
\end{align*}
We use the fact that $\EL$ depends quadratically on the weight of measured foliations. Then, by \cite[Theorem 4]{Kerckhoff_AsympGeoTeich}, we have
\begin{align*}
    d_{\Teich}(\pi(X_0),\pi(X_t)) = & \frac{1}{2} \sup\limits_{\Gcal \in \MF(S_{g,n})} \left| \log\frac{\EL(\Gcal ; \pi(X_0)))}{\EL(\Gcal ; \pi(X_t))}\right|\\
    \ge & \frac{1}{2} \log\frac{\EL(\Phi(\Fcal) ; \pi(X_0))}{\EL(\Phi(\Fcal) ; \pi(X_t))}\\
    = & t- \frac12\log \frac{\EL(\Fcal ;X_0)}{\EL(\Phi(\Fcal) ; \pi(X_0))}.
\end{align*}
The second statement then follows from $d_{\Teich}(X_0,X_t)= t$ in $\Tcal_{g,n+1}$ for $t\ge0$.
\end{proof}


\section{Applications in complex dynamics: partial results for the global finite attractor conjecture}\label{sec:CplxDyn}
Consider a post-critically finite rational map $f\from \hat{\Cbb} \to \hat{\Cbb}$ with post-critical set $P_f$.
Denote by $\sigma_f\from \Tcal(\hat{\Cbb},P_f) \to \Tcal(\hat{\Cbb},P_f)$ Thurston's pullback map on the Teichm\"uller space. By abusing notation, we also denote by $\sigma_f\from \MF(\hat{\Cbb},P_f) \to \MF(\hat{\Cbb},P_f)$ the composition of the pullback map $f^*\from \MF(\hat{\Cbb},P_f)\to \MF(\hat{\Cbb},f^{-1}(P_f))$ and the puncture forgetting map $\Phi\from \MF(\hat{\Cbb},f^{-1}(P_f))\to \MF(\hat{\Cbb},P_f)\cup \{0\}$. Here we forget all the punctures in $f^{-1}(P_f) \setminus P_f$, which usually contains more than one puncture.

A {\it simple multicurve} $\Gamma$ of $(\hat{\Cbb},P_f)$ is a multicurve that can be represented by disjoint simple closed curves. A {\it weighted} simple multicurve is a simple multicurve with a positive real number assigned to each component.

The map $\sigma_F\from \MF(\hat{\Cbb},P_f)\to\MF(\hat{\Cbb},P_f)$ restricted to weighted simple multicurves $[\Gamma]$ are also defined by 
\[
    \sigma_f[\Gamma]=\sum_{[\alpha]\in C(\hat{\Cbb},P_f)} k_{[\alpha]}\cdot [\alpha],
\]
where $k_{[\alpha]}$ is the number of connected components of $f^{-1}(\Gamma_n)$ that are homotopic to $\alpha$.

Recall that $\Proj\from \MF(\hat{\Cbb}, P_f) \to \PMF(\hat{\Cbb}, P_f)$ denotes the projectivization map. Recall that $\Fcal$ is geodesically recurrent if a Teichm\"uller geodesic $(X_t)$ generated by $\Fcal$ is recurrent in the moduli space $\Mcal(\hat{\Cbb},P_f)$.

\begin{thm}
    Let $f$ be a post-critically finite rational map that is not a flexible Latt\`es map. Suppose $||D\sigma_f||<1$. If $\Fcal \in \MF(\hat{\Cbb}, P_f)$ is a geodesically recurrent, then $\sigma_f(\Fcal)=0$. In particular, if $|P_f| = 4$, $\sigma_f(\Fcal) = 0$ for all irrational measured foliations $\Fcal$.
\end{thm}
\begin{proof}
    By Thurston's characterization \cite{DH_ThurstonChar}, the pullback map $\sigma_f\from \Tcal(\hat{\Cbb},P_f)\to \Tcal(\hat{\Cbb},P_f)$ has a unique fixed point, which we denote by $X$. Let $(X_t)_{t\ge0}$ denote the Teichm\"uller geodesic starting from $X=X_0$ generated by $\Fcal$.
    
    Suppose that $\Fcal$ is geodesically recurrent. Then there exists a compact subset $K\subset \Mcal(\hat{\Cbb},P_f)$ so that the total amount of time that $(X_t)$ spends in $K$ is infinite. This is immediate if $(X_t)$ eventually remains in $K$. Otherwise, suppose that $(X_t)$ leaves and reenters $K$ infinitely many times. There exists a compact subset $K'\subset \Mcal(\hat{\Cbb},P_f)$ whose interior contains $K$ such that the distance between the boundaries of $K$ and $K'$ is $t_0>0$. Moreover, each time $(X_t)$ enters $K$, it must remain in $K'$ for at least time $t_0$ before it can leave $K$. Hence the total amount of time that $(X_t)$ spends in $K'$ is infinite. Then we replace $K$ with $K'$.

    Let $q\from \Tcal(\hat{\mathbb{C}},P_f)\to \mathcal{M}(\hat{\mathbb{C}},P_f)$ be the quotient map onto the moduli space. The norm of the derivative $||D_Y\sigma_f||$ depends only on $q(Y)$ \cite[Proposition 3.2]{DH_ThurstonChar}. Hence the path $(q(X_t))_{t\ge0}$ in $\mathcal{M}(\hat{\mathbb{C}},P_f)$ returns infinitely often to a compact subset $K$ where $||D_Y\sigma_f||<1-\epsilon$ for some small $\epsilon=\epsilon(K)>0$. Hence, $1-||D\pi_*(\dot{X_t})||_{\Teich}\nrightarrow 0$ as $t\to \infty$. Then, by applying Theorem \ref{thm:LowBddTeichDist} for $f^*\Fcal\in \MF(\hat{\Cbb},f^{-1}(P_f))$, we obtain $\Phi(\Fcal)=0$. 
\end{proof}

Suppose that $f$ is a post-critically finite rational map. A {\it simple multicurve} $\Gamma$ of $(\hat{\Cbb},P_f)$ is a multicurve that can be represented by disjoint simple closed curves. For $n\ge 1$, we say that a simple multicurve $\Gamma$ is {\it $f^n$-invariant} if for any $\gamma \in \Gamma$, every essential connected component of $f^{-n}(\gamma)$ is homotopic to some component of $\Gamma$. Moreover, we say that $\Gamma$ is {\it completely $f^n$-invariant} if $\Gamma$ is $f^n$-invariant and every component of $\Gamma$ is an essential component of $f^{-n}(\Gamma)$. If $f$ does not have a global finite curve attractor, then there exist two possibilities.
\medskip

\noindent (1) {\bf Backward wandering (multi-)curves}

There exists a sequence of distinct homotopy classes of simple closed curves $([\gamma_n])_{n\ge0}$ such that $\gamma_{n+1}$ is homotopic to a component of $f^{-1}(\gamma_n)$ in $(\hat{\Cbb},P_f)$. 

In this case, we define $[\Gamma_1]=[\gamma_1]$ and $[\Gamma_{n+1}]:=(\sigma_f)^{n}([\Gamma_1])$ for all $n\ge1$. Then $\gamma_n\in \Gamma_n$. We define a sequence of probability measures $(\mu_u)_{n\ge0}$ by
\begin{equation}\label{eqn:ProbMsrWandering}
    \mu_n:= \frac{1}{n} \sum_{k=0}^{n-1} \delta_{\Proj([\Gamma_k])}.
\end{equation}

\noindent (2) {\bf Infinitely many virtually invariant multicurves}

There exist infinitely many pairwise distinct simple multicurves $\Gamma_1,\Gamma_2,\dots$ such that for each $n\ge1$, $[\Gamma_n]$ is completely $f^{ p_n}$-invariant for some $p_n>0$. By \cite[Theorem 1.5]{Pilgrim_Alg_Thu_Chac}, $p_n\to \infty$. By further decomposing $\Gamma_n$ into $f^{p_n}$-invariant simple sub-multicurves, we may additionally assume that each $\Gamma_n$ is {\it $f^{p_n}$-irreducible}, i.e., for every ordered pair $\gamma_i,\gamma_j\in \Gamma_n$, there exists $k>0$ so that $f^{-kp_n}(\gamma_i)$ contains a component homotopic to $\gamma_j$.

For each $\Gamma_n$, we have a linear transformation $T_n;=(\sigma_f)^{p_n}|_{\Rbb^{\Gamma_n}}\from \Rbb^{\Gamma_n}\to \Rbb^{\Gamma_n}$ on a finite dimensional vector space $\Rbb^{\Gamma_n}$. 
As a matrix, after conjugation by a permutation matrix, $T_n$ becomes a nonnegative integer irreducible matrix. By the Perron-Frobenius theorem, it has a positive eigenvector, which is unique up to constant multiple, with eigenvalue $\lambda_n\ge1$. Denote by $\Gamma_n^w$ the weighted multicurve defined by an eigenvector of $T_n$. Then, we have $\lambda_n [\Gamma_n^w]=(\sigma_f)^{p_n}([\Gamma_n^w])$ in $\MF(\hat{\Cbb},P_f)$ and $\Proj([\Gamma_n^w])=(\sigma_f)^{p_n}(\Proj([\Gamma_n^w]))$ in $\PMF(\hat{\Cbb},P_f)$. Define a measure $\mu_n$ on $\PMF(\hat{\Cbb},P_f)$ by
\begin{equation}\label{eqn:ProbMsrIntPer}
    \mu_n:=\frac{1}{p_n}\sum_{i=0}^{p_n-1} \delta_{(\sigma_f)^{\circ i}(\Proj([\Gamma_n^w]))}.
\end{equation}
By construction, $\mu_n$ is an $(\sigma_f)_*$-invariant probability measure on $\PMF(\hat{\Cbb},P_f)$.

\begin{thm}\label{thm:LimitMsr nonRecurrent}
    Let $f$ be a post-critically finite rational map with $|P_f|>3$ that is not a flexible Latt\`es map. Let $\Gamma$ be a multicurve of $(\hat{\Cbb},P_f)$. Suppose that $f$ does not have a global finite attractor of $C(\hat{\Cbb},P_f)$. Let $(\mu_n)$ be a sequence of probability measures on $\PMF(\hat{\Cbb},P_f)$ constructed in Equation \eqref{eqn:ProbMsrWandering} or \eqref{eqn:ProbMsrIntPer}. Then, for every weak$^*$-limit $\mu$ of $(\mu_n)_{n\ge0}$, we have $(\sigma_f)_*\mu=\mu$ and 
    \[
        \mu(\{\Proj(\Fcal)\in \PMF(\hat{\Cbb},P_f) ~|~ \Fcal~{\rm is~geodesically~recurrent}\})=0.
    \]
\end{thm}

We remark that if $|P_f|=4$, then $\mu$ is a (possibly countably infinite) convex sum of $\delta_\Gamma$ for some periodic cycles $\Gamma$ of $\sigma_f$.

\begin{proof}
We first consider $\mu_n$ in Equation \eqref{eqn:ProbMsrWandering}. Let $\mu_{n_k}\to \mu$ as $k\to \infty$ in the weak$^*$-topology. The $(\sigma_f)_*$-invariance is straightforward. Define $r\from \PMF(\hat{\Cbb},P_f)\to [-\infty, \log( \deg(f))]$ by
\[
    r(\Proj(\Fcal))=\log\frac{L(\sigma_f(\Fcal))}{L(\Fcal)}
\]
where $L(\cdot)$ denotes the length with respect to the hyperbolic metric on $(\hat{\Cbb},P_f)$. Then $r$ is upper semi-continuous and $r(\Proj(\Fcal))=-\infty$ for any recurrent measured foliation $\Fcal$.

For any $n>0$, we have
\begin{align*}
    \int_{\PMF} r \,d\mu_n & = \frac{1}{n}\sum_{i=0}^{n-1} \log \frac{L(\sigma_f([\Gamma_i]))}{L([\Gamma_i])}\\
    & \ge \frac{1}{n}\sum_{i=0}^{n-1} \log \frac{L([\Gamma_{i+1}])}{L([\Gamma_i])}\\
    &\ge \frac{1}{n} \log\frac{L([\Gamma_{n}])}{L([\Gamma_0])}.
\end{align*}
Note that $\gamma_n\in \Gamma_n$ and $L([\gamma_n])\to \infty$ as $n\to \infty$. It follows from Lemma \ref{lem:IneqUpperSemiCont} that we have
\[
    0 \le \limsup_{k\to \infty} \int_{\PMF} r \,d \mu_{n_k} \le \int_{\PMF} r \, d\mu.
\]
Since $r(\Proj(\Fcal))=-\infty$ for any recurrent measured foliation $\Fcal$, we have
\[  
    \mu(\{\Proj(\Fcal)\in \PMF(\hat{\Cbb},P_f) ~|~ \Fcal~{\rm is~recurrent}.\})=0.
\]

Next, consider $\mu_n$ in Equation \eqref{eqn:ProbMsrIntPer}. For any $n>0$, we have
\begin{align*}
    \int_{\PMF(\hat{\Cbb},P_f)} r \,d\mu_n & = \frac{1}{p_n} \sum_{i=0}^{p_n-1}\log \frac{L((\sigma_f)^{i+1}(\Gamma_n^w))}{L((\sigma_f)^{i}(\Gamma_n^w))}\\
    &= \frac{1}{p_n} \log \frac{L((\sigma_f)^{p_n}(\Gamma_n^w))}{L(\Gamma_n^w)}\\
    &=\frac1{p_n}\log\lambda_n,
\end{align*}
where $\lambda_k\ge1$ is the leading eigenvalue of $T_k\from \Rbb^{\Gamma_k}\to \Rbb^{\Gamma_k}$. Hence $\int_{\PMF(\hat{\Cbb},P_f)} r \,d\mu_n\ge0$. The rest follows similarly to the previous case.
\end{proof}

\begin{lem}\label{lem:IneqUpperSemiCont}
    Let $X$ be a compact metric space and $f\from X \to [-\infty, D]$ be an upper semi-continuous function for some $D>0$. If $(\mu_k)_{k\ge 1}$ is a sequence of measures on $X$ with $\mu_k \to \mu$ with respect to the weak$^*$-topology, then
    \[
        \limsup_{k\to \infty} \int_X f \,d\mu_k \le \int_X f \, d\mu.
    \]
\end{lem}
\begin{proof}
    By Baire's theorem, the upper semi-continuous function $f$ is the limit of a monotonically decreasing sequence of continuous functions $(f_n)_{n\ge0}$. Since the family $\{f_n\}_{n>0}$ is uniformly bounded from above by $D$, by Beppo Levi's theorem, which is a slight generalization of Lebesgue's monotone convergence theorem, for any $k>0$ we have
    \[
        \int_X f \,d\mu_k = \lim_{n\to \infty} \int_{X} f_n \, d\mu_k
    \]
    By taking $\limsup_{k\to \infty}$, we have
    \[
        \limsup_{k\to \infty} \int_{X} f \, d\mu_k = \limsup_{k\to \infty}\lim_{n\to \infty} \int_{X} f_n \, d\mu_k.
    \]
    Since $f_n$ is monotonically decreasing in $n$, by Lemma \ref{lem:IneqDoubleSeq} we have
    \begin{align*}
        \limsup_{k\to \infty}\lim_{n\to \infty} \int_{X} f_n \, d\mu_k & \le \lim_{n\to \infty} \lim_{k\to \infty} \int_{X} f_n \, d\mu_k\\
        & = \lim_{n\to \infty} \int_{X} f_n \, d\mu\\
        & = \int_{X} f \, d\mu,
    \end{align*}
    where the last equality follows from Beppo Levi's theorem. Hence we have
    \[
        \limsup_{k\to \infty} \int_{X} f \, d\mu_k \le \int_{X} f \, d\mu. \qedhere
    \]
\end{proof}

\begin{lem}\label{lem:IneqDoubleSeq}
    Let $(a_{n,k})_{n,k\ge0}$ be a double sequence. Suppose that for every $k$ the sequence $(a_{n,k})_{n \ge 0}$ is monotonically decreasing. Also suppose that for every $n$ the sequence $(a_{n,k})_{k \ge0}$ converges. Then we have
    \[
        \limsup_{k \to \infty} \lim_{n \to \infty} a_{n,k} \le \lim_{n \to \infty} \lim_{k \to \infty} a_{n,k}.
    \]
\end{lem}
\begin{proof}
    Since $(a_{n,k})$ is monotonically decreasing about $n$, for any $n,k>0$ we have
    \[
        \lim_{n\to \infty} a_{n,k} \le a_{n,k}.
    \]
    By taking $\limsup_{k\to \infty}$, for any $n \ge 0$ we have
    \[
        \limsup_{k \to \infty} \lim_{n \to \infty} a_{n,k} \le \lim_{k \to \infty} a_{n,k}.
    \]
    The monotonicity of $(a_{n,k})_{n\ge0}$ for any $k\ge0$ also implies that $(\lim_{k\to \infty} a_{n,k})_{n\ge 0}$ is also monotonically decreasing. Hence we can take $\lim_{n\to \infty}$ to the left hand side of the above inequality and obtain the conclusion.
\end{proof}

\section{Shards of memory}

In this section, we show the following theorem.
\begin{thm}\label{thm:Shards of memory}
    The set of $x$-generic minimal measured foliations $\Fcal$ with $\varphi(\Fcal)\neq 0$ is a dense subset of $\MF(S_{g,n+1})$.    
\end{thm}

Recall the number $n(\Fcal,\gamma)$ of canonical simple pulls that can be performed to $\gamma$ wrt $\Fcal$, which is discussed in Section \ref{subsec:Non-generic}. We have $n(\Fcal,\gamma)=\infty$ if $\Fcal$ is $x$-generic, and $n(\Fcal,\gamma)<\infty$ if $\Fcal$ is $x$-non-generic (Lemma \ref{lem:non-gen_FiniteCanSimPulls}).

For notational simplicity, we equip the topological space $\MF(S_{g,n+1})$ with a metric $d$.
 
\begin{lem}[Stability of finite sequence of canonical simple pulls]\label{lem:stability of finite simple pulls}
    Suppose $N$ is a finite non-negative integer with $N\le n(\Fcal,\gamma)\le \infty$. There exists $\epsilon = \epsilon(\Fcal,\gamma,N)>0$ such that for each $\Gcal\in \Ncal_\epsilon(\Fcal)$ and each $k \le N$, the $k^{th}$-canonical simple pull of $[\gamma]$ with respect to $\Fcal$ coincides with the $k^{th}$-canonical simple pull with respect to $\Gcal$. As a result, we obtain that the function
    \[
        n(-.\gamma)\from \MF(S_{g,n+1})\to \mathbb{Z}_{\ge0}\cup \{\infty\}
    \]
    is lower semicontinuous.
\end{lem}
\begin{proof}
    Let $\gamma_k$ (resp.\@ $\gamma'_k$) be the $k^{th}$-canonical simple pull of $\gamma$ wrt $\Fcal$ (resp.\@ $\Gcal$). For each $k\le N-1$, there exists $\epsilon_k:=\pair{\Fcal}{[\gamma_{k}]}-\pair{\Fcal}{[\gamma_{k+1}]}>0$ such that $\pair{\Fcal}{[\gamma_{k}]}-\pair{\Fcal}{[\alpha]}<-\epsilon_k$ for any other simple pull $[\alpha]$ of $[\gamma_k]$. Since $(\epsilon_k)_{k=0,1,\dots,N-1}$ is a finite set, by the continuity of the intersection number, these inequality holds for a small $\epsilon$-deformation of $\Fcal$ to $\Gcal\in \Ncal_\epsilon(\Fcal)$, where $\epsilon$ depends only on $\gamma,\Fcal,$ and $N$. Then the desired conclusion follows from the fact that the canonical simple pull is the unique simple pull that decreases the intersection number (Lemma~\ref{lem:IntNumbCanonPull}).
\end{proof}


Recall that $\Fcal$ is said to be immediately $x$-non-generic wrt $\gamma$ if $\Fcal$ and $\gamma$ can be transversely represented so that $\deg_\Fcal(x)\ge2$, which is equivalent to $n(\Fcal,\gamma)=0$.

Lemma \ref{lem:SmallChange2} is the main technical lemma, and Lemma \ref{lem:SmallChange1} serves as its preparation.

\begin{lem} \label{lem:SmallChange1}
    Suppose that $\Fcal$ is immediately $x$-non-generic with respect to a simple closed curve $\gamma$ in $S_{g,n+1}$. Fix an essential simple closed curve $\beta$ in $S_{g,n+1}$. Then, for every $\epsilon>0$, there exists a measured foliation $\Fcal'$ with $d(\Fcal,\Fcal')<\epsilon$ such that
    \begin{enumerate}
        \item $\pair{\Fcal'}{[\beta]}>0$,
        \item $\deg_{\Fcal'}(x)=2$,
        \item one prong $L_x$ of $\Fcal'$ at $x$ is a saddle connection $[x,y]$,
        \item $\deg_\Fcal(y)\in \{1,3\}$, and
        \item $|[x,y]\cap \gamma|>0$.
    \end{enumerate}
\end{lem}
\begin{proof}
    Note that $\mathcal{F}$ can be deformed by an arbitrarily small amount so that $\langle \mathcal{F}, [\beta] \rangle > 0$, $\deg_{\mathcal{F}}(x) = 2$ with its two prongs dense in $S_{g,n+1}$, and $\deg_{\Fcal}(z)\in\{1,3\}$ for every other singular point $z$ of $\Fcal$. Hence, we assume that $\mathcal{F}$ satisfies these properties.

    Let $L_x$ be one prong of $x$. Fix a singular point $y$ with $y \neq x$.

    Since $L_x$ is dense, for any arbitrarily small $\epsilon' > 0$, we can find a transverse arc $\delta$ from $y$ to a point $p \in L_x$ such that $\delta$ is disjoint from $\gamma$, $N := |[x, p] \cap \gamma| > 0$, and $\langle \mathcal{F}, \delta \rangle < \epsilon'$. Let $q$ be the first intersection of $L_x$ with $\gamma$. We choose $\epsilon'>0$ sufficiently small so that, for every measured foliation $\Gcal$ obtained by twisting $\Fcal$ along $\gamma$ by at most $\epsilon'$, we have $d(\Gcal,\Fcal)<\epsilon$ and $\langle \Gcal,[\beta]\rangle>0$.

    Suppose that there exists a horizontally foliated rectangle $R$ embedded in $\mathcal{F}$ such that one horizontal side is $[q, p]$ and one vertical side is $\delta$. Then, we twist $\mathcal{F}$ along $\gamma$ by $\langle \mathcal{F}, \delta \rangle / N=\epsilon'/N$ and obtain the desired measured foliation $\Fcal'$.

    Suppose that such a rectangle does not exist. There exists a maximal horizontally foliated rectangle $R'$ whose interior is embedded in $\mathcal{F}$ and boundary is piece-wise embedded such that one horizontal side is $[q, p]$ and the vertical sides are embedded in $\gamma$ and $\delta$. The side embedded in $\delta$ is strictly smaller than $\delta$, and the other horizontal side must contain a singular point $y'$ such that there exists an arc $\delta'$ from some $p'\in [q,p]$ to $y'$ with $\mathcal{F} \pitchfork \delta'$. We replace $y,\delta$, and $p$ with $y',\delta'$, and $p'$, respectively, and then apply the twisting argument used in the previous case.
\end{proof}

\begin{lem} \label{lem:SmallChange2}
    Suppose that $\Fcal$ is immediately $x$-non-generic with respect to a simple closed curve $\gamma$ in $S_{g,n+1}$. Fix an essential simple closed curve $\beta$ in $S_{g,n+1}$. Then, for every $\epsilon > 0$, there exists $\Gcal \in \MF_{g, n+1}$ such that
    \begin{enumerate}
        \item $\Gcal$ is eventually $x$-non-generic wrt $\gamma$ and of non-annular type,
        \item $d(\Fcal, \Gcal) < \epsilon$,
        \item $\varphi(\Fcal)([\gamma]) -\varphi(\Gcal)([\gamma])< \epsilon$, and
        \item $\pair{\Gcal}{[\beta]}>0$
    \end{enumerate}
\end{lem}
\begin{proof}
    Note that (2) follows if $\Gcal$ can be obtained from $\Fcal$ by an arbitrarily small deformation.

    By Lemma \ref{lem:SmallChange1}, we can assume that $\deg_{\Fcal}(x)=2$ and one prong $L_x$ at $x$ is a saddle connection $[x,y]$ such that $N:=|[x,y]\cap \gamma|>0$, $\deg_\Fcal(y)\in\{1,3\}$, and $\pair{\Fcal}{[\beta]}>0$. Then, (4) follows from (2) with $\epsilon$ sufficiently small.
    
    Denote by $z$ the $1^{st}$-intersection of $L_x$ with $\gamma$. Let $s$ and $s'$ be the two sides of $\gamma$. For an arc $\alpha$ contained in a leaf of $\Fcal$ and a point $p\in \partial \alpha\cap \gamma$, we say that the intersection occurs from the side $s$ (resp.\@ $s'$) if a small neighborhood of $p$ in $\alpha$ belongs to $s$ (resp.\@ $s'$). We label the sides such that the intersection at $z$ between $[x,z]$ and $\gamma$ occurs from the side $s$.\vspace{5pt}

    \noindent{\it (Case 1) $\deg_\Fcal(y)=3$.}\indent  
    We will fatten the segment $[x,z]$ into a pulley of width $2\epsilon_1 > 0$ so that $x$ has degree $1$. To achieve this, the opposite side $s'$ must also be properly fattened by $2\epsilon_1 > 0$. Then, we calibrate $\Fcal$ by twisting along $\gamma$ to obtain the desired $\Gcal$. We will discuss separately the cases where $\gamma$ is separating and non-separating.\vspace{5pt}
    
    {\it (Case i) $\gamma$ is separating.\quad} Let $S_1$ and $S_2$ be the closures of the connected components of $S_{g,n+1}\setminus \gamma$ such that $x\in S_1$. Let $\Fcal_1$ and $\Fcal_2$ be the restricted measured foliations of $\Fcal$ to $S_1$ and $S_2$ respectively. Recall $N:=|L_x\cap \gamma|$, where $L_x=[x,y]$.

    By Poincar\'{e} recurrence \cite[Theorem 5.2]{FLP_ThurstonSurface}, there exists a leaf $\alpha$ of $\Fcal_2$ with $(\alpha,\partial \alpha)\hookrightarrow (S_2,\partial S_2)$.
    
    \begin{figure}
	   \centering
        \def\svgwidth{0.7\textwidth}
        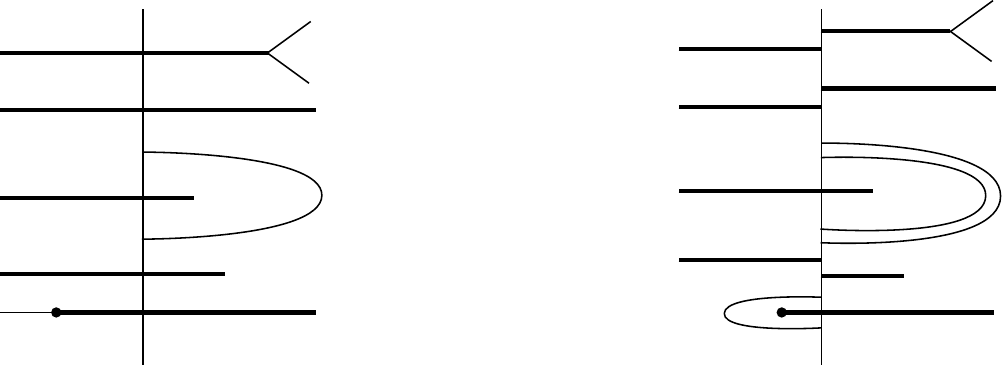
	\caption{The intersection pattern may be different. The leaf $L_x$ of the initial measured foliation is drawn in bold.}
    \label{fig:ImmediateFoli_deform}
    \end{figure}
    
    Fix an arbitrarily small $\epsilon_1>0$. We thicken the segment $[x,z]$ to a pulley of width $2\epsilon_1$ and the arc $\alpha$ to a horizontally foliated rectangle of height $\epsilon_1$, such that the equation $\pair{\Fcal_1}{\partial S_1}=\pair{\Fcal}{\partial S_2}$ is preserved. See Figure \ref{fig:ImmediateFoli_deform}. Denote by $\Fcal'$ the new fattened foliation. The new leaf $L'_x$ is then at most $N\epsilon_1$-away from $y$ after $N$ intersections with $\gamma$. For a sufficiently small $\epsilon_1$, we can find a horizontally foliated rectangle with height at most $N\epsilon_1$ for which $x$ and $y$ are opposite vertices. Then, by twisting $\Fcal'$ along $\gamma$ by at most $\epsilon_1$, we obtain a new foliation $\Gcal$ such that $[x,y]$ is a saddle connection, which is also the leaf $L_x$ of $x$, and $|L_x\cap \gamma|=N$. Moreover, $y$ is not on the boundary of the pulley of $\Gcal$ wrt $\gamma$. By choosing $\epsilon_1$ small enough, we can also obtain a good normalized representative $\Gcal$ that has a clean long pulley. See Figure~\ref{fig:CleanLongPulleys}-(A). Then $n(\Gcal,\gamma)=N$ by Lemma \ref{lem:non-gen_FiniteCanSimPulls}.
    
    Since $\pair{\Gcal}{[\gamma]}=\varphi(\Fcal)([\gamma])+2\epsilon_1$ and each canonical simple pull decreases the intersection number by $2\epsilon_1$, we have $\varphi(\Gcal)([\gamma])=\varphi(\Fcal)([\gamma])-2(N-1)\epsilon_1$. Since we can choose $\epsilon_1$ arbitrarily small with $N$ fixed, we obtain the desired measured foliation $\Gcal$.
    \vspace{5pt}
    
    {\it (Case ii) $\gamma$ is non-separating. \quad} Note that either (a) there exists an arc $\alpha$ contained in a leaf of $\Fcal$ such that $\alpha$ intersects $\gamma$ only at its boundary points, both from the side $s'$, or (b) in the cut surface $S^*:=\overline{S_{g,n+1}\setminus \gamma}$ with the induced foliation $\Fcal^*$, every non-singular leaf is an arc joining two different boundary components, which forces $g=1$.

    In Case (a), the next step is similar to (Case i).
    
    In Case (b), we can choose a representative $\Fcal$ which is the foliation of an irrational slope in the square torus such that every marked point has degree $2$. Since we assume $\chi(S_{g,n})<0$, we have $n+1>1$. Under this representative, the leaf $L_x$ is a saddle connection $[x,y]$ where $y$ is also a marked point with $y\neq x$. Let $N:=|[x,y]\cap \gamma|$. Let $p$ and $q$ be the first and the last intersection point of $L_x$ with $\gamma$. Then $[x,p]$ and $[y,q]$ intersect $\gamma$ from different sides of $\gamma$. We define a new fattened foliation $\Fcal'$ by thickening $[x,p]$ and $[y,q]$ to pulleys $W_x$ and $W_y$ of width $2\epsilon_1$. Denote by $z_x$ and $z_y$ the singular points on the boundaries of the pulleys $W_x$ and $W_y$, respectively.
    By twisting $\Fcal$ along $\gamma$, we can arrange that the new leaf $L'_x$ in $\Fcal'$ meets the boundary of $W_y$ at its $N^{th}$-intersection with $\gamma$, so that $L'_x$ ends at $z_y$. By taking $\epsilon_1$ sufficiently small, we obtain a clean long pulley; the remainder follows as in Case~(i).

    {\it (Case 2) $\deg_{\Fcal}(y)=1$.}\indent Consider the case where $\gamma$ is separating.

    The degree-3 singular point on the right-hand side in Figure \ref{fig:ImmediateFoli_deform} is replaced by a degree-1 singular point. Applying the same argument would yield an $x$-non-generic measured foliation of annular type, so a modification is needed. Instead, we twist $\mathcal{F}$ along $\gamma$ so that the leaf $L_x$ passes by $y$ and returns to the degree-3 singular point on the boundary of the new pulley (shown on the left-hand side of the right panel of Figure~\ref{fig:ImmediateFoli_deform}), thereby making $\mathcal{F}$ good wrt $\gamma$. The resulting foliation $\Gcal$ is then of non-annular type and provides a good normalized representative with a clean long pulley. See Figure~\ref{fig:CleanLongPulleys}-(B). Details are omitted.

    The case where $\gamma$ is non-separating can be treated similarly.
\end{proof}

\begin{proof}[Proof of Theorem \ref{thm:Shards of memory}]
We construct a minimal measured foliation that does not vanish under $\varphi$, from which density will also follow.

Enumerate the homotopy classes of simple closed curves such that $C(S_{g,n+1})=\{[\beta_i]\}_{i\ge1}$.
\medskip

\noindent {\bf Initial condition and the first induction step.}
We begin with an immediately $x$-non-generic $\Fcal_0$ wrt $\gamma_0$ with $\pair{\Fcal_0}{[\gamma_0]}>0$. Let $n_0=0$ and fix an arbitrarily small $\epsilon_1>0$.

We use Lemma \ref{lem:SmallChange2} to obtain a measured foliation $\Fcal_1$ such that
\begin{enumerate}
    \item $\Fcal_1$ is eventually $x$-non-generic wrt $\gamma_0$ and of non-annular type,
    \item $d(\Fcal_0,\Fcal_1)<\epsilon_1$,
    \item $\varphi(\Fcal_0)([\gamma_0])-\varphi(\Fcal_1)([\gamma_0])<\epsilon_1$, and
    \item $\pair{\Fcal_1}{[\beta_1]}>0$.
\end{enumerate}
In addition, we have the following:
\begin{enumerate}[resume]
    \item Let $n_1:=n(\Fcal_1,\gamma_0)>0$. Denote by $[\gamma_k]$ the $k^{th}$ iterate of canonical simple pulls of $[\gamma_0]$ for $0\le k\le n_1$.
    \item Then $\Fcal_1$ is immediately $x$-non-generic wrt $\gamma_{n_1}$.
\end{enumerate}

\noindent {\bf Induction process.}
For a sequence $(\epsilon_i)$ with $\epsilon_{i+1}<\epsilon_i/2$, by using Lemma \ref{lem:SmallChange2}, we iteratively obtain $\Fcal_i$ satisfying
\begin{enumerate}
    \item $\Fcal_i$ is eventually $x$-non-generic wrt $\gamma_{n_{i-1}}$ and of non-annular type,
    \item $d(\Fcal_{i-1},\Fcal_{i})<\epsilon_i$,
    \item $\varphi(\Fcal_{i-1})([\gamma_0])-\varphi(\Fcal_i)([\gamma_0])<\epsilon_i$, and
    \item $\pair{\Fcal_i}{[\beta_i]}>0$ and $|\pair{\Fcal_{i-1}}{[\beta_j]}-\pair{\Fcal_{i}}{[\beta_j]}|<\epsilon_i\cdot\pair{\Fcal_j}{[\beta_j]}$ for every $0\le j< i$. To obtain this, we choose $\Fcal_i$ close to $\Fcal_{i-1}$ and use the continuity of the intersection numbers.
\end{enumerate}
In addition, we have the following:
\begin{enumerate}[resume]
    \item Let $n_i:=n(\Fcal_i,\gamma_0)$. There exists a sequence $([\gamma_k])_{0\le k \le n_i}$ such that $[\gamma_k]$ is the $k^{th}$ iterate of canonical simple pulls of $[\gamma_0]$ with respect to $\Fcal_j$ for each $j\le i$ and $k\le n_j$, where we use the stability of canonical simple pulls (Lemma \ref{lem:stability of finite simple pulls}). See Figure \ref{fig:inductive step diagram}.
    \item $\Fcal_i$ is immediately $x$-non-generic wrt $\gamma_{n_i}$.
\end{enumerate}
\begin{figure}
    \centering
    \begin{tikzcd}
        (\Fcal_0,\gamma_0) \arrow{r}
			& (\Fcal_1,\gamma_0)  \arrow{d}{(\Pull)^{n_1}} & (\Fcal_2,\gamma_0) \arrow{d}{(\Pull)^{n_1}}\\
		& (\Fcal_1,\gamma_{n_1})\arrow{r} & (\Fcal_2,\gamma_{n_1}) \arrow{d}{(\Pull)^{n_2-n_1}}\\
		&& (\Fcal_2,\gamma_{n_2})
    \end{tikzcd}
    \caption{$\Fcal_i$ is immediately $x$-non-generic wrt $\gamma_j$ if and only if $i=j$, i.e., the diagonal entries. Vertical arrows indicate canonical simple pulls. Horizontal arrows are deformations by Lemma \ref{lem:SmallChange2}. The sequence of canonical simple pulls $(\gamma_i)$ is cumulated due to Lemma \ref{lem:stability of finite simple pulls}.}
    \label{fig:inductive step diagram}
\end{figure}

We have $\Fcal_k \to \Fcal_\infty$ such that
\[
    \varphi(\Fcal_\infty)([\gamma_0]) >\varphi(\Fcal_0)([\gamma_0]) - \sum_k \epsilon_k > 0,
\]
and for each $i\ge 1$ we have
\begin{align*}
    \pair{\Fcal_\infty}{[\beta_i]} &\ge \pair{\Fcal_i}{[\beta_i]}-\sum_{k\ge i}|\pair{\Fcal_{k}}{[\beta_i]}-\pair{\Fcal_{k+1}}{[\beta_i]}|\\
    &>\pair{\Fcal_i}{[\beta_i]}\left(1-\sum_{k\ge i}\epsilon_k\right)\\
    &>0
\end{align*}
Therefore, $\Fcal_\infty$ is minimal with $\varphi(\Fcal_\infty)\neq 0$. Moreover, it is easy to show that the measured foliations $\Fcal_0$ that are immediately $x$-non-generic wrt some simple closed curve $\gamma_0$ are dense in $\MF(S_{g,n+1})$. As $\epsilon_1>0$ may be chosen arbitrarily small, the construction above yields a dense subset of measured foliations.
\end{proof}

\section{Horospheres in the universal curves over Teichm\"uller spaces}

In this section, we discuss the relationship between $\Phi\from \MF(S_{g,n+1})\to \MF(S_{g,n})$ with the dynamics on the projective measured foliations $\PMF(S_{g,n_+1})$ and the universal curve $\Ccal_{g,n}$ over Teichm\"uller space. We first prove Theorem \ref{theorem:Horo_Limit_Set}, which relates the condition $\Phi(\Fcal)=0$ with the horospherical limit set $\Lambda_G^{horo}$. Then we answer Question \ref{question:Dyn_Ccal} for $x$-non-generic measured foliations and projectively invariant measured foliations of point-pushing pseudo-Anosov mapping classes.
\medskip

The universal curve $\Ccal_{g,n}=\Tcal_{g,n+1}/\pi_1(S_{g,n})$ can be compared with a hyperbolic $d$-manifold $M$ with ${\rm vol}(M)=\infty$ whose limit set coincides with $\partial_\infty \Hbb^d$. More precisely, consider the point-pushing $\pi_1(S_{g,n})$-action on $\PMF(S_{g,n+1})$. Since invariant projective measured foliations of point-pushing pseudo-Anosov mapping classes are dense in $\PMF(S_{g,n+1})$, it follows that the limit set of the $\pi_1(S_{g,n})$-action is the entire $\PMF(S_{g,n+1})$. The closure of each fiber of $\pi\colon \Tcal_{g,n+1}\to \Tcal_{g,n}$ in the Thurston compactification contains $\PMF(S_{g,n+1})$.

Recall the definitions of horospheres and horoballs in Equation \eqref{eqn:horo_ball_sphere}. As in hyperbolic manifolds, horospheres can be realized as strong stable manifolds of the geodesic flow on the space $\mathcal{Q}^1\Tcal_{g,n}$ of unit-area quadratic differentials. More precisely, the set
\[
    S_{\mathcal{Q}}(\Fcal,c):=\left\{\frac{q_{\Fcal,X}}{\int_X |q_{\Fcal,X}|}\in \mathcal{Q}^1\Tcal_{g,n}\;:\; X\in \Tcal_{g,n}\,{\rm and}\,\int_X |q_{\Fcal,X}|=c\right\}
\]
is a strong stable manifold of the Teichm\"uller geodesic flow on $\mathcal{Q}^1\Tcal_{g,n}$. Since $\int_X|q_{\Fcal,X}|=\EL(\Fcal,X)$, the set $S_{\mathcal{Q}}(\Fcal,c)$ consists of quadratic differentials lying over the horosphere $S_h(\Fcal,c)\subset \Tcal_{g,n}$.

\subsection{Horospherical Limit Sets: Proof of Theorem~\ref{theorem:Horo_Limit_Set}}\label{subsec:Horo_Limit_Set}
Recall the definition of horospherical limit sets, defined by Equation \eqref{eqn:Inf_EL_Group_Action}. We prove Theorem~\ref{theorem:Horo_Limit_Set}, which states that
\[
    \Lambda^{\rm hor}_G
    =
    \left\{ \Proj(\Fcal): \Phi(\Fcal)=0 \right\}.
\]

\begin{proof}[Proof of Theorem~\ref{theorem:Horo_Limit_Set}]
    Let $G=\Push(\pi_1(S_{g,n}))\le \MCG^+(S_{g,n+1})$. Since
    \[
        \inf_{X'\in\pi^{-1}(\pi(X))} \EL(\Fcal;X')\le \inf_{g\in G} \EL(\Fcal;g\cdot X),
    \]
    by Theorem \ref{thm:InequalityEL}  and Equation \eqref{eqn:Inf_EL_Group_Action}, we have $\Lambda^{\rm hor}_G\subset  \left\{ \Proj(\Fcal): \Phi(\Fcal)=0 \right\}$.

    Suppose $\Phi(\Fcal)=0$. By Theorem \ref{theorem:MF} and Proposition \ref{prop:LimitMFfromIntReduction}, there exists a sequence $([\psi_l])$ in $\Push(\pi_1(S_{g,n}))$ such that $[\psi_l]\cdot \Fcal\to 0$, which is equivalent to $\Proj(\Fcal)\in \Lambda_G^{\rm horo}$.

    The rest of the statements then follow from Theorem~\ref{theorem:MF}.
\end{proof}

\subsection{Quasi-minimality of geodesics and density of horospheres}\label{subsec:QuasiMini_Horosph}
Recall that the puncture-forgetting map $\pi\colon \Tcal_{g,n+1}\to \Tcal_{g,n}$ factors through the projections $\pi^1\colon \Tcal_{g,n+1}\to \Ccal_{g,n}$ and $\pi^2\colon \Ccal_{g,n}\to \Tcal_{g,n}$.

The $x$-non-generic measured foliations and the projectively invariant measured foliations of point-pushing pseudo-Anosov mapping classes represent two opposite extremes of Question~\ref{question:Dyn_Ccal}.

\begin{thm}\label{thm:non-generic_quasimin_horo}
    Suppose $\Fcal \in \MF(S_{g,n+1})\setminus \{0\}$ is $x$-non-generic. Let $(X_t)$ be a Teichm\"uller geodesic in $\Tcal_{g,n+1}$ generated by $\Fcal$. Then the following hold.
    \begin{enumerate}
        \item $\Phi(\Fcal)=0$ if and only if $\Fcal$ is a weighted pre-peripheral simple closed curve.
    \end{enumerate}
    Moreover, if $\Phi(\Fcal)\neq 0$, then
    \begin{enumerate}[resume]
        \item The horosphere $p^1(S_h(\Fcal,r))$ in $\Ccal_{g,n}$ is not dense.
        \item The projection $(p^1(X_t))$ in $\Ccal_{g,n}$ is quasi-minimizing.
    \end{enumerate}
\end{thm}
\begin{proof}
    (1) follows from Proposition \ref{prop:PForget NonGeneric}.

    (2) and (3) for an $x$-non-generic $\Fcal$ that is not a weighted pre-peripheral simple closed curve follow from Theorems \ref{thm:Im_horospheres_} and \ref{thm:LowBddTeichDist} and the implication (2)$\Rightarrow$(3) in Question \ref{question:Dyn_Ccal} (which is straightforward, as noted there).
\end{proof}

We expect that (2) and (3) of Theorem \ref{thm:non-generic_quasimin_horo} also hold when $\Fcal$ is a weighted pre-peripheral simple closed curve. This would provide an example for Question~\ref{question:Dyn_Ccal} in which (1) holds, whereas (3) and (4) do not.

Fix $X\in \Tcal_{g,n+1}$. 
We have
\[
    \mathcal{Q}^1\Tcal_{g,n+1}\cong (\MF(S_{g,n+1})\setminus\{0\})\times \PMF(S_{g,n+1}),
\]
where $\MCG^+(S_{g,n+1})$ acts diagonally.

There is a bijection between $\MF(S_{g,n+1})\setminus \{0\}$ and the horospheres in $\Tcal_{g,n+1}$ given by
\[
    \Fcal\mapsto S_{\mathcal{Q}}(\Fcal,1).
\]
Hence, $\overline{\pi_1(S_{g,n})\cdot \Fcal}=\MF(S_{g,n+1})$ if and only if $p_1(S_h(\Fcal,1))$ is dense in $\mathcal{Q}^1\Ccal_{g,n}$, where we denote by $p_1$ the projection $\mathcal{Q}^1\Tcal_{g,n+1}\to \mathcal{Q}^1\Ccal_{g,n}$ as well.

\begin{thm}
    Suppose $[\psi]\in \Push(\pi_1(S_{g,n},x))$ is a pseudo-Anosov mapping class such that $\Fcal$ is its projectively invariant measured foliation. Let $(X_t)$ be a Teichm\"uller geodesic in $\Tcal_{g,n+1}$ generated by $\Fcal$. Then the following hold.
    \begin{enumerate}
        \item $\Phi(\Fcal)=0$.
        \item The projection $p^1(X_t)$ in $\Ccal_{g,n}$ is not quasi-minimizing.
        \item (Choi--Kim \cite{ChoiKim_InvMsr}) The horosphere $p^1(S_\mathcal{Q}(\Fcal,r))$ is dense in $\mathcal{Q}^1\Ccal_{g,n}$.
    \end{enumerate}
\end{thm}
\begin{proof}
    (1) follows from Theorem~\ref{theorem:MF}.

    (2) There exists an axis $A_\psi$ in $\Tcal_{g,n+1}$ preserved by $[\psi]$ so that $A_\psi$ projects onto a closed geodesic in $\Ccal_{g,n}$. Then $p^1(X_t)$ approaches the closed geodesic, and thus it is not quasi-minimizing.

    (3) Since $S_\mathcal{Q}(r\cdot\Fcal,1)=S_\mathcal{Q}(\Fcal,r^2)$, we may assume that $r=1$. By the discussion above, it suffices to show $\overline{\pi_1(S_{g,n})\cdot\Fcal}=\MF(S_{g,n+1})$. In \cite[Theorem 10.1]{ChoiKim_InvMsr}, Choi--Kim showed the following: if $G$ is a non-elementary subgroup of the mapping class group $\MCG(S)$, and $\Proj(\Fcal)$ lies in the conical limit set of $G$, then $\overline{G\cdot \Fcal}=\MF(S)$. Note that for a pseudo-Anosov mapping class $[\psi]$, its invariant projective measured foliation $\Proj(\Fcal)$ lies in the conical limit set. In \cite[Theorem 10.1]{ChoiKim_InvMsr}, they assume $G$ is convex-cocompact. In general, one can find a non-elementary convex-cocompact subgroup of $G$ containing some power of $[\psi]$ \cite[Theorem 1.4]{FarbLee_ConvexCocompact}. 
\end{proof}

\bibliography{Punctureforget}
\bibliographystyle{alpha}
\end{document}

%% file: AsympRays.pdf_tex
\begingroup%
  \makeatletter%
  \providecommand\color[2][]{%
    \errmessage{(Inkscape) Color is used for the text in Inkscape, but the package 'color.sty' is not loaded}%
    \renewcommand\color[2][]{}%
  }%
  \providecommand\transparent[1]{%
    \errmessage{(Inkscape) Transparency is used (non-zero) for the text in Inkscape, but the package 'transparent.sty' is not loaded}%
    \renewcommand\transparent[1]{}%
  }%
  \providecommand\rotatebox[2]{#2}%
  \newcommand*\fsize{\dimexpr\f@size pt\relax}%
  \newcommand*\lineheight[1]{\fontsize{\fsize}{#1\fsize}\selectfont}%
  \ifx\svgwidth\undefined%
    \setlength{\unitlength}{211.41784908bp}%
    \ifx\svgscale\undefined%
      \relax%
    \else%
      \setlength{\unitlength}{\unitlength * \real{\svgscale}}%
    \fi%
  \else%
    \setlength{\unitlength}{\svgwidth}%
  \fi%
  \global\let\svgwidth\undefined%
  \global\let\svgscale\undefined%
  \makeatother%
  \begin{picture}(1,0.66604704)%
    \lineheight{1}%
    \setlength\tabcolsep{0pt}%
    \put(0,0){\includegraphics[width=\unitlength,page=1]{AsympRays.pdf}}%
    \put(0.28674806,0.63127714){\color[rgb]{0,0,0}\makebox(0,0)[lt]{\lineheight{1.25}\smash{\begin{tabular}[t]{l}$\Hcal_n$\end{tabular}}}}%
    \put(0.34174462,0.00951492){\color[rgb]{0,0,0}\makebox(0,0)[lt]{\lineheight{1.25}\smash{\begin{tabular}[t]{l}$\Vcal_n$\end{tabular}}}}%
    \put(0.66535297,0.61844332){\color[rgb]{0,0,0}\makebox(0,0)[lt]{\lineheight{1.25}\smash{\begin{tabular}[t]{l}$y_n=p_1(n)$\end{tabular}}}}%
    \put(0.68935834,0.04732886){\color[rgb]{0,0,0}\makebox(0,0)[lt]{\lineheight{1.25}\smash{\begin{tabular}[t]{l}$z_n=p_2(n)$\end{tabular}}}}%
    \put(0.2114516,0.28532446){\color[rgb]{0,0,0}\makebox(0,0)[lt]{\lineheight{1.25}\smash{\begin{tabular}[t]{l}$w_n$\end{tabular}}}}%
    \put(-0.00121066,0.31842126){\color[rgb]{0,0,0}\makebox(0,0)[lt]{\lineheight{1.25}\smash{\begin{tabular}[t]{l}$x$\end{tabular}}}}%
    \put(0.11401344,0.47340083){\color[rgb]{0,0,0}\makebox(0,0)[lt]{\lineheight{1.25}\smash{\begin{tabular}[t]{l}$p_1(0)$\end{tabular}}}}%
    \put(0.1120445,0.15811681){\color[rgb]{0,0,0}\makebox(0,0)[lt]{\lineheight{1.25}\smash{\begin{tabular}[t]{l}$p_2(0)$\end{tabular}}}}%
  \end{picture}%
\endgroup%

%% file: SimplePull.pdf_tex
\begingroup%
  \makeatletter%
  \providecommand\color[2][]{%
    \errmessage{(Inkscape) Color is used for the text in Inkscape, but the package 'color.sty' is not loaded}%
    \renewcommand\color[2][]{}%
  }%
  \providecommand\transparent[1]{%
    \errmessage{(Inkscape) Transparency is used (non-zero) for the text in Inkscape, but the package 'transparent.sty' is not loaded}%
    \renewcommand\transparent[1]{}%
  }%
  \providecommand\rotatebox[2]{#2}%
  \newcommand*\fsize{\dimexpr\f@size pt\relax}%
  \newcommand*\lineheight[1]{\fontsize{\fsize}{#1\fsize}\selectfont}%
  \ifx\svgwidth\undefined%
    \setlength{\unitlength}{367.03159603bp}%
    \ifx\svgscale\undefined%
      \relax%
    \else%
      \setlength{\unitlength}{\unitlength * \real{\svgscale}}%
    \fi%
  \else%
    \setlength{\unitlength}{\svgwidth}%
  \fi%
  \global\let\svgwidth\undefined%
  \global\let\svgscale\undefined%
  \makeatother%
  \begin{picture}(1,0.54736225)%
    \lineheight{1}%
    \setlength\tabcolsep{0pt}%
    \put(0,0){\includegraphics[width=\unitlength,page=1]{SimplePull.pdf}}%
    \put(0.49566438,0.16589223){\color[rgb]{0,0,0}\makebox(0,0)[lt]{\lineheight{1.25}\smash{\begin{tabular}[t]{l}$\delta$\end{tabular}}}}%
    \put(0.12391662,0.2720996){\color[rgb]{0,0,0}\makebox(0,0)[lt]{\lineheight{1.25}\smash{\begin{tabular}[t]{l}$x$\end{tabular}}}}%
    \put(0.68383281,0.0062579){\color[rgb]{0,0,0}\makebox(0,0)[lt]{\lineheight{1.25}\smash{\begin{tabular}[t]{l}$\gamma$\end{tabular}}}}%
    \put(0.1148067,0.13725328){\color[rgb]{0,0,0}\makebox(0,0)[lt]{\lineheight{1.25}\smash{\begin{tabular}[t]{l}$\partial_1 U$\end{tabular}}}}%
    \put(0.83434604,0.21229852){\color[rgb]{0,0,0}\makebox(0,0)[lt]{\lineheight{1.25}\smash{\begin{tabular}[t]{l}$\partial_2 U$\end{tabular}}}}%
  \end{picture}%
\endgroup%

%% file: DependenceRep.pdf_tex
\begingroup%
  \makeatletter%
  \providecommand\color[2][]{%
    \errmessage{(Inkscape) Color is used for the text in Inkscape, but the package 'color.sty' is not loaded}%
    \renewcommand\color[2][]{}%
  }%
  \providecommand\transparent[1]{%
    \errmessage{(Inkscape) Transparency is used (non-zero) for the text in Inkscape, but the package 'transparent.sty' is not loaded}%
    \renewcommand\transparent[1]{}%
  }%
  \providecommand\rotatebox[2]{#2}%
  \newcommand*\fsize{\dimexpr\f@size pt\relax}%
  \newcommand*\lineheight[1]{\fontsize{\fsize}{#1\fsize}\selectfont}%
  \ifx\svgwidth\undefined%
    \setlength{\unitlength}{470.22872805bp}%
    \ifx\svgscale\undefined%
      \relax%
    \else%
      \setlength{\unitlength}{\unitlength * \real{\svgscale}}%
    \fi%
  \else%
    \setlength{\unitlength}{\svgwidth}%
  \fi%
  \global\let\svgwidth\undefined%
  \global\let\svgscale\undefined%
  \makeatother%
  \begin{picture}(1,0.32082047)%
    \lineheight{1}%
    \setlength\tabcolsep{0pt}%
    \put(0,0){\includegraphics[width=\unitlength,page=1]{DependenceRep.pdf}}%
    \put(0.11971882,0.00397655){\color[rgb]{0,0,0}\makebox(0,0)[lt]{\lineheight{1.25}\smash{\begin{tabular}[t]{l}$\Gamma_1$\end{tabular}}}}%
    \put(0.48141244,0.00397655){\color[rgb]{0,0,0}\makebox(0,0)[lt]{\lineheight{1.25}\smash{\begin{tabular}[t]{l}$\Gamma$\end{tabular}}}}%
    \put(0.84310613,0.00397655){\color[rgb]{0,0,0}\makebox(0,0)[lt]{\lineheight{1.25}\smash{\begin{tabular}[t]{l}$\Gamma_2$\end{tabular}}}}%
    \put(0.10582243,0.11501612){\color[rgb]{0,0,0}\makebox(0,0)[lt]{\lineheight{1.25}\smash{\begin{tabular}[t]{l}$x$\end{tabular}}}}%
  \end{picture}%
\endgroup%

%% file: Cube.pdf_tex
\begingroup%
  \makeatletter%
  \providecommand\color[2][]{%
    \errmessage{(Inkscape) Color is used for the text in Inkscape, but the package 'color.sty' is not loaded}%
    \renewcommand\color[2][]{}%
  }%
  \providecommand\transparent[1]{%
    \errmessage{(Inkscape) Transparency is used (non-zero) for the text in Inkscape, but the package 'transparent.sty' is not loaded}%
    \renewcommand\transparent[1]{}%
  }%
  \providecommand\rotatebox[2]{#2}%
  \newcommand*\fsize{\dimexpr\f@size pt\relax}%
  \newcommand*\lineheight[1]{\fontsize{\fsize}{#1\fsize}\selectfont}%
  \ifx\svgwidth\undefined%
    \setlength{\unitlength}{202.42270642bp}%
    \ifx\svgscale\undefined%
      \relax%
    \else%
      \setlength{\unitlength}{\unitlength * \real{\svgscale}}%
    \fi%
  \else%
    \setlength{\unitlength}{\svgwidth}%
  \fi%
  \global\let\svgwidth\undefined%
  \global\let\svgscale\undefined%
  \makeatother%
  \begin{picture}(1,0.72390095)%
    \lineheight{1}%
    \setlength\tabcolsep{0pt}%
    \put(0,0){\includegraphics[width=\unitlength,page=1]{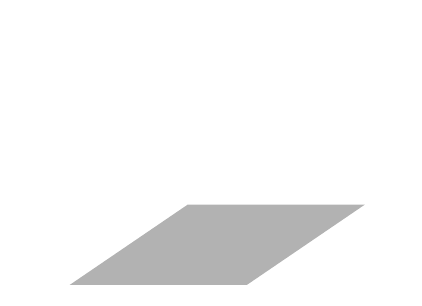}}%
    \put(0.43321046,0.13361931){\color[rgb]{0,0,0}\makebox(0,0)[lt]{\lineheight{1.25}\smash{\begin{tabular}[t]{l}$C_{m-1}$\end{tabular}}}}%
    \put(0,0){\includegraphics[width=\unitlength,page=2]{Cube.pdf}}%
    \put(0.03866329,0.0106107){\color[rgb]{0,0,0}\makebox(0,0)[lt]{\lineheight{1.25}\smash{\begin{tabular}[t]{l}$[\Gamma]$\end{tabular}}}}%
    \put(-0.07536689,0.49463956){\color[rgb]{0,0,0}\makebox(0,0)[lt]{\lineheight{1.25}\smash{\begin{tabular}[t]{l}$[\Gamma_{\{m\}}]$\end{tabular}}}}%
    \put(0.87853703,0.20011274){\color[rgb]{0,0,0}\makebox(0,0)[lt]{\lineheight{1.25}\smash{\begin{tabular}[t]{l}$[\Gamma_J]$\end{tabular}}}}%
    \put(0.87617481,0.68691239){\color[rgb]{0,0,0}\makebox(0,0)[lt]{\lineheight{1.25}\smash{\begin{tabular}[t]{l}$[\Gamma_{J\cup\{m\}}]$\end{tabular}}}}%
  \end{picture}%
\endgroup%

%% file: AnnulusArc.pdf_tex
\begingroup%
  \makeatletter%
  \providecommand\color[2][]{%
    \errmessage{(Inkscape) Color is used for the text in Inkscape, but the package 'color.sty' is not loaded}%
    \renewcommand\color[2][]{}%
  }%
  \providecommand\transparent[1]{%
    \errmessage{(Inkscape) Transparency is used (non-zero) for the text in Inkscape, but the package 'transparent.sty' is not loaded}%
    \renewcommand\transparent[1]{}%
  }%
  \providecommand\rotatebox[2]{#2}%
  \newcommand*\fsize{\dimexpr\f@size pt\relax}%
  \newcommand*\lineheight[1]{\fontsize{\fsize}{#1\fsize}\selectfont}%
  \ifx\svgwidth\undefined%
    \setlength{\unitlength}{436.7388099bp}%
    \ifx\svgscale\undefined%
      \relax%
    \else%
      \setlength{\unitlength}{\unitlength * \real{\svgscale}}%
    \fi%
  \else%
    \setlength{\unitlength}{\svgwidth}%
  \fi%
  \global\let\svgwidth\undefined%
  \global\let\svgscale\undefined%
  \makeatother%
  \begin{picture}(1,0.39554659)%
    \lineheight{1}%
    \setlength\tabcolsep{0pt}%
    \put(0,0){\includegraphics[width=\unitlength,page=1]{AnnulusArc.pdf}}%
    \put(0.02050176,0.07028966){\color[rgb]{0,0,0}\makebox(0,0)[lt]{\lineheight{1.25}\smash{\begin{tabular}[t]{l}$\gamma_1$\end{tabular}}}}%
    \put(0.13497117,0.03973822){\color[rgb]{0,0,0}\makebox(0,0)[lt]{\lineheight{1.25}\smash{\begin{tabular}[t]{l}$\delta$\end{tabular}}}}%
    \put(0.21200203,0.34843363){\color[rgb]{0,0,0}\makebox(0,0)[lt]{\lineheight{1.25}\smash{\begin{tabular}[t]{l}$\alpha$\end{tabular}}}}%
    \put(0.27820706,0.07032032){\color[rgb]{0,0,0}\makebox(0,0)[lt]{\lineheight{1.25}\smash{\begin{tabular}[t]{l}$\gamma_1'$\end{tabular}}}}%
    \put(0.33122179,0.28135083){\color[rgb]{0,0,0}\makebox(0,0)[lt]{\lineheight{1.25}\smash{\begin{tabular}[t]{l}$\gamma_i$\end{tabular}}}}%
    \put(0,0){\includegraphics[width=\unitlength,page=2]{AnnulusArc.pdf}}%
    \put(0.79317329,0.19950545){\color[rgb]{0,0,0}\makebox(0,0)[lt]{\lineheight{1.25}\smash{\begin{tabular}[t]{l}$\tilde{x}$\end{tabular}}}}%
    \put(0,0){\includegraphics[width=\unitlength,page=3]{AnnulusArc.pdf}}%
    \put(0.71783939,0.32549089){\color[rgb]{0,0,0}\makebox(0,0)[lt]{\lineheight{1.25}\smash{\begin{tabular}[t]{l}$w$\end{tabular}}}}%
  \end{picture}%
\endgroup%

%% file: Annulus.pdf_tex
\begingroup%
  \makeatletter%
  \providecommand\color[2][]{%
    \errmessage{(Inkscape) Color is used for the text in Inkscape, but the package 'color.sty' is not loaded}%
    \renewcommand\color[2][]{}%
  }%
  \providecommand\transparent[1]{%
    \errmessage{(Inkscape) Transparency is used (non-zero) for the text in Inkscape, but the package 'transparent.sty' is not loaded}%
    \renewcommand\transparent[1]{}%
  }%
  \providecommand\rotatebox[2]{#2}%
  \newcommand*\fsize{\dimexpr\f@size pt\relax}%
  \newcommand*\lineheight[1]{\fontsize{\fsize}{#1\fsize}\selectfont}%
  \ifx\svgwidth\undefined%
    \setlength{\unitlength}{440.36996436bp}%
    \ifx\svgscale\undefined%
      \relax%
    \else%
      \setlength{\unitlength}{\unitlength * \real{\svgscale}}%
    \fi%
  \else%
    \setlength{\unitlength}{\svgwidth}%
  \fi%
  \global\let\svgwidth\undefined%
  \global\let\svgscale\undefined%
  \makeatother%
  \begin{picture}(1,0.2398979)%
    \lineheight{1}%
    \setlength\tabcolsep{0pt}%
    \put(0,0){\includegraphics[width=\unitlength,page=1]{Annulus.pdf}}%
    \put(0.01607061,0.00454304){\color[rgb]{0,0,0}\makebox(0,0)[lt]{\lineheight{1.25}\smash{\begin{tabular}[t]{l}$\partial_1$\end{tabular}}}}%
    \put(0.37469368,0.00454304){\color[rgb]{0,0,0}\makebox(0,0)[lt]{\lineheight{1.25}\smash{\begin{tabular}[t]{l}$\partial_2$\end{tabular}}}}%
    \put(0.70961544,0.00454304){\color[rgb]{0,0,0}\makebox(0,0)[lt]{\lineheight{1.25}\smash{\begin{tabular}[t]{l}$\gamma$\end{tabular}}}}%
    \put(0,0){\includegraphics[width=\unitlength,page=2]{Annulus.pdf}}%
  \end{picture}%
\endgroup%

%% file: PathtoArc.pdf_tex
\begingroup%
  \makeatletter%
  \providecommand\color[2][]{%
    \errmessage{(Inkscape) Color is used for the text in Inkscape, but the package 'color.sty' is not loaded}%
    \renewcommand\color[2][]{}%
  }%
  \providecommand\transparent[1]{%
    \errmessage{(Inkscape) Transparency is used (non-zero) for the text in Inkscape, but the package 'transparent.sty' is not loaded}%
    \renewcommand\transparent[1]{}%
  }%
  \providecommand\rotatebox[2]{#2}%
  \newcommand*\fsize{\dimexpr\f@size pt\relax}%
  \newcommand*\lineheight[1]{\fontsize{\fsize}{#1\fsize}\selectfont}%
  \ifx\svgwidth\undefined%
    \setlength{\unitlength}{716.24527098bp}%
    \ifx\svgscale\undefined%
      \relax%
    \else%
      \setlength{\unitlength}{\unitlength * \real{\svgscale}}%
    \fi%
  \else%
    \setlength{\unitlength}{\svgwidth}%
  \fi%
  \global\let\svgwidth\undefined%
  \global\let\svgscale\undefined%
  \makeatother%
  \begin{picture}(1,0.25852027)%
    \lineheight{1}%
    \setlength\tabcolsep{0pt}%
    \put(0,0){\includegraphics[width=\unitlength,page=1]{PathtoArc.pdf}}%
  \end{picture}%
\endgroup%

%% file: PreEssForgetDomain.pdf_tex
\begingroup%
  \makeatletter%
  \providecommand\color[2][]{%
    \errmessage{(Inkscape) Color is used for the text in Inkscape, but the package 'color.sty' is not loaded}%
    \renewcommand\color[2][]{}%
  }%
  \providecommand\transparent[1]{%
    \errmessage{(Inkscape) Transparency is used (non-zero) for the text in Inkscape, but the package 'transparent.sty' is not loaded}%
    \renewcommand\transparent[1]{}%
  }%
  \providecommand\rotatebox[2]{#2}%
  \newcommand*\fsize{\dimexpr\f@size pt\relax}%
  \newcommand*\lineheight[1]{\fontsize{\fsize}{#1\fsize}\selectfont}%
  \ifx\svgwidth\undefined%
    \setlength{\unitlength}{127.50097381bp}%
    \ifx\svgscale\undefined%
      \relax%
    \else%
      \setlength{\unitlength}{\unitlength * \real{\svgscale}}%
    \fi%
  \else%
    \setlength{\unitlength}{\svgwidth}%
  \fi%
  \global\let\svgwidth\undefined%
  \global\let\svgscale\undefined%
  \makeatother%
  \begin{picture}(1,0.73457892)%
    \lineheight{1}%
    \setlength\tabcolsep{0pt}%
    \put(0,0){\includegraphics[width=\unitlength,page=1]{PreEssForgetDomain.pdf}}%
  \end{picture}%
\endgroup%

%% file: PrePeriForgetDomain.pdf_tex
\begingroup%
  \makeatletter%
  \providecommand\color[2][]{%
    \errmessage{(Inkscape) Color is used for the text in Inkscape, but the package 'color.sty' is not loaded}%
    \renewcommand\color[2][]{}%
  }%
  \providecommand\transparent[1]{%
    \errmessage{(Inkscape) Transparency is used (non-zero) for the text in Inkscape, but the package 'transparent.sty' is not loaded}%
    \renewcommand\transparent[1]{}%
  }%
  \providecommand\rotatebox[2]{#2}%
  \newcommand*\fsize{\dimexpr\f@size pt\relax}%
  \newcommand*\lineheight[1]{\fontsize{\fsize}{#1\fsize}\selectfont}%
  \ifx\svgwidth\undefined%
    \setlength{\unitlength}{127.75920116bp}%
    \ifx\svgscale\undefined%
      \relax%
    \else%
      \setlength{\unitlength}{\unitlength * \real{\svgscale}}%
    \fi%
  \else%
    \setlength{\unitlength}{\svgwidth}%
  \fi%
  \global\let\svgwidth\undefined%
  \global\let\svgscale\undefined%
  \makeatother%
  \begin{picture}(1,0.62150198)%
    \lineheight{1}%
    \setlength\tabcolsep{0pt}%
    \put(0,0){\includegraphics[width=\unitlength,page=1]{PrePeriForgetDomain.pdf}}%
  \end{picture}%
\endgroup%

%% file: quasitransverse.pdf_tex
\begingroup%
  \makeatletter%
  \providecommand\color[2][]{%
    \errmessage{(Inkscape) Color is used for the text in Inkscape, but the package 'color.sty' is not loaded}%
    \renewcommand\color[2][]{}%
  }%
  \providecommand\transparent[1]{%
    \errmessage{(Inkscape) Transparency is used (non-zero) for the text in Inkscape, but the package 'transparent.sty' is not loaded}%
    \renewcommand\transparent[1]{}%
  }%
  \providecommand\rotatebox[2]{#2}%
  \newcommand*\fsize{\dimexpr\f@size pt\relax}%
  \newcommand*\lineheight[1]{\fontsize{\fsize}{#1\fsize}\selectfont}%
  \ifx\svgwidth\undefined%
    \setlength{\unitlength}{488.61945205bp}%
    \ifx\svgscale\undefined%
      \relax%
    \else%
      \setlength{\unitlength}{\unitlength * \real{\svgscale}}%
    \fi%
  \else%
    \setlength{\unitlength}{\svgwidth}%
  \fi%
  \global\let\svgwidth\undefined%
  \global\let\svgscale\undefined%
  \makeatother%
  \begin{picture}(1,0.3398903)%
    \lineheight{1}%
    \setlength\tabcolsep{0pt}%
    \put(0,0){\includegraphics[width=\unitlength,page=1]{quasitransverse.pdf}}%
    \put(-0.0008139,0.02909055){\color[rgb]{0,0,0}\makebox(0,0)[lt]{\lineheight{1.25}\smash{\begin{tabular}[t]{l}$\delta$\end{tabular}}}}%
    \put(0.3104428,0.31496632){\color[rgb]{0,0,0}\makebox(0,0)[lt]{\lineheight{1.25}\smash{\begin{tabular}[t]{l}$\gamma$\end{tabular}}}}%
    \put(0,0){\includegraphics[width=\unitlength,page=2]{quasitransverse.pdf}}%
    \put(0.5882247,0.02909055){\color[rgb]{0,0,0}\makebox(0,0)[lt]{\lineheight{1.25}\smash{\begin{tabular}[t]{l}$\delta$\end{tabular}}}}%
    \put(0.89948113,0.31496627){\color[rgb]{0,0,0}\makebox(0,0)[lt]{\lineheight{1.25}\smash{\begin{tabular}[t]{l}$\gamma$\end{tabular}}}}%
    \put(0,0){\includegraphics[width=\unitlength,page=3]{quasitransverse.pdf}}%
  \end{picture}%
\endgroup%

%% file: LocalmoveCut.pdf_tex
\begingroup%
  \makeatletter%
  \providecommand\color[2][]{%
    \errmessage{(Inkscape) Color is used for the text in Inkscape, but the package 'color.sty' is not loaded}%
    \renewcommand\color[2][]{}%
  }%
  \providecommand\transparent[1]{%
    \errmessage{(Inkscape) Transparency is used (non-zero) for the text in Inkscape, but the package 'transparent.sty' is not loaded}%
    \renewcommand\transparent[1]{}%
  }%
  \providecommand\rotatebox[2]{#2}%
  \newcommand*\fsize{\dimexpr\f@size pt\relax}%
  \newcommand*\lineheight[1]{\fontsize{\fsize}{#1\fsize}\selectfont}%
  \ifx\svgwidth\undefined%
    \setlength{\unitlength}{250.89551851bp}%
    \ifx\svgscale\undefined%
      \relax%
    \else%
      \setlength{\unitlength}{\unitlength * \real{\svgscale}}%
    \fi%
  \else%
    \setlength{\unitlength}{\svgwidth}%
  \fi%
  \global\let\svgwidth\undefined%
  \global\let\svgscale\undefined%
  \makeatother%
  \begin{picture}(1,0.29789161)%
    \lineheight{1}%
    \setlength\tabcolsep{0pt}%
    \put(0,0){\includegraphics[width=\unitlength,page=1]{LocalmoveCut.pdf}}%
    \put(0.21927074,0.19024951){\color[rgb]{0,0,0}\makebox(0,0)[lt]{\lineheight{1.25}\smash{\begin{tabular}[t]{l}$I$\end{tabular}}}}%
    \put(0.58417678,0.19024951){\color[rgb]{0,0,0}\makebox(0,0)[lt]{\lineheight{1.25}\smash{\begin{tabular}[t]{l}$II$\end{tabular}}}}%
    \put(-0.00557875,0.08574195){\color[rgb]{0,0,0}\makebox(0,0)[lt]{\lineheight{1.25}\smash{\begin{tabular}[t]{l}$III$\end{tabular}}}}%
    \put(0.82997401,0.08574195){\color[rgb]{0,0,0}\makebox(0,0)[lt]{\lineheight{1.25}\smash{\begin{tabular}[t]{l}$IV$\end{tabular}}}}%
    \put(0,0){\includegraphics[width=\unitlength,page=2]{LocalmoveCut.pdf}}%
  \end{picture}%
\endgroup%

%% file: LocalmoveGlue.pdf_tex
\begingroup%
  \makeatletter%
  \providecommand\color[2][]{%
    \errmessage{(Inkscape) Color is used for the text in Inkscape, but the package 'color.sty' is not loaded}%
    \renewcommand\color[2][]{}%
  }%
  \providecommand\transparent[1]{%
    \errmessage{(Inkscape) Transparency is used (non-zero) for the text in Inkscape, but the package 'transparent.sty' is not loaded}%
    \renewcommand\transparent[1]{}%
  }%
  \providecommand\rotatebox[2]{#2}%
  \newcommand*\fsize{\dimexpr\f@size pt\relax}%
  \newcommand*\lineheight[1]{\fontsize{\fsize}{#1\fsize}\selectfont}%
  \ifx\svgwidth\undefined%
    \setlength{\unitlength}{130.29643502bp}%
    \ifx\svgscale\undefined%
      \relax%
    \else%
      \setlength{\unitlength}{\unitlength * \real{\svgscale}}%
    \fi%
  \else%
    \setlength{\unitlength}{\svgwidth}%
  \fi%
  \global\let\svgwidth\undefined%
  \global\let\svgscale\undefined%
  \makeatother%
  \begin{picture}(1,1.13281579)%
    \lineheight{1}%
    \setlength\tabcolsep{0pt}%
    \put(0,0){\includegraphics[width=\unitlength,page=1]{LocalmoveGlue.pdf}}%
  \end{picture}%
\endgroup%

%% file: LocalmoveGlue1.pdf_tex
\begingroup%
  \makeatletter%
  \providecommand\color[2][]{%
    \errmessage{(Inkscape) Color is used for the text in Inkscape, but the package 'color.sty' is not loaded}%
    \renewcommand\color[2][]{}%
  }%
  \providecommand\transparent[1]{%
    \errmessage{(Inkscape) Transparency is used (non-zero) for the text in Inkscape, but the package 'transparent.sty' is not loaded}%
    \renewcommand\transparent[1]{}%
  }%
  \providecommand\rotatebox[2]{#2}%
  \newcommand*\fsize{\dimexpr\f@size pt\relax}%
  \newcommand*\lineheight[1]{\fontsize{\fsize}{#1\fsize}\selectfont}%
  \ifx\svgwidth\undefined%
    \setlength{\unitlength}{157.40373987bp}%
    \ifx\svgscale\undefined%
      \relax%
    \else%
      \setlength{\unitlength}{\unitlength * \real{\svgscale}}%
    \fi%
  \else%
    \setlength{\unitlength}{\svgwidth}%
  \fi%
  \global\let\svgwidth\undefined%
  \global\let\svgscale\undefined%
  \makeatother%
  \begin{picture}(1,0.94830369)%
    \lineheight{1}%
    \setlength\tabcolsep{0pt}%
    \put(0,0){\includegraphics[width=\unitlength,page=1]{LocalmoveGlue1.pdf}}%
    \put(-0.0061179,0.02408794){\color[rgb]{0,0,0}\makebox(0,0)[lt]{\lineheight{1.25}\smash{\begin{tabular}[t]{l}$\gamma_1$\end{tabular}}}}%
    \put(-0.0061179,0.13866171){\color[rgb]{0,0,0}\makebox(0,0)[lt]{\lineheight{1.25}\smash{\begin{tabular}[t]{l}$\gamma_1'$\end{tabular}}}}%
    \put(0,0){\includegraphics[width=\unitlength,page=2]{LocalmoveGlue1.pdf}}%
  \end{picture}%
\endgroup%

%% file: Forget_nondeg_1.pdf_tex
\begingroup%
  \makeatletter%
  \providecommand\color[2][]{%
    \errmessage{(Inkscape) Color is used for the text in Inkscape, but the package 'color.sty' is not loaded}%
    \renewcommand\color[2][]{}%
  }%
  \providecommand\transparent[1]{%
    \errmessage{(Inkscape) Transparency is used (non-zero) for the text in Inkscape, but the package 'transparent.sty' is not loaded}%
    \renewcommand\transparent[1]{}%
  }%
  \providecommand\rotatebox[2]{#2}%
  \newcommand*\fsize{\dimexpr\f@size pt\relax}%
  \newcommand*\lineheight[1]{\fontsize{\fsize}{#1\fsize}\selectfont}%
  \ifx\svgwidth\undefined%
    \setlength{\unitlength}{188.9523003bp}%
    \ifx\svgscale\undefined%
      \relax%
    \else%
      \setlength{\unitlength}{\unitlength * \real{\svgscale}}%
    \fi%
  \else%
    \setlength{\unitlength}{\svgwidth}%
  \fi%
  \global\let\svgwidth\undefined%
  \global\let\svgscale\undefined%
  \makeatother%
  \begin{picture}(1,1.82051741)%
    \lineheight{1}%
    \setlength\tabcolsep{0pt}%
    \put(0,0){\includegraphics[width=\unitlength,page=1]{Forget_nondeg_1.pdf}}%
  \end{picture}%
\endgroup%

%% file: Forget_nondeg_2.pdf_tex
\begingroup%
  \makeatletter%
  \providecommand\color[2][]{%
    \errmessage{(Inkscape) Color is used for the text in Inkscape, but the package 'color.sty' is not loaded}%
    \renewcommand\color[2][]{}%
  }%
  \providecommand\transparent[1]{%
    \errmessage{(Inkscape) Transparency is used (non-zero) for the text in Inkscape, but the package 'transparent.sty' is not loaded}%
    \renewcommand\transparent[1]{}%
  }%
  \providecommand\rotatebox[2]{#2}%
  \newcommand*\fsize{\dimexpr\f@size pt\relax}%
  \newcommand*\lineheight[1]{\fontsize{\fsize}{#1\fsize}\selectfont}%
  \ifx\svgwidth\undefined%
    \setlength{\unitlength}{192.65743946bp}%
    \ifx\svgscale\undefined%
      \relax%
    \else%
      \setlength{\unitlength}{\unitlength * \real{\svgscale}}%
    \fi%
  \else%
    \setlength{\unitlength}{\svgwidth}%
  \fi%
  \global\let\svgwidth\undefined%
  \global\let\svgscale\undefined%
  \makeatother%
  \begin{picture}(1,1.79305276)%
    \lineheight{1}%
    \setlength\tabcolsep{0pt}%
    \put(0,0){\includegraphics[width=\unitlength,page=1]{Forget_nondeg_2.pdf}}%
  \end{picture}%
\endgroup%

%% file: Forget_nondeg_3.pdf_tex
\begingroup%
  \makeatletter%
  \providecommand\color[2][]{%
    \errmessage{(Inkscape) Color is used for the text in Inkscape, but the package 'color.sty' is not loaded}%
    \renewcommand\color[2][]{}%
  }%
  \providecommand\transparent[1]{%
    \errmessage{(Inkscape) Transparency is used (non-zero) for the text in Inkscape, but the package 'transparent.sty' is not loaded}%
    \renewcommand\transparent[1]{}%
  }%
  \providecommand\rotatebox[2]{#2}%
  \newcommand*\fsize{\dimexpr\f@size pt\relax}%
  \newcommand*\lineheight[1]{\fontsize{\fsize}{#1\fsize}\selectfont}%
  \ifx\svgwidth\undefined%
    \setlength{\unitlength}{182.18141331bp}%
    \ifx\svgscale\undefined%
      \relax%
    \else%
      \setlength{\unitlength}{\unitlength * \real{\svgscale}}%
    \fi%
  \else%
    \setlength{\unitlength}{\svgwidth}%
  \fi%
  \global\let\svgwidth\undefined%
  \global\let\svgscale\undefined%
  \makeatother%
  \begin{picture}(1,1.89410061)%
    \lineheight{1}%
    \setlength\tabcolsep{0pt}%
    \put(0,0){\includegraphics[width=\unitlength,page=1]{Forget_nondeg_3.pdf}}%
  \end{picture}%
\endgroup%

%% file: CleanLongPulley1.pdf_tex
\begingroup%
  \makeatletter%
  \providecommand\color[2][]{%
    \errmessage{(Inkscape) Color is used for the text in Inkscape, but the package 'color.sty' is not loaded}%
    \renewcommand\color[2][]{}%
  }%
  \providecommand\transparent[1]{%
    \errmessage{(Inkscape) Transparency is used (non-zero) for the text in Inkscape, but the package 'transparent.sty' is not loaded}%
    \renewcommand\transparent[1]{}%
  }%
  \providecommand\rotatebox[2]{#2}%
  \newcommand*\fsize{\dimexpr\f@size pt\relax}%
  \newcommand*\lineheight[1]{\fontsize{\fsize}{#1\fsize}\selectfont}%
  \ifx\svgwidth\undefined%
    \setlength{\unitlength}{284.85937644bp}%
    \ifx\svgscale\undefined%
      \relax%
    \else%
      \setlength{\unitlength}{\unitlength * \real{\svgscale}}%
    \fi%
  \else%
    \setlength{\unitlength}{\svgwidth}%
  \fi%
  \global\let\svgwidth\undefined%
  \global\let\svgscale\undefined%
  \makeatother%
  \begin{picture}(1,0.30910401)%
    \lineheight{1}%
    \setlength\tabcolsep{0pt}%
    \put(0,0){\includegraphics[width=\unitlength,page=1]{CleanLongPulley1.pdf}}%
    \put(0.35997521,0.29332368){\color[rgb]{0,0,0}\makebox(0,0)[lt]{\lineheight{1.25}\smash{\begin{tabular}[t]{l}$\gamma$\end{tabular}}}}%
  \end{picture}%
\endgroup%

%% file: CleanLongPulley2.pdf_tex
\begingroup%
  \makeatletter%
  \providecommand\color[2][]{%
    \errmessage{(Inkscape) Color is used for the text in Inkscape, but the package 'color.sty' is not loaded}%
    \renewcommand\color[2][]{}%
  }%
  \providecommand\transparent[1]{%
    \errmessage{(Inkscape) Transparency is used (non-zero) for the text in Inkscape, but the package 'transparent.sty' is not loaded}%
    \renewcommand\transparent[1]{}%
  }%
  \providecommand\rotatebox[2]{#2}%
  \newcommand*\fsize{\dimexpr\f@size pt\relax}%
  \newcommand*\lineheight[1]{\fontsize{\fsize}{#1\fsize}\selectfont}%
  \ifx\svgwidth\undefined%
    \setlength{\unitlength}{284.85937644bp}%
    \ifx\svgscale\undefined%
      \relax%
    \else%
      \setlength{\unitlength}{\unitlength * \real{\svgscale}}%
    \fi%
  \else%
    \setlength{\unitlength}{\svgwidth}%
  \fi%
  \global\let\svgwidth\undefined%
  \global\let\svgscale\undefined%
  \makeatother%
  \begin{picture}(1,0.30910401)%
    \lineheight{1}%
    \setlength\tabcolsep{0pt}%
    \put(0,0){\includegraphics[width=\unitlength,page=1]{CleanLongPulley2.pdf}}%
    \put(0.35997521,0.29332368){\color[rgb]{0,0,0}\makebox(0,0)[lt]{\lineheight{1.25}\smash{\begin{tabular}[t]{l}$\gamma$\end{tabular}}}}%
  \end{picture}%
\endgroup%

%% file: WHmove.pdf_tex
\begingroup%
  \makeatletter%
  \providecommand\color[2][]{%
    \errmessage{(Inkscape) Color is used for the text in Inkscape, but the package 'color.sty' is not loaded}%
    \renewcommand\color[2][]{}%
  }%
  \providecommand\transparent[1]{%
    \errmessage{(Inkscape) Transparency is used (non-zero) for the text in Inkscape, but the package 'transparent.sty' is not loaded}%
    \renewcommand\transparent[1]{}%
  }%
  \providecommand\rotatebox[2]{#2}%
  \newcommand*\fsize{\dimexpr\f@size pt\relax}%
  \newcommand*\lineheight[1]{\fontsize{\fsize}{#1\fsize}\selectfont}%
  \ifx\svgwidth\undefined%
    \setlength{\unitlength}{473.02012046bp}%
    \ifx\svgscale\undefined%
      \relax%
    \else%
      \setlength{\unitlength}{\unitlength * \real{\svgscale}}%
    \fi%
  \else%
    \setlength{\unitlength}{\svgwidth}%
  \fi%
  \global\let\svgwidth\undefined%
  \global\let\svgscale\undefined%
  \makeatother%
  \begin{picture}(1,0.42292003)%
    \lineheight{1}%
    \setlength\tabcolsep{0pt}%
    \put(0,0){\includegraphics[width=\unitlength,page=1]{WHmove.pdf}}%
  \end{picture}%
\endgroup%

%% file: ImmediateFoli_deform.pdf_tex
\begingroup%
  \makeatletter%
  \providecommand\color[2][]{%
    \errmessage{(Inkscape) Color is used for the text in Inkscape, but the package 'color.sty' is not loaded}%
    \renewcommand\color[2][]{}%
  }%
  \providecommand\transparent[1]{%
    \errmessage{(Inkscape) Transparency is used (non-zero) for the text in Inkscape, but the package 'transparent.sty' is not loaded}%
    \renewcommand\transparent[1]{}%
  }%
  \providecommand\rotatebox[2]{#2}%
  \newcommand*\fsize{\dimexpr\f@size pt\relax}%
  \newcommand*\lineheight[1]{\fontsize{\fsize}{#1\fsize}\selectfont}%
  \ifx\svgwidth\undefined%
    \setlength{\unitlength}{480.65414789bp}%
    \ifx\svgscale\undefined%
      \relax%
    \else%
      \setlength{\unitlength}{\unitlength * \real{\svgscale}}%
    \fi%
  \else%
    \setlength{\unitlength}{\svgwidth}%
  \fi%
  \global\let\svgwidth\undefined%
  \global\let\svgscale\undefined%
  \makeatother%
  \begin{picture}(1,0.38123089)%
    \lineheight{1}%
    \setlength\tabcolsep{0pt}%
    \put(0,0){\includegraphics[width=\unitlength,page=1]{ImmediateFoli_deform.pdf}}%
    \put(0.04208873,0.04382368){\color[rgb]{0,0,0}\makebox(0,0)[lt]{\lineheight{1.25}\smash{\begin{tabular}[t]{l}$x$\end{tabular}}}}%
    \put(0.26787764,0.12997927){\color[rgb]{0,0,0}\makebox(0,0)[lt]{\lineheight{1.25}\smash{\begin{tabular}[t]{l}$\alpha$\end{tabular}}}}%
    \put(0.14479357,0.00445716){\color[rgb]{0,0,0}\makebox(0,0)[lt]{\lineheight{1.25}\smash{\begin{tabular}[t]{l}$\gamma$\end{tabular}}}}%
    \put(0,0){\includegraphics[width=\unitlength,page=2]{ImmediateFoli_deform.pdf}}%
  \end{picture}%
\endgroup%